\documentclass[11pt,reqno,a4paper]{amsart}

\usepackage{graphicx}
\usepackage{color}
\usepackage[USenglish]{babel}
\usepackage{latexsym}
\usepackage{amsmath}
\usepackage{amsfonts}
\usepackage{amssymb}
\usepackage{comment}
\usepackage{esint}
\usepackage[all]{xy}

\newcommand{\grad}{\nabla}

\renewcommand{\to}{\rightarrow}
\newcommand{\pa}{\partial}

\newcommand{\ino}{\int_{\Omega}}

\newcommand{\fo}{\forall}

\newcommand{\ov}[1]{\overline{#1}}
\newcommand{\un}[1]{\underline{#1}}

\renewcommand{\dfrac}{\displaystyle\frac}
\newcommand{\finedim}{\hspace{\fill}$\square$}
\newcommand{\intbar}{\mathop{\int\makebox(-15.5,0){\rule[6pt]{.7em}{0.3pt}}\kern-6pt}\nolimits}

\newcommand{\ii}{\infty}

\newcommand{\eps}{\varepsilon}
\newcommand{\dt}{\delta}

\newcommand{\al}{\alpha}

\newcommand{\sg}{\sigma}

\newcommand{\om}{\Omega}

\newcommand{\graf}[1]{\left\{\begin{array}{ll}#1\end{array}\right.}

\renewcommand{\a }{\alpha }

\newcommand{\D }{\Delta }

\newcommand{\lm }{\lambda }

\newcommand{\rh }{\rho }

\renewcommand{\O }{\Omega }

\def\p{\partial}

\newcommand{\be}{\begin{equation}}
\newcommand{\ee}{\end{equation}}
\newcommand{\beq}{\begin{equation}}
\newcommand{\eeq}{\end{equation}}

\newcommand{\R}{\mathbb{R}}

\newcommand{\N}{\mathbb{N}}

\newcommand{\dis}{\displaystyle}

\newcommand{\vn}{v_n}

\newcommand{\bns}{\alpha_{n,\sigma}}

\newcommand{\bis}{\alpha_{\infty,\sigma}}

\newtheorem{theorem}{Theorem}[section]
\newtheorem{proposition}[theorem]{Proposition}
\newtheorem{definition}{Definition}[section]
\newtheorem{corollary}[theorem]{Corollary}
\newtheorem{remark}[theorem]{Remark}
\newtheorem{example}[theorem]{Example}

\newtheorem{lemma}[theorem]{Lemma}

\newcommand{\bpr}{\begin{proposition}}
\newcommand{\epr}{\end{proposition}}
\newcommand{\bex}{\begin{example}\rm}
\newcommand{\eex}{\end{example}}
\newcommand{\brm}{\begin{remark}\rm}
\newcommand{\erm}{\end{remark}}
\newcommand{\bdf}{\begin{definition}\rm}
\newcommand{\edf}{\end{definition}}
\newcommand{\bte}{\begin{theorem}}
\newcommand{\ete}{\end{theorem}}
\newcommand{\ble}{\begin{lemma}}
\newcommand{\ele}{\end{lemma}}
\newcommand{\bco}{\begin{corollary}}
\newcommand{\eco}{\end{corollary}}
\newcommand{\mycomment}[1]{}

\numberwithin{equation}{section}

\begin{document}
\title[Singular Mean Field equations]
{Onsager's Mean field theory of vortex flows with singular sources: blow up and concentration without quantization.}

\author{D. Bartolucci, P. Cosentino, L. Wu}

\address{Daniele Bartolucci, Department of Mathematics, University of Rome {\it ``Tor Vergata"} \\  Via della ricerca scientifica n.1, 00133 Roma, Italy. }
\email{bartoluc@mat.uniroma2.it}

\address{Paolo Cosentino, Department of Mathematics, University of Rome {\it ``Tor Vergata"} \\  Via della ricerca scientifica n.1, 00133 Roma, Italy. }
\email{cosentino@mat.uniroma2.it}

\address{Lina Wu, School of Mathematics and Statistics, Beijing Jiaotong University, Beijing, 100044, China}
\email{lnwu@bjtu.edu.cn}

\begin{abstract}

Motivated by the Onsager statistical mechanics description of turbulent Euler flows with point singularities,
we make a first step in the generalization of the mean field theory in [Caglioti, Lions, Marchioro, Pulvirenti; Comm. Math. Phys. (1995)]. On one side we prove the equivalence of statistical ensembles, on the other side we are bound to the
analysis of a new blow up phenomenon,
which we call ``blow up and concentration without quantization", where the mass associated with the concentration is allowed to take values in a
full interval of real numbers. This singular behavior may be regarded as lying between the classical blow up-concentration-quantization and the blow up without concentration phenomenon first proposed in [Lin, Tarantello; C.R. Math. Acad. Sci. Paris (2016)]. A careful analysis is needed to generalize known pointwise estimates in this non standard context, resulting in a complete
description of the allowed asymptotic profiles.
\end{abstract}



\maketitle
\section{Introduction}
In a pioneering paper (\cite{On}) L. Onsager derived a statistical mechanics description of large vortex structures observed in two
dimensional turbulent flows. Among other things he concluded that the maximum entropy (thermodynamic equilibrium) states could be realized by
highly concentrated configurations,
{where} like vortices attract each other. These states are negative temperature states,
meaning that, above a certain energy $E_0\in (0,+\ii)$, the equilibrium entropy $S(E)$ is decreasing as a function of the energy $E$.
Those physical arguments has been later turned into rigorous proofs (\cite{clmp1}, \cite{clmp2}, \cite{K}, \cite{KL}, \cite{Ne}). We refer to \cite{ESp}, \cite{ESr}
and references therein for a complete discussion about the Onsager theory and a more detailed account of the impact of those ideas and to
\cite{KWa}, \cite{Ne} for the Onsager model with random vorticity.

\medskip

Following the Onsager ideas and assuming the vorticity to be positive, in the framework of mean field theory and in a certain sharp range of temperatures
one can prove (\cite{clmp1}, \cite{clmp2}) equivalence of statistical ensembles
where the thermodynamic equilibrium is attained by entropy maximizing configurations which
solve the so called mean field equation (which is \eqref{mf:lm.intro} below with $\sg=0$ and $\lm<8\pi$). However, when defined on $\mathbb{S}^2$,
as far as $\lm<16\pi$ solutions of the mean field equation are just the trivial solutions (see \cite{Lin0}, \cite{Lin1} and \cite{GM1}). Therefore, with the aim of describing vorticity fields of geophysical interest \cite{Marcus},
it seems natural to introduce in the Onsager mean field model a fixed, macroscopic co-rotating or counter-rotating vortex.
However,
with the unique exception of a short comment in \cite{ck},
it seems that this case has not been discussed in literature. Indeed, all the works at hand so far about point singularities in mean field equations
has been concerned with the so called ``conical singularities" see \cite{bclt,BdMM,BGJM,BJL,bl1,BMal,BM2,bt,bt2,CLin4,CLin5,dwz,KuLi,pt,tar-sd,Tar3,wz1,wz2,wuz,z}
and Remark \ref{rem:conicsing} below. It is our aim here to make a first step in the generalization of the mean field theory in \cite{clmp2} to cover  this singular case.

\medskip

Let $\O\subset\R^2$ be a bounded piecewise $C^2$ domain (see definition \ref{piecewisec2}), $\mathcal{P}(\O)$ be the space of vorticity densities,
\begin{equation*}
 \mathcal{P}(\O)=\left\{\rho\in {L}^1(\O)|\,\rho\geq0\,\,\text{a.e. in}\,\,\O,\int_{\O}\rho\log \rho<+\infty, \int_{\O}\rho=1  \right\},
\end{equation*}
and set
$$
G[\rho](x)=\int_{\O}G_\O(x,y)\rho(y)\,dy,
$$
where $G_\O(x,y)$ is the Green function on $\O$ with Dirichlet boundary conditions, that is,
\begin{equation*}
 G_\O(x,y)=-\tfrac{1}{2\pi}\ln|x-y|+R_\O(x,y),
\end{equation*}
where $R_\O$ is the regular part. Also, let $0\in \O$ and assume the existence of a fixed vortex whose strength is
$\sigma\in\R$ and whose center coincides with $0$. For a fixed $\rho\in\mathcal{P}(\O)$ we define the entropy,
\begin{equation*}
 \mathcal{S}(\rho)=-\int_{\O}\rho(x)\log(\rho(x))\,dx,
\end{equation*}
and the energy of the density $\rho$,
\begin{equation*}
 \mathcal{E}(\rho)=\tfrac{1}{2}\int_{\O}\rho(x)G[\rho](x)\,dx,
\end{equation*}
and of the fixed vortex,
$$
\mathcal{E}_\sg(\rho)=\sigma\int_{\O}\rho \,G_\O(x,0)\,dx.
$$
We define the (scaled) Free Energy functional,
\begin{align}
 \nonumber
 \mathcal{F}_{\lambda}(\rho)&= -\mathcal{S}(\rho)-\lambda(\mathcal{E}(\rho)-\mathcal{E}_\sg(\rho)) \\
 \nonumber
 &= \int_{\O}\rho\log(\rho)\,dx-\tfrac{\lambda}{2}\int_{\O}\rho\, G[\rho]\,dx-
 \tfrac{\sigma\lambda}{2\pi}\int_{\O}\rho\log|x|\,dx+\sigma\lambda\int_{\O}\rho \,R_\O(x,0)\,dx.
\end{align}

The Canonical Variational Principle (CVP) with a singular source, is defined as follows,
\begin{align}\label{cvp:intro}
 f(\lambda)=\begin{cases}
             \underset{\rho\in\mathcal{P}(\O)}\inf \mathcal{F_\lambda(\rho)}, &\text{if}\,\,\lambda\geq0, \\
             \underset{\rho\in\mathcal{P}(\O)}\sup \mathcal{F_\lambda(\rho)}, &\text{if}\,\,\lambda<0,
            \end{cases}
\end{align}
where $\mathcal{F}_\lambda$ is the (scaled) free energy, while the Microcanonical Variational Principle (MVP) with a singular source takes the form,
\begin{align}\label{mvp:intro}
 S(E)=\underset{\rho\in\mathcal{P}_E(\O)}\sup \mathcal{\mathcal{S}(\rho)},
\end{align}
\begin{equation*}
 \mathcal{P}_E(\O)=\left\{\rho\in \mathcal{P}(\O)|\, \mathcal{E}(\rho)-\mathcal{E}_\sg(\rho)=E\right\}.
\end{equation*}

\begin{remark}\label{signsg} We denote with $\lm=-\frac{1}{\kappa T_{\rm stat}}$ the negative of the inverse statistical Temperature  $T_{\rm stat}$, where $\kappa$ is the Boltzmann constant and consider only the case $\lm\geq 0$.
The convention about the ``$-$" sign for the $\mathcal{E}_\sg(\rho)$ term is just to keep the notations close enough to the one about
conical singularities, see Remark \ref{rem:conicsing}.  The drawback is that a vortex with total vorticity
$\sg<0$ is co-rotating with the positive density $\rho$, while it is counter-rotating if $\sg>0$.
\end{remark}

We assume without loss of generality that $\om$ has unit area,
$$
|\om|=1,
$$
and consider only the energies larger than that of the uniform distribution $E_0=E_0(\om)$ (see \eqref{e0sec4}), since, from known results about the ``regular case" $\sg=0$, this range should correspond to the physically more interesting negative temperatures states,
see \cite{clmp2}, \cite{bart-5} and references quoted therein. Here and in the rest of this work we set,
\begin{align}
\label{lambdasigma}
\lambda_{\sigma}=\begin{cases}
            \frac{8\pi}{1+2|\sigma|}, &\text{if}\,\,\sigma<0, \\
            8\pi, &\text{if}\,\,\sigma\geq 0,
            \end{cases}
\end{align}
and
\begin{equation}\label{Hlm:intro}
 H_{\lm}(x)=e^{-\sigma\lm G_\O(x,0)}=|x|^{\frac{\lambda}{2\pi}\sigma}e^{-\sigma\lm R_\O(x,0)}.
\end{equation}

Our first result is a generalization to the singular case of well known facts established in \cite{clmp1},\cite{clmp2}. The domain $\om$ is assumed to be piecewise $C^2$ according to Definition \ref{piecewisec2} below.

\begin{theorem}\label{thm:cvpintro} For any $\lm\in [0,\lm_\sg)$ there exists one and only one solution $\rho_{\lm}$ of  \eqref{cvp:intro}
 and $\psi_{\lm}=G[\rho_\lm]$ is a solution of the Singular Mean Field Equation,
 \begin{align}
\begin{cases} \label{mf:lm.intro}
 -\D \psi_{\lambda}=\rho_{\lambda}=\dfrac{H_{\lm}e^{\lambda \psi_{\lambda}}}{ \int\limits_\O  H_{\lm}e^{\lambda \psi_{\lambda}}}\,\,\,\,\,\,&\text{in}\,\,\O,\\
 \psi_{\lambda}=0 &\text{in}\,\,\p\O.
\end{cases}
\end{align}
\end{theorem}

The proof uses known arguments about non convex optimization (\cite{clmp1}, \cite{csw}) and a singular Moser-Trudinger inequality (\cite{as}). The same
threshold is obtained for the (CVP) on compact surfaces, whose details will be discussed elsewhere (\cite{BCYZZ}). Remark that
$H_{\lm}$ is an $L^1(\om)$ function only if the following minimal necessary assumption is satisfied,
\begin{equation}\label{lsg}
\mbox{if }\quad \sg<0\quad \mbox{ then }\quad \lm<\frac{4\pi}{|\sg|},
\end{equation}
which is granted in Theorem \ref{thm:cvpintro} by the fact that $\lm<\frac{8\pi}{1+2|\sigma|}<\frac{4\pi}{|\sg|}$ as far as $\sg<0$.

\begin{remark}\label{rem:conicsing} Theorem \ref{thm:cvpintro} is sharp in the sense that if $\sg<0$ and $\om$ is a disk centered at the origin
(which is the location of the fixed vortex), then the unique solutions of \eqref{mf:lm.intro} blow up as $\lm\to(\lm_\sg)^-$
and there is no solution for $\lm\geq \lm_\sg$. The claim about blow up of solutions follows from explicit evaluations about
radial solutions and the uniqueness of solutions  (\cite{BJL},\cite{BCJL}), see section \ref{sec5}. The claim about
non existence follows from a straightforward adaptation of a well known argument based on the Pohozaev identity (\cite{clmp2}).
As far as $\sg>0$, the threshold $\lm_\sg$ is $8\pi$ and is sharp as well, although the problem is more subtle,
see \cite{bl1}, \cite{bl2}. Remark that in all the quoted results (\cite{BJL,bl1,bl2}) the ``singular" term in \eqref{Hlm:intro}
takes the form $|x|^{2\alpha}$, which is the classical expression associated to conical singularities of
order $\alpha>-1$. As far as the results in \cite{BCJL,BJL,bl1,bl2} are concerned, the overall effect of this difference
is just that of a scaling of parameters.
We will see in section \ref{sec2} that the situation is more involved as far as blow up issues are concerned.\\
This being said, the universal convention about conical points is that the $\alpha$ in the exponent of $|x|^{2\al}$ is the order of the singularity, whence
positive singularities are characterized by the fact that $\alpha>0$.
This is the origin of our (physically rather wired) convention about the sign of the energy
$\mathcal{E}_\sg$, see Remark \ref{signsg}.
\end{remark}

The situation for the (MVP) is more subtle since the energy constraint, due to the presence of the fixed vortex, is not closed in the weak-$L^1(\om)$ topology.
To overcome this issue we regularize the energy term $\mathcal{E}_\sg(\rho)$ and replace the Green function with
\begin{equation}\label{def:Gn}
    G_n(x)=-\frac{1}{2\pi}\log(h_n(x)) + R_n(x),\quad x\in \overline{\om},
\end{equation}
where
\begin{equation*}
 h_n(x)=\left({\epsilon_n^2+|x|^2}\right)^{\frac{1}{2}},\quad \epsilon_n\to 0^+,
\end{equation*}
which converges to $|x|$ in $C^{t}(\O)$, for any $t\in(0,1)$, as $n\to\infty$
and $R_n$ is the unique harmonic function
such that $R_n(x)=\frac{1}{2\pi}\log(h_n(x))$ on $\pa \om$. It is easy to see that $0\leq G_n(x)\leq G_\om(x,0)+O(\eps_n^2)$.
Clearly $R_n\to R_\om(x,0)$ uniformly in
$\om$,  and $G_n\to G_\om(x,0)$ locally uniformly far  away from $x=0$, where $G_\om$ and $R_\om$ are the standard Green function and its regular part, 
see section \ref{sec4}. Thus, as far as the (MVP) is concerned,
$\mathcal{E}_\sg(\rho)$ will be replaced by the regularized energy,
$$
\mathcal{E}_{\sg,n}(\rho)=\sg\int_{\O} \rho G_n= -\tfrac{\sigma}{2\pi}\int_{\O}\rho\log(h_n)+\sigma\int_{\O}\rho \,R_n.
$$
At this point we would like to solve the variational principle
\eqref{mvp:intro} at fixed $n$ with the regularized energy constraint
$$
E=\mathcal{E}(\rho)-\mathcal{E}_{\sg,n}(\rho),
$$
and then pass to the limit as $n\to +\ii$. It is exactly at this stage that we face a first subtle problem,
since the stream function relative to the entropy maximizers (say $\rho_n$) of the regularized variational principle (see Theorem \ref{thm:mvpn} below)
satisfies a mean field equation of the form,
\begin{align}
\begin{cases} \label{mf:n.intro}
 -\D \psi_{n}=\rho_{n}=\dfrac{H_{n}e^{  \lm_n\psi_{n}}}{\int_\O H_{n}e^{  \lm_n\psi_{n}}}\,\,\,\,\,\,&\text{in}\,\,\O,\\
 \psi_{n}=0 &\text{in}\,\,\p\O,
\end{cases}
\end{align}
where,
\begin{equation}\label{Hlmn.intro}
 H_{n}(x)=e^{-\sigma\lm_n G_n(x)}=(h_n(x))^{\frac{\lm_n}{2\pi}\sigma}e^{-\sigma\lm_n R_n(x)}.
\end{equation}
In this situation we would like to rely on some sharp a priori estimates for solutions of \eqref{mf:n.intro}-\eqref{Hlmn.intro},
in the same spirit of the concentration-compactness theory for singular Liouville-type equations in \cite{bm,yy} and \cite{bt}.
It turns out, as discussed in \cite{det}, \cite{llty} and \cite{os},
that such theory is not enough to describe the asymptotic behavior of sequences of solutions
of \eqref{mf:n.intro}-\eqref{Hlmn.intro}. The point is that if a sequence of solutions $\psi_n$ blows up at the origin,
since $h_n(x)$ is converging to $|x|$ but it is not exactly $|x|$, then \un{it is not true in general} that
the density $\rho_{n}$ in \eqref{mf:n.intro} concentrates, modulo subsequences, to a linear combination of Dirac measures. This is the so called phenomenon of
``blow up without concentration", see in particular the examples constructed in \cite{det}, \cite{llty}. Nevertheless, by using known blow up arguments (\cite{B0,bt,det,llty}) one can prove that
the blow up at the origin comes with a
``minimal inverse temperature" or either a ``minimal mass" (\cite{B0},\cite{bt},\cite{det}),
that is, even if $\rho_n$ need not weakly converge to a Dirac delta, still its weak-$*$ limit is a measure that charges the origin with some
quantity uniformly bounded from below, see the (minimal mass) Lemma \ref{minimalmasslemma} below. Of course, similar conclusions could be drawn
by using the general result in \cite{os}, which however was derived by somehow different methods, see the related discussion in the introduction of \cite{os}.
This is why, in view of the new aspects of the blow up behavior which will be discussed here,
we need to derive an independent proof of these facts, see Theorem \ref{Concentration} and Remark \ref{rem:llty}. It is not surprising afterall that the minimal mass is exactly $\lm_\sg$,
showing once more that Theorem \ref{thm:cvpintro} is sharp. We will see that
these facts are enough to establish the existence of solutions of the (MVP). Indeed we have,

\begin{theorem}\label{thm:mvp} Let $\om$ be a piecewise $C^2$ domain which contains the origin. For any $E\in [E_0,+\ii)$, there exists at least one solution
$\rho$ of the {\rm (MVP)} and there exists $\lm=\lm(E)$ such that $\rho=\rho_{\lm}$ is a solution of,
\begin{equation}\label{mvpe}
\rho_{\lm}=\dfrac{H_{\lm} e^{\lm G[\rho_{\lm}]}}{\int\limits_{\om} H_{\lm} e^{\lm G[\rho_{\lm}]}}\quad \mbox{a.e. in }\om,
\end{equation}
and $\psi_{\lm}=G[\rho_{\lm}]$ is a solution of the Singular Mean Field Equation \eqref{mf:lm.intro}, where $\lm(E)<\frac{4\pi}{|\sg|}$ as far as $\sg<0$.
In particular $\rho$ is the limit of a sequence of
entropy maximizers $\rho_n:=\rho_{\lm_n}$ as defined in \eqref{mf:n.intro}.
\end{theorem}

However the proof is not trivial, essentially because we do not know much about $\lm_n=\lm_n(E)$ in \eqref{mf:n.intro}, which arises
just as a Lagrange multiplier related to the energy constraint. Actually, as far as $\sg<0$, even the fact that
$\lm(E)<\frac{4\pi}{|\sg|}$ requires a proof. The classical uniform bound from above for $\lm_n$ via the Pohozaev identity (\cite{clmp1})
fails for connected but not simply connected domains, as well known examples show (\cite{dPKM}, \cite{KPV}, \cite{ns}).
We will need a careful adaptation of an argument recently provided in \cite{BCN24} to obtain a similar estimate for the $\lm_n$
corresponding to entropy maximizers obtained in Theorem \ref{thm:mvpn}. The underlying idea is to use a different (compared to \cite{clmp2})
argument to come up with the positivity of the entropy maximizers, suitable to be used in the various situations at hand (see Proposition \ref{rho:0}).

\bigskip

At this point, as in \cite{clmp1}, \cite{clmp2}, we can prove the equivalence of the statistical ensembles,
for a particular class of domains.

\begin{definition}\label{typeI} A piecewise $C^2$ domain $\om$ containing the origin is said to be of Type I if for any $E>E_{0}(\om)$
there exists a solution of \eqref{mvpe} with $\lm(E)<\lm_\sg$, is said to be of Type II otherwise.
\end{definition}

For domains of Type I, there is a one to one correspondence between solutions of the (CVP) and of the (MVP) and the entropy $S(E)$ is a smooth and concave
function of the energy, see Theorem \ref{thm:Equiv}.
As far as $\sg>0$, modulo the variable rescaling needed to move from conical singularities
(see Remark \ref{rem:conicsing}) domains of Type I/II have been completely characterized in \cite{bl2}. In particular one can prove that,
for domains of Type I, as $E\to +\ii$ the entropy maximizers concentrate with inverse temperature $\lm(E)\to (8\pi)^{-}$
to a Dirac delta whose singular point is a maximum point of,
$$
\gamma_\om(x)-\sg G_\om(x,0)\quad \mbox{where} \quad \gamma_\om(x)=R_\om(x,x), \;x\in\om.
$$
Similar results hold for $\sg<0$ and the full characterization of domains of Type I, in the same spirit of
\cite{CCL} and \cite{bl2}, \cite{bl}, will be discussed elsewhere (\cite{BCJL}). However, still for $\sg<0$,  any disk centered at the origin is of Type I, see section \ref{sec5}.\\
Concerning the case $\sg>0$
(i.e. a {negative} counter-rotating vortex fixed at
the origin with strength $-\sg$), the states were the (positive) vorticity $\rho$ possibly concentrates
on the fixed counter-rotating vortex location indicate a sort of collision of vortices of opposite sign. For solutions of \eqref{mf:lm.intro} and for domains of Type I, general results in \cite{bt} (see Theorem \ref{massboundarycontrol0} below)
show that whenever this sort of collision occurs, then necessarily $\sg<\frac12$. In other words, in this singular version of the Onsager mean field model, a positive vortex cannot collide with the fixed negative vortex of strength $-\sg$ as far as $\sg\geq \frac12$.\\
From the point of view of applications to geophysical flows, the
Hexagon vortex structure at the north pole of Saturn (\cite{Sat}) resembles quite much profiles of planar solutions of singular Liouville equations with $\frac{\sg \lm}{4\pi}\in \mathbb{N}$
(\cite{ck,pt,TY}), and it would be interesting to have some quantitative estimate about the entropy of these non radial solutions. It would be also interesting to come up,
by using  \eqref{mf:lm.intro}, with entropy estimates for a stationary flow mimicking the Jupiter Red Spot (\cite{Marcus}) and its local circulation. We will discuss applications of this sort and refined estimates in the same spirit of \cite{bwz1,bwz2}
in a forthcoming paper (\cite{BCYZZ}).\\ 

\bigskip

Instead, we will be concerned here with the high energy limit, relative to the case $\sg>0$, of solutions $\rho_n$
(see \eqref{mf:n.intro}) of the regularized problem derived in Theorem \ref{thm:mvpn}. The motivation has to do with a peculiar aspect of the
theory which we explain in details in section \ref{sec6}, while for the time being we describe the needed blow up analysis of solutions of \eqref{mf:n.intro}. First of all, since we
will analyze the limit $E_n=\mathcal{E}(\rho_n)-\mathcal{E}_{\sg,n}(\rho_n)\to +\ii$ for $\sg>0$, we have to deal with the competition between the energy of the density and that of the vortex. Therefore we are bound to make one step beyond the minimal mass Lemma.
By using the (minimal mass) Lemma \ref{minimalmasslemma} together with an argument first pursued in \cite{bt} or either with the general theory in \cite{os}, we have a natural assumption
which guaranties that if blow up happens we also have concentration,
see Theorem \ref{Concentration} below. It is interesting that
the needed hypothesis, which, assuming $\lm_n\to \lm_\ii$, takes the form,
\begin{equation}\label{lsgp}
\mbox{if }\quad \sg>0\quad \mbox{ then }\quad \lm_\ii\leq \frac{4\pi}{\sg},
\end{equation}
restores somehow the symmetry with respect to the necessary assumption \eqref{lsg} for $\sg<0$. Remark that,
 due to the peculiar dependence of the exponent of $H_{n}$ by $\lm_n$, if $\sg$ is small  then $\lm_n$ can be larger than $8\pi$, but
this is not in contrast with the examples of blow up without concentration discussed in \cite{llty}, where the quantity which here we denote
$\al_{n,\sg}=\frac{\lm_n}{4\pi}{\sg}$ and which, in view of \eqref{lsgp}, satisfies $\al_{n,\sg}\to \al_{\ii,\sg}\leq 1$, is kept fixed at some value larger than 1. Observe also that
the minimal mass Lemma to be used together with \eqref{lsgp}
implies that $\sg\leq \frac12$ (see Theorem \ref{massboundarycontrol}).  In this situation we face a peculiar new phenomenon which we call ``\un{blow-up and concentration without quantization}",
which stands somewhere between the classical blow up-concentration-quantization phenomenon (\cite{bm}, \cite{bt}, \cite{ls}) and the blow up without concentration
as pushed forward in \cite{det,llty,os}, see Theorem \ref{massboundarycontrol}.
In other words, we have blow up and concentration but, unlike the case of conical singularities (\cite{bt}),
the mass associated to the singular blow up point
is not anymore fixed, being allowed instead to take values in the non degenerate interval of real numbers of the form,
\begin{equation}\label{lm:intro}
8\pi\leq \lm_\ii\leq\min\left\{\frac{8\pi}{1-2\sg},\frac{4\pi}{\sg}\right\}, \quad 0\leq \sg\leq \frac12.
\end{equation}
Therefore we are bound to provide a refined description of the bubbling behavior in a situation where we lack quantization.
This is a rather delicate problem of independent analytical interest. On one side, it is well known that the moving plane method as in \cite{yy}
cannot be applied here, due to the fact that the weight $h_n$ has the ``wrong monotonicity" property. On the other side, the pointwise estimates in
\cite{bclt,bt2} crucially rely on the fact that {one already knows} the value of the mass associated to the blow up phenomenon.
As mentioned above, unfortunately this is not anymore the case in general for the situation at hand and we face different cases.
In this paper we will be concerned with the case where $\lm_n\to \lm_\ii<\frac{4\pi}{\sg}$ and $\sg\in (0,\frac12)$,
which implies $\lm_\ii<16\pi$ and, from the point of view of geometric singularities, $\al_{n,\sg}\to \al_{\ii,\sg}=\frac{\lm_\ii}{4\pi}\sg<1$.
We first obtain the value of the mass associated to the various allowed asymptotic behaviors. Toward this goal we have to discuss separately different asymptotic regimes and
use a recently derived ``sup+Cinf" estimate (see \cite{bcwyzHarnackInequality}), showing that, in the same spirit of \cite{ls}, under suitable decay conditions, there is no residual mass outside the bubble corresponding to a maximizing sequence. Remark that, if an inequality of ``sup+Cinf"-type holds, then blow up implies concentration, whence, due to the counterexamples in \cite{llty}, no inequality of this sort can be true in general as far as $\al_{\ii,\sg}>1$.
At this point, since $\lm_\ii<16\pi$,
we find that either $\lm_\ii=8\pi$, with two different asymptotic regimes (which are not radial around the singularity), or $\lm_\ii\in (8\pi,\frac{8\pi}{1-2\sg}]$ if $\sg\in(0,\frac14)$, or $\lm_\ii\in (8\pi, \frac{4\pi}{\sg})$ if $\sg\in[\frac14, \frac12)$ corresponding to the asymptotically radial (around the singularity \cite{GM1}) profiles described in \cite{llty}, see Theorem \ref{profile-new}. Remark that in the cases where $\lm_\ii=8\pi$
in fact we are exactly at the boundary of the equivalence region, see section \ref{sec6}. At this point, to solve the undeterminate form in the energy
$\mathcal{E}(\rho_n)-\mathcal{E}_{\sg,n}(\rho_n)$, we have to generalize known pointwise estimates (\cite{bclt,bt2,yy}) to cover
all the profiles allowed by this non standard bubbling behavior, which is done by a careful adaptation of the techniques in  \cite{bclt,bt2} and \cite{tar-sd}.
 The decay rate is controlled by the power $\frac{\lm_n}{4\pi}$ (see \eqref{profilev:H}, \eqref{profilev:H1}, \eqref{profilevtilde:H1-IIb})
and the analysis of the Pohozaev identity as in \cite{bclt} does not yield in general estimates strong enough to replace $\frac{\lm_n}{4\pi}$ with
$\frac{\lm_\ii}{4\pi}$. It is an interesting open problem to understand whether or not and possibly under which conditions
this replacement can be realized. The application of these estimates to the large energy limit is the subject of section \ref{sec6}.\\

If $\sg=\frac12$ the problem becomes more subtle since the energy takes the form,
$$\mathcal{E}(\rho_{n})-\mathcal{E}_{\frac12,n}(\rho_{n})=\frac12\ino \rho_nG[\rho_n]- \sg\ino\rho_n G_n(x)=\frac12 \ino \rho_n(G[\rho_n]- G_n(x))$$
and one may need higher order estimates. On the other side, if $\lm_\ii=\frac{4\pi}{\sg}$, $\sg\in [\frac14,\frac12]$, we have $\al_{\ii,\sg}=1$ and we miss the ``sup+Cinf" inequality recently derived in \cite{bcwyzHarnackInequality}. A possible solution of the problem
could be obtained via a suitable ``v(0)+Cinf v" inequality, in the same spirit of \cite{T1}, or either via mild boundary conditions as in \cite{bclt}
to exploit the original moving plane argument of \cite{bls}. Another interesting and subtle problem is that if $\lm_\ii=\frac{4\pi}{\sg}$ we
would have $\lm_\ii=16\pi$ as far as $\sg=\frac14$, which seems to yield a form of new and non standard multiple bubbling phenomenon.
We will discuss these problems in a forthcoming paper (\cite{BCYZZ2}).\\

This paper is organized as follows. In section \ref{sec2} we prove the minimal mass Lemma and its first consequences about the perturbed Liouville type equation \eqref{eqbase} below. In section \ref{sec2.1} we discuss the refined profile estimates concerning the concentration without quantization. 
In sections \ref{sec4} and 5 we prove the results about  Canonical and Microcanonical Variational Principles.
In section \ref{sec5} we illustrate an example of a domain of Type I. In section \ref{sec6} we discuss the high energy limit of solutions of \eqref{mf:n.intro} which blow up at the counter-rotating fixed vortex.

\bigskip

\section{Blow up analysis: minimal mass Lemma and first consequences}\label{sec2}
Here and in the rest of this paper we will often pass to subsequences which will not be relabelled and,
for the ease of the presentation, adopt the following notations:
\begin{equation*}
 \bns:= \tfrac{\lambda_n}{4\pi}\sigma, \quad \bis:= \tfrac{\lambda_\infty}{4\pi}\sigma,\quad \sigma\in\R,
\end{equation*}
\begin{equation}
 \label{H_n}
H_n(x)=h_{n}(x)^{2\bns}K_n(x),\quad
h_{n}(x)=\left({\epsilon_n^2+|x|^2}\right)^{\frac12}, \quad\epsilon_n\to0.
\end{equation}
Recall that,
\begin{align}
\nonumber
\lambda_{\sigma}=\begin{cases}
            \frac{8\pi}{1+2|\sigma|}, &\text{if}\,\,\sigma<0, \\
            8\pi, &\text{if}\,\,\sigma\geq 0.
            \end{cases}
\end{align}
Let us consider a sequence of solutions $v_n$ of the problem:
\begin{equation}
 \label{eqbase}
 \graf{
 -\D v_n=H_n e^{\dis v_n} \quad \text{in}\quad B_1,\\
 \hfill \\
 \int_{B_1}h_{n}^{2\bns} e^{\dis v_n}\leq C, }
\end{equation}
\begin{equation}\label{lambdan}
 \lambda_n=\int_{B_1}H_{n}e^{\dis v_n},
\end{equation}
and assume that $K_n$ satisfies,
\begin{equation}
 \label{K+}
 K_n\geq 0, \,\, K_n\in C^{0}(\overline B_1),
\end{equation}
\begin{equation}
 \label{Kconvergence1}
 K_n\to K \quad \text{uniformly in}\,\,\overline{B}_1.
\end{equation}

We assume that there exists a sequence of points $\{x_n\}\subset B_1$ such that
\begin{align}
\label{blowup:hyp}
 x_n\to 0 \,\,\,\,\text{and}&\,\,\,\, \underset{B_1}\sup\, v_n= v_n(x_n)\to +\infty, \,\,\text{as}\,\, n\to +\infty.
\end{align}
We will address some results concerning the blow up analysis of solutions of \eqref{eqbase}, \eqref{lambdan} satisfying \eqref{blowup:hyp} in the same
spirit of \cite{bt}, along the lines pursued in \cite{det} and \cite{llty}. The results we will discuss are obtained
by arguments similar to those adopted in  \cite{bt}, \cite{det} and \cite{llty} and it is likely that some of them could also be seen as consequences of
the more general theory in \cite{os}. Nevertheless, due the peculiarity of the blow up analysis we will purse and of the interplay of the
 parameters in \eqref{eqbase} we need to provide a detailed proof.  Among other things, we will see that, as expected from \cite{det}, \cite{llty}, \cite{os},
 a full generalization of the results in \cite{bt} can be recovered if $\sg<0$, while only partial results are available as far as $\sg>0$, see also Remark
 \ref{rem:llty} below.

\bigskip

\begin{lemma}[Minimal Mass Lemma]\label{minimalmasslemma}$\,$\\
Let $v_n$ be a sequence of solutions of \eqref{eqbase} and \eqref{lambdan} satisfying \eqref{blowup:hyp}, where $K_n$ and $K$ satisfy \eqref{K+}
and \eqref{Kconvergence1}. Assume furthermore that,
\begin{equation}\label{lambdantoinfty}
\lambda_n\to\lambda_\infty>0, \,\, \text{as} \,\, n\to\infty,
\end{equation}
that \eqref{lsg} is satisfied and
\begin{equation*}
 H_ne^{\dis v_n} \rightharpoonup \eta \qquad \text{weakly in the sense of measures in}\; B_1,
\end{equation*}
for some bounded Radon measure $\eta$, where we define,
$$
\beta:=\eta(\{0\}).
$$
Then,
\begin{equation}\label{Kzero}
K(0)>0
\end{equation}
and
\begin{align}
 \label{minimalmass}
\lm_\ii\geq \beta\geq\begin{cases}
            8\pi(1+\bis), &\text{if}\,\,\sigma<0, \\
            8\pi, &\text{if}\,\,\sigma\geq 0,
            \end{cases}
\end{align}
and in particular
\begin{equation}\label{mmbeta1}
\lm_\ii\geq \lm_{\sigma}.
\end{equation}

\end{lemma}
\begin{proof}
 Clearly, in view of \eqref{lambdan}, we have $\lm_\ii\geq \beta$ and, as far as we are concerned with the proof of the r.h.s. inequality in \eqref{minimalmass},
 it is enough to show that, for every $r\in(0,1)$,
\begin{align}
 \label{minimalmass2negative}
\underset{n\to\infty}\liminf \int_{B_r(0)}H_n e^{\dis v_n}\geq\begin{cases}
            8\pi(1+\bis), &\text{if}\,\,\sigma<0, \\
            8\pi, &\text{if}\,\,\sigma\geq 0.
            \end{cases}
\end{align}

Also, in view of \eqref{lambdan}, we see that \eqref{minimalmass} immediately implies \eqref{mmbeta1}. Therefore we restrict our
attention to the proof of \eqref{minimalmass2negative} where \eqref{lambdan} will not be needed anymore.\\
Let us fix $r\in(0,1)$ and, for $n$ large enough, by \eqref{blowup:hyp} assume without loss of generality that,
$$
\underset{\overline B_r}\sup\, v_n= v_n(x_n)\to +\ii \;\quad \mbox{and}\;\quad |x_n|\to 0.
$$
Here an in the rest of this proof we set
\begin{equation}\label{dtn}
     \delta_n^{2(1+\bns)}=e^{-\dis {v_n(x_n)}}\to 0,\,\, \text{as} \,\, n\to\infty,
\end{equation}
and
\[
 t_n=\max\{|x_n|,\delta_n\}\to 0,\,\, \text{as} \,\, n\to\infty.
\]

Here, we are naturally led to analyze three different cases:
\begin{itemize}
 \item Case \;(I): there exists a subsequence such that $\tfrac{\epsilon_n}{t_n}\to+\infty$, as $n\to+\infty$;
 \item Case (II): there exists a constant $C_1>0$ such that
\begin{equation}
 \label{Hypothesis:CaseII}
 \frac{\epsilon_n}{t_n}\leq C_1,\,\,\text{for all}\,\,n\in\N
\end{equation}
 and, possibly along a subsequence, $\tfrac{|x_n|}{\delta_n}\to+\infty$, as $n\to+\infty$;
 \item Case (III): there exist constants $C_1>0$ and $C_2>0$ such that
 \begin{equation}
 \label{Hypothesis:CaseIII}
 \frac{\epsilon_n}{t_n}\leq C_1,\quad \tfrac{|x_n|}{\delta_n}\leq C_2,\,\,\text{for all}\,\,n\in\N.
\end{equation}
\end{itemize}
We start with Case (I) and notice that,
\begin{equation}
 \label{epsilondelta&epsilonxn}
 \frac{\epsilon_n}{\delta_n}\to+\infty\qquad\text{and}\qquad \frac{\epsilon_n}{|x_n|}\to+\infty.
\end{equation}
Let us define
\begin{equation}\label{bn1}
B^{(n)}=B_{\frac{r}{\epsilon_n}}(0),
\end{equation}
and
\begin{equation}\label{wn:IA}
 w_n(z)=v_n(\epsilon_n z)+2(1+\bns)\log\epsilon_n\quad\text{in}\quad B^{(n)},
\end{equation}
which satisfies
\begin{align*}
\begin{cases}
-\D w_n=V_n(z)e^{\dis w_n} \quad\text{in}\quad B^{(n)},\\
V_n(z)=\left({1+|z|^2}\right)^{\bns} K_n(\epsilon_n z)\to
K(0)\left({1+|z|^2}\right)^{\bis}\;\mbox{in}\;C^t_{loc}(\R^2),\\
\int\limits_{B^{(n)}}\left({1+|z|^2}\right)^{\bns}e^{\dis w_n}\leq C,\\
w_n(\tfrac{x_n}{\epsilon_n})=2(1+\bns)\log(\tfrac{\epsilon_n}{\delta_n})\to +\infty.
\end{cases}
\end{align*}
In particular, we notice that $\tfrac{x_n}{\epsilon_n}\to 0$ and also that,
\[
 \int_{B_{R}(0)}e^{\dis w_n}\leq \tilde C,
\]
for any $R>0$ and some constant $\tilde C>0$. Thus, by an immediate application of the concentration-compactness alternative
for regular Liouville-type equations
on $B_{R}(0)$ (see \cite{bm} and \cite{ls}), we see that $K(0)>0$ and
\[
 \underset{n\to\infty}\liminf \int_{B_{R}(0)}\left({1+|z|^2}\right)^{\bns}K_n(\epsilon_n z)e^{\dis w_n}=8m\pi,
\]
where $m\in\N^+$, so that, for any $r\in (0,1)$, we have that
\[
 \underset{n\to\infty}\liminf \int_{B_r(0)}H_n e^{\dis v_n}=
 \underset{n\to\infty}\liminf \int_{B^{(n)}}\left({1+|z|^2}\right)^{\bns}K_n(\epsilon_n z)e^{\dis w_n} \geq
\]
\[
 \geq \underset{n\to\infty}\liminf \int_{B_{R}(0)}\left({1+|z|^2}\right)^{\bns}K_n(\epsilon_n z)e^{\dis w_n}=8\pi m,
\]

which is \eqref{minimalmass2negative}, thereby concluding the analysis of Case (I).
\medskip

Next we consider Case (II) and  notice that in this situation we have that $\delta_n\leq|x_n|$, for $n$ large enough, whence
\begin{equation}
 \label{epsilonxn}
 \frac{\epsilon_n}{|x_n|}\leq C,\,\,\text{for all}\,\,n\in\N,
\end{equation}
for some $C>0$ constant. Thus, besides the fact that \[
 \frac{|x_n|}{\delta_n}\to+\infty,\,\,\text{as}\,\,n\to+\infty,
\]
possibly along a subsequence,
we also have that,
\[
 \frac{\epsilon_n}{|x_n|}\to\bar \epsilon_0\geq  0,\quad \frac{x_n}{|x_n|}\to \bar z_0,\;\;|\bar z_0|=1,\quad \text{as}\,\,n\to+\infty.
\]
 In this situation we define,
\begin{equation}\label{dn1}
D^{(n)}=B_{\frac{r}{|x_{n}|}}(0),
\end{equation}
\begin{equation}\label{barwn:IB}
 \bar w_{n}(z)=v_n(|x_{n}| z)+2(1+\bns)\log(|x_{n}|)\quad\text{in}\quad D^{(n)},
\end{equation}
which satisfies
\begin{align*}
\begin{cases}
-\D \bar w_{n}=\bar V_{n}(z)e^{\dis \bar w_{n}} \quad\text{in}\quad D^{(n)},\\
\bar V_{n}(z)=\left({\frac{\epsilon_n^2}{|x_{n}|^2}+|z|^2}\right)^{\bns}K_n(|x_{n}| z)\to K(0)(\bar \epsilon_0^2+|z|^2)^{2\bis}\;\mbox{in}\;C^t_{loc}(\R^2),\\
\int\limits_{D^{(n)}}\left({\frac{\epsilon_n^2}{|x_{n}|^2}+|z|^2}\right)^{\bns} e^{\dis \bar w_n}\leq C,\\
\bar w_{n}(\tfrac{x_{n}}{|x_{n}|})=2(1+\bns)\log(\tfrac{|x_{n}|}{\delta_{n}})\to +\infty,
\end{cases}
\end{align*}
 Again, we notice that,
\[
 \int_{B_{R}(\bar z_0)}e^{\dis \bar w_n}\leq \bar C,
\]
for some $R>0$ small enough and some constant $\bar C>0$. Therefore, as an immediate application of the concentration-compactness
alternative for regular Liouville-type equations on $B_{R}(\bar z_0)$ (see \cite{bm} and \cite{ls}), we see that $K(0)>0$ and
\[
 \underset{n\to\infty}\liminf \int_{B_{R}(\bar z_0)}
 \left({\frac{\epsilon_n^2}{|x_n|^2}+|z|^2}\right)^{\bns}K_{n}(|x_n|z)e^{\dis \bar w_n}=8\pi m,
\]
where $m\in\N^+$ and $K(0)>0$. As a consequence we have that,
\[
 \underset{n\to\infty}\liminf \underset{B_r(0)}\int H_n e^{\dis v_n}= \underset{n\to\infty}\liminf \underset{D^{(n)}}
 \int \left({\frac{\epsilon_n^2}{|x_n|^2}+|z|^2}\right)^{\bns}K_{n}(|x_n|z)e^{\dis \bar w_n}
\]
\[\geq\underset{n\to\infty}\liminf \int_{B_{R}(0)} \left({\frac{\epsilon_n^2}{|x_n|^2}+|z|^2}\right)^{\bns}K_{n}(|x_n|z)
e^{\dis \bar w_n}=8\pi m.
\]
This is \eqref{minimalmass2negative}, which concludes the study of Case (II).

\medskip

At last, we examine Case (III). We notice that \eqref{Hypothesis:CaseIII} implies that,
\begin{equation}
 \label{epsilondeltan}
 \frac{\epsilon_n}{\delta_n}\leq C,\,\,\text{for all}\,\,n\in\N,
\end{equation}
for some $C>0$ constant. Indeed we have that,
\begin{align*}
 \frac{\epsilon_n}{\delta_n}=\frac{\epsilon_n}{t_n}\frac{t_n}{|x_n|}\frac{|x_n|}{\delta_n}=\begin{cases}
         \frac{\epsilon_n}{t_n},&\quad\text{if}\,\,\delta_n\geq|x_n|,\\
         \frac{\epsilon_n}{t_n}\frac{|x_n|}{\delta_n}, &\quad\text{if}\,\,\delta_n\leq|x_n|,                                                                                    \end{cases}
\end{align*}
immediately implying that, $\frac{\epsilon_n}{\delta_n}\leq\max\{{C_1},{C_1C_2}\}$, for every $n\in\N$.\\
Let us define
$$
D_n=B_{\frac{r}{\delta_n}}(0),
$$
and possibly along a subsequence, assume without loss of generality that,
\begin{equation}
 \label{x0}
 \frac{x_n}{\delta_n}\to y_0, \,\,\text{as}\,\,n\to+\infty
\end{equation}
 and
\begin{equation}
 \label{epsilon_0}
 \frac{\epsilon_n}{\delta_n}\to \epsilon_0, \,\,\text{as}\,\,n\to+\infty,
\end{equation}
for some $y_0\in\R^2$ and $\epsilon_0\geq 0$. At this point, we define,
\begin{equation}
 \label{phicaseIA}
 \widetilde w_n(y)=v_n(\delta_n y)+2(1+\bns)\log \delta_n=v_n(\delta_n y)-v_n(x_n)\quad\text{in}\,\, D_n,
\end{equation}
which satisfies,
\begin{align}\label{equationgoodcase}
\begin{cases}
 -\D \widetilde w_n=\left({\frac{\epsilon_n^2}{\delta_n^2}+\left|y\right|^2}\right)^{\bns}
 K_n(\delta_n y)e^{\dis \widetilde w_n}:=f_n \qquad\text{in}\,\,D_n, \\
\int_{D_n}\left({\frac{\epsilon_n^2}{\delta_n^2}+|y|^2}\right)^{\bns}e^{\dis \widetilde w_n}\leq C,\\
\widetilde w_n(y)\leq\widetilde w_n(\tfrac{x_n}{\delta_n})=\max\limits_{D_n}\widetilde w_n=0.
\end{cases}
\end{align}
Remark that $f_n$ is uniformly bounded in $L^{p}(B_R)$, for $R$ large enough where $p=p(|\sigma|)>1$, if $\sigma<0$, while $p=\infty$, if $\sigma\geq0$.
Moreover we see that, by the Harnack inequality (see for example Lemma 1 in \cite{bt}), for every $R\geq 1+ 2|y_0|$ there exists a constant $C_R>0$ such that
\begin{equation}
 \label{Estimateboundary}
  \sup_{\p B_{R}}|\widetilde w_n|\leq C_R.
\end{equation}
Indeed, since $f_n$ is uniformly bounded in $L^p(B_{2R})$, for some $p>1$, and
$$\sup_{\p B_{2R}}\widetilde w_n\leq \sup_{ B_{4R}}\widetilde w_n =\widetilde w_n(\tfrac{x_n}{\delta_n})=0,$$
then there exist $\tau\in(0,1)$ and $C_0>0$, which does not depend by $n$, such that
\[
 \sup_{ B_{R}} \widetilde w_n\leq \tau\inf_{B_{R}}\widetilde w_n +C_0
\]
and this implies that
\[
 \inf_{\p B_{R}}\widetilde w_n=\inf_{ B_{R}}\widetilde w_n\geq -\tau^{-1}C_0,
\]
where we used the fact $\sup\limits_{ B_{R}} \widetilde w_n=0$, for $n$ large enough, and the superharmonicity of $\widetilde w_n$.\\
Therefore, by \eqref{equationgoodcase}, \eqref{Estimateboundary} and standard elliptic estimates, we conclude that $\widetilde w_n$ is uniformly
bounded in $C^{t}_{\rm loc}(\R^2)$, for some $t\in(0,1)$ and then, possibly along a subsequence, $\widetilde w_n\to \widetilde w$
uniformly on compact sets of $\R^2$, where $\widetilde w$ satisfies
\begin{align*}
\begin{cases}
 -\D \widetilde w=(\epsilon_0^2+|y|^2)^{\bis}K(0)e^{\dis \widetilde w} \quad\text{in}\quad\R^2, \\
\int_{\R^2}(\epsilon_0^2+|y|^2)^{\bis}e^{\dis \widetilde w}\leq C.
\end{cases}
\end{align*}
Since $\widetilde w$ is bounded from above, then necessarily $K(0)\neq0$. At this point, by well known results (\cite{cl2,det}, see Lemma \ref{massstrange} below), we have that either,

\begin{align*}
\underset{\R^2}\int(\epsilon_0^2+|y|^2)^{\bis}K(0)e^{\dis \widetilde w}\geq 8\pi(1+\bis), \quad&\text{if}\quad \sigma<0,
\end{align*}
or,
\begin{align*}
\underset{\R^2}\int(\epsilon_0^2+|y|^2)^{\bis}K(0)e^{\dis \widetilde w}> \max\{8\pi,4\pi(1+\bis)\}\geq 8\pi, \quad &\text{if}\quad \sigma\geq 0.
\end{align*}
Therefore we have,
\[
 \underset{n\to\infty}\liminf \underset{B_r(0)}\int H_n e^{\dis v_n}=
 \underset{n\to\infty}\liminf \underset{D_n}\int \left({\frac{\epsilon_n^2}{\delta_n^2}+\left|y\right|^2}\right)^{\bns}
 K_n(\delta_n y)e^{\dis \widetilde w_n}\geq\underset{\R^2}\int (\epsilon_0^2+|y|^2)^{\bis}K(0)e^{\dis \widetilde w}
 \]
which implies that in this case we have either,
\begin{align*}
\underset{n\to\infty}\liminf \underset{B_r(0)}\int H_n e^{\dis v_n}\geq 8\pi(1+\bis), \quad&\text{if}\quad \sigma<0,
\end{align*}
or
\begin{align*}
\underset{n\to\infty}\liminf \underset{B_r(0)}\int H_n e^{\dis v_n}> \max\{8\pi,4\pi(1+\bis)\}\geq8\pi, \quad &\text{if}\quad \sigma\geq 0.
\end{align*}
At this point it is readily seen that \eqref{minimalmass2negative} holds in Case (III) as well, which concludes the proof.
\end{proof}

\bigskip

Our next aim is to show that, under suitable conditions, if $p=0$ is the unique blow up point for $v_n$ in $B_1$, that is,
\[
 \text{for any}\,\,r\in(0,1),\exists \, C_r>0, \,\,\text{such that}:
\]
\begin{equation}
    \label{bounded}
   \underset{\overline{B_1\backslash B_r}}{\max}\, v_n\leq C_r,
\end{equation}
\begin{equation}
    \label{explosion}
    \underset{\overline{B_r}}\max\, v_n\to\infty,
\end{equation}
then $v_n$ concentrates in the following sense: there exists a subsequence, which we will not relabel, such that, as $n\to\infty$, we have
\begin{align}
 \label{-infinity}
 &v_{n}\to-\infty,\qquad\qquad \text{as}\,\,n\to\infty,\,\text{uniformly on compact sets of}\,\,B_1\backslash\{0\},
\\
 \label{concentration0}
 &H_{n}e^{\dis v_{n}}\rightharpoonup \beta\delta_{p=0}, \quad\text{weakly in the sense of measures in}\,\, B_1,\\
 \label{mmbeta}
 & \text{where}\quad \lm_\ii\geq \beta\geq\begin{cases}
            8\pi(1+\bis), &\text{if}\,\,\sigma<0, \\
            8\pi, &\text{if}\,\,\sigma\geq 0.
            \end{cases}
\end{align}

The proof of the following result uses the same arguments as that in Theorem 4 of \cite{bt}, we provide the details just for reader's convenience. The result also follows from the general theory in \cite{os} (see Theorem 2.1 in \cite{os}).

\begin{theorem}[Concentration]\label{Concentration}\hfill\\
Let $v_n$ be a sequence of solutions of \eqref{eqbase} and \eqref{lambdan}, where $K_n$ satisfies \eqref{K+} and
\begin{equation}
 \label{Kconvergence2}
 K_n\to K \quad \text{uniformly in}\,\,\overline{B}_1\,\,\text{and in} \,\, C^1_{loc}(B_1).
\end{equation}
Assume furthermore that  \eqref{lambdantoinfty} is satisfied
and that \eqref{bounded}, \eqref{explosion} holds true. Then \eqref{Kzero} and \eqref{mmbeta1} hold true and if either
\begin{equation}\label{conditionmixed:neg}
\sigma<0 \,\,\,\,\,\,\text{and}\,\,\,\,\,\,\, \lambda_\infty<\tfrac{4\pi}{|\sigma|},
\end{equation}
 or if
\begin{equation}
 \label{conditionmixed}
 \sigma>0\,\,\,\,\,\,\text{and}\,\,\,\,\,\,\, \lambda_\infty\leq\tfrac{4\pi}{\sigma},
\end{equation}
then necessarily $\sg\leq \frac12$ and there exists a subsequence of $v_n$ such that \eqref{-infinity}, \eqref{concentration0} and \eqref{mmbeta} are
satisfied.\\
Otherwise, if for $\sg>0$ the assumption about $\lm_\ii$ in \eqref{conditionmixed} is not satisfied, then we have that $\lm_\ii\geq 8\pi$ and,
either  \eqref{-infinity}, \eqref{concentration0} and \eqref{mmbeta} are
satisfied or,
for some $0<\ov{r}<1$ and $\dt\in (0,1)$, and for any $0<r_0\leq \ov{r}$ we have that,
\begin{align}
 \label{vn:bdd}
 &v_{n}\to\xi ,\hspace{3.4cm}\mbox{in}\;C^{1,\delta}_{\rm loc}(B_{\bar r}(0)\backslash\{0\}),
\\
 \label{vn:concentration0}
 &H_{n}e^{\dis v_{n}}\rightharpoonup f_\ii(x) +\beta\delta_{p=0}, \quad
 \text{weakly in the sense of measures in}\,\, B_1 \,\,  \\
 &\hspace{4.7cm} \mbox{and in}\,\, C^{0,\delta}_{\rm loc}(B_{\bar r}(0)\backslash\{0\}),\nonumber\\
 \label{vn:bdd:mmbeta}
 & f_\ii(x)=|x|^{2\bis}K e^{\dis \xi}\in L^1(B_{\ov{r}}),\quad \xi(x)=-\frac{\beta}{2\pi}\log(|x|)+\phi(x)+\gamma(x), \\
 & \label{vn:bdd:mmbeta2}\text{where }8\pi\leq \beta<4\pi+\lm_\ii\sg,\quad \int\limits_{B_{\ov{r}}}f_\ii=\lm_\ii-\beta,\\
 &\text{for some smooth function $\gamma$ in $B_{r_0}$ and }\phi(x)=-\tfrac{1}{2\pi}\int_{B_{r_0}}\ln|x-y|f_\ii(y)\nonumber.
\end{align}
\end{theorem}

\begin{proof}
Clearly, by the (minimal mass) Lemma \ref{minimalmasslemma}, we have that $K(0)>0$ (which is \eqref{Kzero}), \eqref{mmbeta1} is satisfied and in particular, putting
$\beta=\eta(\{0\})$, we have that \eqref{mmbeta} holds true as well. Therefore, as far as $\sg>0$, from \eqref{mmbeta} and \eqref{conditionmixed}
we have that,
$$
8\pi\leq \lm_\ii\leq \tfrac{4\pi}{\sigma}
$$
i.e. $\sg\leq \frac12$.\\
Next, assuming either \eqref{conditionmixed:neg} or \eqref{conditionmixed}, we will prove the following:\\
{\bf CLAIM}: for every $r\in(0,1)$, possibly along a subsequence, we have
\begin{equation}\label{vconc}
 \underset{|x|=r}\min\,v_n\to-\infty,\qquad\text{as}\,\,n\to\infty.
\end{equation}
{\it Proof of Claim.} We argue by contradiction and assume that there exists $\bar r\in (0,1)$ and a constant $C>0$ such that
$$
 \underset{|x|=\bar r}\min\,v_n\geq C,\qquad\text{for every}\,\,n\in\N.
$$
By the maximum principle and \eqref{bounded}, $v_n$ is uniformly bounded in $L^\infty_{\rm loc}(B_{\bar r}(0)\backslash\{0\})$, whence, possibly along a subsequence, we may assume that,
\begin{equation*}
 v_n\to\xi\,\,\text{pointwise a.e. and in $C^{1,\delta}_{\rm loc}(B_{\bar r}(0)\backslash\{0\})$, for some $\delta\in(0,1)$,}
\end{equation*}
\begin{equation*}
 \big(\epsilon_n^2+|x|^2\big)^{\bns}K_n e^{\dis v_n}\to |x|^{2\bis}K e^{\dis \xi}, \,\,\text{in $C^{0}_{\rm loc}(B_{\bar r}(0)\backslash\{0\})$}
\end{equation*}
By Fatou's lemma, $|x|^{2\bis}K e^{\dis \xi}\in L^1(B_{\bar r})$, whence we have that,
\begin{equation*}
 \big(\epsilon_n^2+|x|^2\big)^{\bns}K_n e^{\dis v_n}\rightharpoonup \eta=|x|^{2\bis}K e^{\dis \xi}+\beta\delta_{p=0},
\end{equation*}
weakly in the sense of measure in $B_{\bar r}(0)$, where, again by  Lemma \ref{minimalmasslemma}, $\beta$ satisfies \eqref{mmbeta}, where
obviously $\lm_\ii\geq \beta$. Moreover, since
$K(0)>0$, we conclude that,
\begin{equation*}
 \big(\epsilon_n^2+|x|^2\big)^{\bns}e^{\dis v_n}\rightharpoonup |x|^{2\bis}e^{\dis \xi}+\tfrac{\beta}{K(0)}\delta_{p=0},
\end{equation*}
weakly in the sense of measure in $B_{\bar r}(0)$. Let us fix $0<r_0\leq \bar r$ and let us denote $B_0=B_{r_0}(0)$ and
\begin{equation*}
 \varphi_n(x)=\big(\epsilon_n^2+|x|^2\big)^{\bns}K_n e^{\dis v_n}\quad \text{and}\quad \varphi(x)=|x|^{2\bis}K e^{\dis \xi}.
\end{equation*}

By using Green's representation formula for $v_n$ in $B_0$ and passing to the limit as $n\to\infty$, we deduce that
\begin{equation*}
 \xi(x)=\phi(x)-\tfrac{\beta}{2\pi}\ln|x|+\gamma(x),
\end{equation*}
with
\begin{equation*}
 \phi(x)=-\tfrac{1}{2\pi}\int_{B_0}\ln|x-y|\varphi(y)\,dy
\end{equation*}
and
\begin{equation*}
 \gamma(x)=-\tfrac{1}{2\pi}\int_{\p B_0}\xi(y)\tfrac{(x-y)\cdot \nu}{|x-y|}\,dy +\tfrac{1}{2\pi}\int_{\p B_0}\tfrac{\p\xi(y)}{\p\nu}\ln|x-y|\,dy.
\end{equation*}
We notice that $\gamma\in C^1(B_r(0))$, for every $r\in(0,r_0)$, $\phi\in L^{\infty}(B_0)$ and
\begin{equation*}
 \varphi(x)=|x|^{\frac{\sigma\lambda_\infty-\beta}{2\pi}}Ke^{\dis \phi(x)+\gamma(x)}.
\end{equation*}
Since $\varphi\in L^1(B_0)$ we also deduce that,
\begin{equation}
 \label{conditionintegrability}
 \beta<4\pi(1+\bis)=4\pi+\lm_\ii \sg.
\end{equation}
At this point, if $\sigma<0$ then \eqref{conditionintegrability} immediately contradicts \eqref{mmbeta}. Otherwise \eqref{conditionmixed} holds true and,
in the same time, by \eqref{mmbeta} we have $\beta\geq 8\pi$. Thus, in this particular case, by using \eqref{conditionintegrability}, we  find that
$\lm_\ii> \frac{4\pi}{\sg}$, which contradicts \eqref{conditionmixed}. This fact concludes the proof of the claim, whence \eqref{vconc} is satisfied.\finedim

\bigskip
\bigskip

At this point, for every compact set $U\Subset B_1\backslash\{0\}$, there exist two radii $0<r_1<r_2<1$ such that $U\subset B_{r_2}\backslash B_{r_1}$ and
then, by the Harnack inequality (see for example Lemma 1 in \cite{bt}), we see that,
\[
 \underset{U}\sup\, v_n\leq\underset{B_{r_2}\backslash B_{r_1}}\sup\,v_n\leq C_0
 \underset{B_{r_2}\backslash B_{r_1}}\inf\,v_n+ C_1=C_0\underset{\p(B_{r_2}\backslash B_{r_1})}\inf\,v_n + C_1,
\]
where the last equality follows from the fact that $v_n$ is superharmonic. Therefore \eqref{-infinity} holds true as well,
implying that $\eta$ is supported at zero. In particular, along a subsequence we have,
\[
H_{n}e^{\dis v_n}\rightharpoonup \eta=\beta\delta_{p=0},
\]
weakly in the sense of measure in $B_1$, where obviously \eqref{minimalmass} implies that $\beta$ satisfies \eqref{mmbeta}, as claimed.\\

Therefore we are left with the case $\sg>0$ where the condition about $\lm_\ii$ in \eqref{conditionmixed} is not satisfied, i.e. $\lm_\ii>\frac{4\pi}{\sg}$.
In this situation there are only two possibilities left: either the claim is true and then again as above we have that \eqref{-infinity} holds true as well,
implying that $\eta$ is supported at zero and we are back to the previous case, or the claim is false, that is we can assume that there exists $\bar r\in (0,1)$ and a constant $C>0$ such that
$$
 \underset{|x|=\bar r}\min\,v_n\geq C,\qquad\text{for every}\,\,n\in\N.
$$
Thus we can just follow the argument in the proof of the claim, to deduce that \eqref{vn:bdd}, \eqref{vn:concentration0} and \eqref{vn:bdd:mmbeta} hold
true as well. Concerning \eqref{vn:bdd:mmbeta2} just remark that, in view of \eqref{bounded}-\eqref{explosion} we have, as $n\to +\ii$,
$$
\lm_n=\int\limits_{B_{1}}H_{n}e^{\dis v_{n}}=o(1)+\int\limits_{B_{\ov{r}}}H_{n}e^{\dis v_{n}}=
o(1)+\int\limits_{B_{\ov{r}}\setminus B_\eps}f_\ii+\int\limits_{B_\eps}H_{n}e^{\dis v_{n}}=
$$
$$
o(1)+\int\limits_{B_{\ov{r}}\setminus B_\eps}f_\ii+\beta +\int\limits_{B_\eps}f_\ii,
$$
so that, by using \eqref{lambdantoinfty}, the conclusion follows just taking the limit as $n\to +\ii$ and $\eps\to 0^+$.
\end{proof}

\medskip

\begin{remark}\label{rem:llty}
As expected from the examples of blow up without concentration in \cite{llty}, it is not true in general that blow up sequences of
\eqref{eqbase} ``concentrate" in the sense of \cite{bm}, \cite{bt}.
This is why, as far as $\sg>0$, the assumption \eqref{conditionmixed} is needed in Theorem \ref{Concentration}. However, as expected
from \cite{det}, \cite{llty} and \cite{os}, as far as $\sg<0$, indeed blow up implies concentration. Actually, even under much stronger assumptions,
this problem shows up again, causing the phenomenon of blow up-concentration without quantization,
see Theorem \ref{massboundarycontrol} below for more details concerning this point.
\end{remark}

As a consequence of Theorem \ref{Concentration}, in the same spirit of \cite{bt} as later refined in \cite{os}, we have the following concentration-compactness alternative:
\begin{theorem}[Concentration-Compactness]\label{Concentration-Compactness alternative}\hfill \\
Let $v_n$ be a sequence of solutions of \eqref{eqbase}, \eqref{lambdan}, \eqref{lambdantoinfty} where $K_n$ satisfies \eqref{K+} and \eqref{Kconvergence2}.
Assume furthermore that either \eqref{conditionmixed:neg} or \eqref{conditionmixed} holds true.
Then, possibly along subsequence, one of the following alternatives is satisfied:
\begin{itemize}
 \item[(i)] $\underset{B_1}\sup\, \big|v_{n}\big|\leq C_\O$, for every $\O\Subset B_1$;
 \item[(ii)] $\underset{B_1}\sup\, v_{n}\to-\infty$, for every $\O\Subset B_1$;
 \item[(iii)] there exists a finite and nonempty set $S=\{q_1,\ldots,q_m\}\subset B_1$, $m\in\N$, and sequences of points
 $\{x_n^1\}_{n\in\N},\ldots,\{x_n^m\}_{n\in\N}\subset B_1$, such that $x_n^i\to q_i$ and $v_{n}(x_n^i)\to\infty$,
 for every $i\in\{1,\ldots,m\}$. Moreover,
 $\underset{B_1}\sup\, v_{n}\to-\infty$
 on any compact set $\O\subset B_1\backslash S$ and $H_{n}e^{\dis v_{n}}\rightharpoonup \sum_{i=1}^m\beta_i\delta_{q_i}$
weakly in the sense of measures in $B_1$, with $K(q_i)>0$, $\beta_i\in 8\pi\N$ if $q_i\neq0$ while if $q_i=0$ for some $i\in\{1,\ldots,m\}$ then
$\beta_i$ satisfies \eqref{mmbeta}. In particular $\lm_\ii\geq \lm_\sg m$.
\end{itemize}
Otherwise, if $\sg>0$ and the assumption about $\lm_\ii$ in \eqref{conditionmixed} is not satisfied, then either one of the alternatives
{\rm (i),(ii),(iii)} holds true or else
\begin{itemize}
 \item[(iv)] the blow up set coincides with the singular point $q=q_1=0$, $K(q)>0$ and \eqref{vn:bdd}, \eqref{vn:concentration0},
 \eqref{vn:bdd:mmbeta} and  \eqref{vn:bdd:mmbeta2} hold true.
\end{itemize}
\end{theorem}
\begin{proof}
The proof of Theorem \ref{Concentration-Compactness alternative} follows from Lemma \ref{minimalmasslemma} and
Theorem \ref{Concentration} exactly by the same argument adopted in \cite{bt} which is why we omit it here. The unique difference
 is about
the use of \eqref{lambdan} in case $q_i=0$ for some $i\in\{1,\ldots,m\}$, in which case, as far as $\sg<0$, we have,
$$
\lm_\ii \geq 8\pi (m-1)+ 8\pi + 2\lm_\ii \sg,
$$
immediately implying that $\lm_\ii\geq \lm_\sg m$.
\end{proof}

\bigskip

\bigskip

By arguing as in  \cite{bt}, we refine the estimate about the value of the mass relative to the blow up at singular sources whenever
we have a mild control about the oscillation of $v_n$ at the boundary, as first pushed forward in \cite{yy}.

\begin{theorem}[Concentration without quantization] \label{massboundarycontrol} \hfill \\
Under the assumptions of Theorem \ref{Concentration} about $\vn$, $K_n$ and $\lm_n$, suppose in addition that $v_n$ satisfies,
 \begin{equation}
  \label{oscillation}
  \underset{\p B_1}\max\,v_n-\underset{\p B_1}\min\,v_n\leq C.
 \end{equation}
 Then
\begin{equation}\label{beqlm}
 \beta\equiv \lm_\ii
\end{equation}
and, either $\sg<0$ and \eqref{conditionmixed:neg} is satisfied, in which case we have that,
\begin{equation}
 \label{quantized:neg}
8\pi+2\sg\lm_\ii=8\pi(1+\bis)\leq  \beta \leq 8\pi
\end{equation}
and in particular
\begin{equation}\label{quantized:neglm}
\frac{8\pi}{1+2|\sg|}\leq \lm_\ii\leq 8\pi
\end{equation}
or $\sg>0$, that is \eqref{conditionmixed} is satisfied, and then we have,
\begin{equation}\label{sg:nec}
0<\sg\leq \frac12\,,
\end{equation}
\begin{equation}
 \label{quantized:pos}
8\pi\leq  \beta \leq 8\pi(1+\bis)=8\pi+2\sg\lm_\ii
\end{equation}
and in particular
\begin{equation}\label{quantized:poslm}
8\pi\leq \lm_\ii\leq\min\left\{\frac{8\pi}{1-2\sg},\frac{4\pi}{\sg}\right\}.
\end{equation}
\end{theorem}
\begin{proof}\hfill \\
Under the assumptions of Theorem \ref{Concentration} we obviously have that \eqref{bounded}, \eqref{explosion} and then also
\eqref{concentration0}, \eqref{mmbeta}, \eqref{sg:nec} are satisfied. In particular it is not difficult to see that,
$$
\lm_\ii=\lim\limits_{n\to +\ii}\int\limits_{B_1} K_n e^{\dis u_n}=\beta,
$$
which is just \eqref{beqlm}. As a consequence, as soon as we prove \eqref{quantized:neg}, since \eqref{conditionmixed:neg} is satisfied by assumption,
we readily verify that \eqref{quantized:neglm} is satisfied as well. On the other side, if \eqref{quantized:pos} holds true, then we have
$$
\lm_\ii=\beta\leq 8\pi +2\sg \lm_\ii,
$$
immediately implying that, since \eqref{conditionmixed} is satisfied by assumption, then \eqref{quantized:poslm} holds true as well.\\

Therefore in the rest of this proof we will be just concerned with \eqref{quantized:neg}, \eqref{quantized:pos}.

\medskip

Let us define $s_n$ to be the unique solution of,
\begin{equation*}
 \begin{cases}
  -\D s_n=0\quad &\text{in}\,\,B_1,\\
  s_n=v_n-\underset{\p B_1}\min\,v_n\quad &\text{on}\,\,\p B_1.
 \end{cases}
\end{equation*}
By standard elliptic estimates, we have that
\begin{equation*}
 \|s_n\|_\infty\leq C,
\end{equation*}
for a suitable positive constant $C$, and, along a subsequence, we may assume that $s_n\to s$ in $C^1_{loc}(B_1)$. Now, let us consider the function
\begin{equation*}
 w_{n}(x)=v_n(x)-\underset{\p B_1}\min\,v_n-s_n(x),
\end{equation*}
which satisfies,
\begin{equation*}
 \begin{cases}
  -\D w_n=W_n e^{w_n}\qquad \text{in}\,\, B_1, \\
  \int_{B_1}W_ne^{w_n}\leq C,\\
  w_n=0, \qquad \text{on}\,\,\p B_1,
 \end{cases}
\end{equation*}
where
\begin{equation*}
 W_n(x)=H_n(x)e^{\gamma_n(x)}
\end{equation*}
and
\begin{equation*}
 \gamma_n(x)=s_n(x)+\underset{\p B_1}\min \,v_n.
\end{equation*}
We notice that,
\begin{equation*}
 \nabla\gamma_n\to\nabla \gamma\,\,\text{and}\,\, \nabla K_n\to\nabla K, \,\,\text{uniformly on compact sets of $B_1$},
\end{equation*}
with $\gamma=s$. Moreover, since $W_n(x)e^{\dis w_n}=H_ne^{\dis v_n}$, by Theorem \ref{Concentration}, we have that
\begin{equation}
 \label{vanishing}
 W_ne^{\dis w_n}\to0\qquad \text{uniformly on compact sets of $B_1\backslash\{0\}$}
\end{equation}
and
\begin{equation}
 \label{concentrationwn}
 W_ne^{\dis w_n}\rightharpoonup\beta\delta_{p=0}\qquad \text{weakly in the sense of measure in $B_1$},
\end{equation}
where $\beta$ satisfies \eqref{mmbeta}. Moreover,
\begin{equation}
 \label{concentration2}
 h_{n}^{2\bns}e^{\dis \gamma_n}e^{\dis w_n}\rightharpoonup\tfrac{\beta}{K(0)}\delta_{p=0}\qquad \text{weakly in the sense of measure in $B_1$}.
\end{equation}
Let us set $f_n(x)=W_n(x)e^{\dis w_n(x)}$, then by Green's representation formula,
\begin{equation*}
 w_n(x)=-\tfrac{1}{2\pi}\int_{B_1}\ln (|x-y| )f_n(y)\,dy+\int_{B_1}R(x,y)f_n(y)\,dy,
\end{equation*}
where $R(x,y)$ is the regular part of Green's function associated to the Laplacian operator with Dirichlet boundary conditions on $B_1$. Then, passing to the limit in the previous identity, we deduce that
\begin{equation}
 \label{wnconverges}
 w_n(x)\to -\tfrac{\beta}{2\pi}\ln|x|+ g(x)\quad \text{in} \,\,C^1_{\rm loc}(B_1\backslash\{0\}),
\end{equation}
for some $g\in C^1(B_1)$. Let us set,
\[
 w_0(x)= -\tfrac{\beta}{2\pi}\ln|x|+ g(x),
\]
we use the well known Pohozaev identity,
\begin{align}
 \nonumber
 \int_{\p B_r(0)}\Big((x,\nu)\tfrac{\dis |\grad w_n|^2}{2}-&(\nu,\grad w_n)(x,\grad w_n)\Big)d\sigma= \\ \label{Pohozaevnwn}
 &=\int_{\p B_r(0)}(x,\nu)W_n e^{\dis w_n}d\sigma-\int_{B_r}(2W_n+x\cdot\grad W_n)e^{\dis w_n}\,dx.
\end{align}

First of all, observe that
\begin{align*}
 &-\int_{B_r}x\cdot\grad W_ne^{\dis w_n}\,dx= \\
 &=-\int_{B_r}x\cdot\grad (K_ne^{ \gamma_n})h_{n}^{2\bns}e^{\dis w_n}\,dx-
 2\bns\int_{B_r}|x|^2\big(\epsilon_n^2+|x|^2\big)^{\bns-1}K_ne^{ \gamma_n}e^{\dis w_n}\,dx\\
 &=-\int_{B_r}\tfrac{x\cdot\grad (K_ne^{ \gamma_n})}{K_ne^{ \gamma_n}}W_ne^{\dis w_n}\,dx-2\bns\int_{B_r}W_ne^{\dis w_n}\,dx+
 2\bns\int_{B_r}\epsilon_n^2\big(\epsilon_n^2+|x|^2\big)^{\bns-1}K_ne^{\dis \gamma_n}e^{\dis w_n}\,dx \\
 &=-\int_{B_r}\tfrac{x\cdot\grad (K_ne^{ \gamma_n})}{K_ne^{ \gamma_n}}W_ne^{\dis w_n}\,dx-2\bns\int_{B_r}W_ne^{\dis w_n}\,dx+
 2\bns\int_{B_r}\tfrac{\epsilon_n^2}{\epsilon_n^2+|x|^2}W_ne^{\dis w_n}\,dx.
\end{align*}

At this point, we split the discussion in two cases.\\ \\
CASE $\sg<0$.\\
In this case, in view of \eqref{concentrationwn}, we have that,
\begin{align*}
2\beta \bis \leq &\liminf_{n\to\infty}\left(2\bns\int_{B_r}\tfrac{\epsilon_n^2}{\epsilon_n^2+|x|^2}W_ne^{\dis w_n}\,dx\right)\leq \\
 &\limsup_{n\to\infty}\left(2\bns\int_{B_r}\tfrac{\epsilon_n^2}{\epsilon_n^2+|x|^2}W_ne^{\dis w_n}\,dx\right)\leq 0,
\end{align*}
and then, by using also \eqref{vanishing}, \eqref{concentration2}, \eqref{wnconverges},
we can pass to the limit in \eqref{Pohozaevnwn} as $n\to\infty$ and deduce that, for any $r\in(0,1)$,
\begin{align*}
 -2\beta \leq \int_{\p B_r(0)}\Big((x,\nu)\tfrac{|\grad w_0|^2}{2}-&(\nu,\grad w_0)(x,\grad w_0)\Big)d\sigma
 \leq -2\beta(1+\bis),
\end{align*}
that is,
\begin{equation*}
 -2\beta\leq \tfrac{r}{2}\int_{\p B_r}|\grad w_0|^2 -r\int_{\p B_r}(\nu,\grad w_0)^2\leq-2\beta(1+\bis),
\end{equation*}
which implies that, as $r\to0^+$,
\[
 -2\beta\leq -\tfrac{\beta^2}{4\pi}+o(1)\leq-2\beta(1+\bis).
\]
By letting $r\to0^+$, we deduce that
\[
 8\pi\geq \beta\geq 8\pi(1+\bis),
\]
which is just \eqref{quantized:neg}.

\bigskip

CASE $\sg>0$, that is \eqref{conditionmixed} hods true.\\
In this case, in view of \eqref{concentrationwn}, we have that,
\begin{align*}
0 \leq &\liminf_{n\to\infty}\left(2\bns\int_{B_r}\tfrac{\epsilon_n^2}{\epsilon_n^2+|x|^2}W_ne^{w_n}\,dx\right)\leq \\
 &\limsup_{n\to\infty}\left(2\bns\int_{B_r}\tfrac{\epsilon_n^2}{\epsilon_n^2+|x|^2}W_ne^{w_n}\,dx\right)\leq 2\beta \bis,
\end{align*}
whence the same argument adopted in case $\sg<0$ proves that,
\[
 8\pi\leq \beta\leq 8\pi(1+\bis),
\]
which is \eqref{quantized:pos}.
\end{proof}

\bigskip

We conclude this section with some results about the case where $h_{n}^{2\bns}$ is replaced by the the homogenous term $|x|^{2\al_{n,\sg}}$.
In this case, by arguing as in \cite{bt}, it is easy to see that
the assumption \eqref{lambdan} does not affect the concentration-compactness alternative, the overall result being just
that of a modification of the values of the masses in the quantization phenomenon. Some of the needed arguments have been already
worked out above, whence we just report without proof the corresponding results. Let us consider a sequence of solutions of,
\begin{equation}
 \label{eqbasedelta}
 \graf{
 -\D v_n=\left(\prod\limits_{j=1}^N |x-p_j|^{2\al_{n,\sg_j}}\right) K_n e^{\dis v_n}\quad \text{in}\quad  B_1,\\
 -1<\al_{n,\sg_j}=\frac{\lm_n}{4\pi}\sg_j\neq 0,\quad\al_{n,\sg_j}\to \al_{\ii,\sg_j}=\frac{\lm_\ii}{4\pi}\sg_j\notin\{-1,0\},\quad j\in\{1,\ldots,N\},\\
 \int\limits_{B_1}\left(\prod\limits_{j=1}^N |x-p_j|^{2\al_{n,\sg_j}}\right) e^{\dis v_n}\leq C,\\
 \lm_n=\int\limits_{B_1}\left(\prod\limits_{j=1}^N |x-p_j|^{2\al_{n,\sg_j}}\right) K_n e^{\dis v_n}. }
 \end{equation}

\medskip

\begin{theorem}[Concentration-Compactness alternative]\label{Concentration-Compactness alternative0}\hfill \\
Let $v_n$ be a sequence of solutions of \eqref{eqbasedelta}. Suppose that $\lambda_n\to\lambda_\infty>0$, $K_n$ satisfies \eqref{K+} and
\eqref{Kconvergence2}.
Then, possibly along a subsequence, one of the following alternatives is satisfied:
\begin{itemize}
 \item[(i)] $\underset{B_1}\sup\, \big|v_{n}\big|\leq C_\O$, for every $\O\Subset B_1$;
 \item[(ii)] $\underset{B_1}\sup\, v_{n}\to-\infty$, for every $\O\Subset B_1$;
 \item[(iii)] there exists a finite and nonempty set $S=\{q_1,\ldots,q_m\}\subset B_1$, $m\in\N$, and sequences of points
 $\{x_n^1\}_{n\in\N},\ldots,\{x_n^m\}_{n\in\N}\subset B_1$, such that $x_n^i\to q_i$ and $v_{n}(x_n^i)\to\infty$,
 for every $i\in\{1,\ldots,m\}$. Moreover, $\underset{B_1}\sup\, v_{n}\to-\infty$
 on any compact set $\O\subset B_1\backslash S$ and
 $\left(\prod\limits_{j=1}^N |x-p_j|^{2\al_{n,\sg_j}}\right)K_{n}e^{v_{n}}\rightharpoonup \sum_{i=1}^m\beta_i\delta_{q_i}$
weakly in the sense of measure in $B_1$, with $\beta_i\in 8\pi\N$ if $q_i\notin \{p_1,\ldots,p_N\}$ while if $q_i=p_j$ for some $i\in\{1,\ldots,m\}$,
$j\in\{1,\ldots,N\}$ then either $\sg_j<0$ and then $\beta_i\geq 8\pi(1+\alpha_{\ii,\sigma_j})$ or $\sg_j>0$ and then
$\beta_j\geq 8\pi$. Let $J_{-}\subset\{1,\ldots,N\}$ be the set of indices of singular blow up points with $\sg_j<0$, then in particular
$$
\lm_\ii\geq \tfrac{8\pi}{1+2\sum\limits_{j\in J_{-}}^{\,}|\sigma_j|} m.
$$

\end{itemize}
\end{theorem}

\begin{theorem}[Mass Quantization with Mild Boundary Conditions] \label{massboundarycontrol0} \hfill \\
  Under the assumptions of Theorem \ref{Concentration-Compactness alternative0}, suppose in addition that
  $u_n$ satisfies,
 \begin{equation*}
  \underset{\p B_1}\max\,v_n-\underset{\p B_1}\min\,v_n\leq C.
 \end{equation*}
If ${\rm (iii)}$ occurs and $q_i=p_j$ for some $i\in\{1,\ldots,m\}$, $j\in\{1,\ldots,N\}$, then
\begin{equation}
 \label{quantized0}
 \beta_i=8\pi(1+\alpha_{\ii,j}).
\end{equation}
In particular, assume that $\lm_\ii=\sum\limits_{i=1}^m\beta_i$ and let $J\subset\{1,\ldots,N\}$ be the set of indices of singular blow up points. Then necessarily
$$
\sum\limits_{j\in J}^{\,}\sigma_j<\frac12,
$$
and
\begin{equation}\label{massphys}
\lm_\ii=\tfrac{8\pi}{1-2\sum\limits_{j\in J}^{\,}\sigma_j} m.
\end{equation}
\end{theorem}

\begin{remark} Concerning the proof of \eqref{massphys}, the assumption $\lm_\ii=\sum\limits_{i=1}^m\beta_i$ is essentially just equivalent to the
lack of boundary blow up points for for $v_n$ in $B_1$.
Thus we infer from \eqref{quantized0} that,
$$
\lm_\ii=8\pi (m-|J|)+ 8\pi |J|+2\lm_\ii\sum\limits_{j\in J}^{\,}\sigma_j,
$$
immediately implying that necessarily we have $\sum\limits_{j\in J}^{\,}\sigma_j<\frac12$, in which case
\begin{equation}
\label{lambdaexplicit:0}
 \lm_\ii=\tfrac{8\pi}{1-2\sum\limits_{j\in J}^{\,}\sigma_j} m.
\end{equation}
\end{remark}

As an immediate consequence of Theorem \ref{massboundarycontrol0} we have the following:
\begin{corollary}\label{co:bdd}
Under the assumptions of Theorem \ref{Concentration-Compactness alternative0}, suppose in addition that
  $u_n$ satisfies,
 \begin{equation*}
  \underset{\p B_1}\max\,v_n-\underset{\p B_1}\min\,v_n\leq C,
 \end{equation*}
$\lm_\ii=\sum\limits_{i=1}^m\beta_i$ and that
$$
\sum\limits_{j\in J}^{\,}\sigma_j\geq \frac12.
$$
Then $v_n$ is uniformly bounded from above near  $\{p_1,\ldots,p_N\}$.
\end{corollary}

\bigskip
\bigskip

\section{Blow up analysis: profile estimates without quantization}\label{sec2.1}
We will need the following Harnack-type inequality recently obtained in \cite{bcwyzHarnackInequality}.
\begin{theorem}[\cite{bcwyzHarnackInequality}]\label{thm:sup+inf}
Let $\O$ be an open bounded domain in $\R^2$ which contains the origin, $\{0\}\subset \O$. Assume that
$v_n$ is a sequence of solutions of
\begin{equation*}\label{intro:eq}
 -\D v_n=\left({\epsilon_n^2+|x|^2}\right)^{\alpha_n}K_n(x)e^{\displaystyle v_n} \qquad\text{in} \,\,\, \Omega,
\end{equation*}
where $\epsilon_n\to0^+$, as $n\to+\infty$,
\begin{equation}
 \label{intro:alphan}
 \a_n\to\a_\infty\in(-1,1),
\end{equation}
and $K_n$ satisfies,
\begin{equation*}\label{intro: Vn}
 0<a\leq K_n\leq b<+\ii, \quad K_n\in C^{0}(\O), \quad K_n\to K \,\, \text{locally uniformly in}\,\,{\O}.
\end{equation*}
Then, for any
 $$ C_0>\max\{1,\tfrac{1+\alpha_\infty}{1-\alpha_\infty}\}$$
and for any compact set $\om^{'}\Subset\O$, there exists a constant $C_1>0$, which depends only by $a,b,dist(\om^{'},\p\O), \alpha_\infty,$ and by the
uniform modulus of continuity of $K$ on $\om^{'}$, such that,
\begin{equation}
 \label{intro:sup+Cinf}
 \underset{\om^{'}}\sup\, v_n + C_0\underset{\O}\inf\, v_n\leq C_1.
\end{equation}
\end{theorem}

Remark that no assumption is made at all about $\int\limits_{\om}\left({\epsilon_n^2+|x|^2}\right)^{\alpha_n}K_n(x)e^{\displaystyle v_n}$, whence unlike 
problem \eqref{eqbase}-\eqref{lambdan}, here $\al_n$  is any sequence of numbers satisfying \eqref{intro:alphan}.

In view of Theorem \ref{thm:sup+inf}, in the same spirit of \cite{ls} we have a partial quantization result for solutions
of \eqref{eqbase}, \eqref{lambdan}, \eqref{lambdantoinfty}. It is of independent interest as it does not rely on boundary conditions but rather just on
the condition
\beq\label{conditionmixed-n}
\lm_\ii<\frac{4\pi}{\sg}.
\eeq
Recall that $\dt_n$ always denotes the quantity defined in \eqref{dtn}.
\begin{lemma}[A quantization Lemma] \label{LS-quant} Let $v_n$ be a sequence of solutions of \eqref{eqbase}, \eqref{lambdan}, \eqref{lambdantoinfty} where ${\sg}>0$,
$K_n$ satisfies \eqref{K+}, \eqref{Kconvergence1} and $\lm_\ii$ satisfies \eqref{conditionmixed-n}. Assume that \eqref{bounded} and \eqref{explosion} 
are satisfied and that  there exists $d_n\to 0^+$ with the following properties:\\
$$
\frac{\eps_n}{d_n}\leq C,\qquad \frac{\dt_n}{d_n}\to 0,
$$
$$
\int\limits_{B_{4d_n}}H_n e^{\dis v_{n}}\to \beta\geq 8\pi,
$$
and
\begin{equation}
\label{hypols-lem}
v_n(x)+2(1+\al_{n,\sg})\log(|x|)\leq C, \;\forall\, 4d_n\leq|x|\leq 1,
\end{equation}
for some $C>0$. Then
$$
\int\limits_{B_{1}\setminus B_{4d_n}}H_n e^{\dis v_{n}}\to 0\quad\mbox{and}\quad
\lm_n\to \lm_\ii=\beta.
$$
\end{lemma}
\begin{proof}\hfill \\
We will use $C>0$ to denote several constants whose value may change possibly line to line all along the proof.\\
We argue as in \cite{ls}. In view of \eqref{Kzero}, there exists $r_0>0$ such that $\min\limits_{\overline{B_{r_0}}} K>0$. 
Let us fix any $r\in [\tfrac{r_0}{4},\tfrac{r_0}{2}]$ and set $s\in[4 d_n,r]$, 
$$
\hat\O:=B_2\backslash B_{\frac{1}{2}},
$$
and
\[
 \hat v_n(x)=v_n(sx)+2(1+\bns)\log s,\quad x\in B_2.
\]
Then,
\[
 -\D \hat v_n(x)=\left({\frac{\epsilon_n^2}{s^2}+|x|^2}\right)^{\bns}K_n(sx)e^{\dis \hat v_{n}(x)}=:f_n(x) \quad\text{in}\quad B_2,
\]
and by \eqref{hypols-lem} and the fact that 
$$
\tfrac{\epsilon_n}{s}\leq\tfrac{\epsilon_n}{4d_n}\leq\tfrac{C}{4},
$$  
we have $\|f_n\|_{L^{\infty}(\hat \O)}\leq C$. Let $\kappa_n$ be the solution of
\begin{align*}
\begin{cases}
-\D \kappa_n=f_n & \hat\O, \\
\kappa_n=0 &\p\hat\O,
\end{cases}
\end{align*}
by standard elliptic estimates we have $\|\kappa_n\|_{L^\infty(\hat \O)}\leq C$. Then, once again by \eqref{hypols-lem},  the harmonic function $g_n=\kappa_n-\hat v_n$ is bounded from below by $-C$ and, by the Harnack inequality, there exists a universal constant $\tau_0\in(0,1)$ such that
\[
 \tau_0\sup_{\p B_1} (g_n+C)\leq \inf_{\p B_1}(g_n+C).
\]
Moving back to $\hat v_n$ we see that,
\[
 \sup_{\p B_1} \hat v_n\leq\tau_0 \inf_{\p B_1}\hat v_n+\hat C,
\]
which implies that,
\begin{equation}
 \label{HarnackI-lem}
 \sup_{\p B_s} v_n\leq\tau_0 \inf_{\p B_s}v_n-2(1-\tau_0)(1+\bns)\log  s+\hat C,
\end{equation}
for every $s\in[4d_n,r]$. On the other hand, by assumption we have \eqref{conditionmixed-n},
that is, $\al_{\ii,\sg}=\frac{\lm_\ii}{4\pi}\sg<1$, and consequently that $\hat v_n$ satisfies \eqref{intro:sup+Cinf}, namely, there exist two positive constants $C_0,C_1$, depending by $\al_{\ii,\sg}$, such that
\[
 \sup_{B_{\frac{1}{2}}} \hat v_n+C_0\inf_{B_1}\hat v_n\leq C_1,
\]
which immediately implies that,
\[
 \sup_{B_{\frac{s}{2}}} v_n+C_0\inf_{B_s} v_n\leq C_1-2(1+C_0)(1+\bns)\log s,
\]
for every $s\in[4d_n,r]$.  By noticing that  $ \sup\limits_{B_{\frac{s}{2}}} v_n= v_n(x_n)$, we can rewrite the above inequality as follows,
\begin{equation}
 \label{sup+infI-lem}
 \inf_{B_s} v_n\leq \tfrac{C_1}{C_0}-2(1+\tfrac{1}{C_0})(1+\bns)\log s-\tfrac{1}{C_0}v_n(x_n).
\end{equation}
Combining \eqref{HarnackI-lem} and \eqref{sup+infI-lem} and using the superharmonicity of $v_n$, we have that
\[
 \sup_{\p B_s}v_n\leq \tau_0\tfrac{C_1}{C_0}+\hat C-\tfrac{\tau_0}{C_0}v_n(x_n)-2(1+\tfrac{\tau_0}{C_0})(1+\bns)\log s,
\]
for every $s\in[4d_n,r]$, or equivalently,
\[
 e^{v_n(x)}\leq C_2\delta_n^{2(1+\bns)\tfrac{\tau_0}{C_0}}|x|^{-2(1+\bns)(1+\tfrac{\tau_0}{C_0})},
\]
for every $|x|\in[4d_n,r]$. In particular, this implies that
\begin{align*}
 &\int\limits_{B_{\frac{r_0}{2}}\setminus B_{4d_n}}H_n e^{\dis v_{n}}=\hspace{-.3cm}\int\limits_{B_{\frac{r_0}{2}}\setminus B_{4d_n}}(\eps_n^2+|x|^2)^{\bns} K_n e^{\dis v_{n}}\leq\\
 & C\hspace{-.3cm}\int\limits_{B_{\frac{r_0}{2}}\setminus B_{4d_n}}(\eps_n^2+|x|^2)^{\bns} \left(\frac{\delta_n}{|x|}\right)^{2(1+\bns)\tfrac{\tau_0}{C_0}} \left(\frac{1}{|x|}\right)^{2(1+\bns)}\leq\\
 & C\left(\frac{\delta_n}{d_n}\right)^{2(1+\bns)\tfrac{\tau_0}{C_0}}\underset{4\leq|z|\leq \tfrac{1}{d_n}}\int \left(\frac{\eps_n^2}{d_n^2}+|z|^2\right)^{\bns} \left(\frac{1}{|z|}\right)^{2(1+\bns)(1+\tfrac{\tau_0}{C_0})}\leq
 \\
 &C\left(\frac{\delta_n}{d_n}\right)^{2(1+\bns)\tfrac{\tau_0}{C_0}}\underset{4\leq|z|\leq \tfrac{1}{d_n}}\int \left(\frac{1}{|z|}\right)^{2}\left(\frac{1}{|z|}\right)^{2(1+\bns)\tfrac{\tau_0}{C_0}}\leq C\left(\frac{\delta_n}{d_n}\right)^{2(1+\bns)\tfrac{\tau_0}{C_0}}\to 0,
\end{align*}
where we used $\frac{\delta_n}{d_n}\to 0$. At this point, by using also \eqref{-infinity}, we deduce that,
\begin{align*}
\lambda_n=\int\limits_{B_{4d_n}}H_n e^{\dis v_{n}}+\int\limits_{B_{\frac{r_0}{2}}\setminus B_{4d_n}}H_n e^{\dis v_{n}}+o(1)=\beta+o(1),
\end{align*}
as claimed.
\end{proof}

\bigskip
\bigskip

At this point we provide a refined estimate about the phenomenon of concentration without quantization in Theorem \ref{massboundarycontrol}.
In view of the application in section \ref{sec6}, we will be concerned only with the more delicate case $\sg>0$, where the the energy terms due to the vorticity and the singularity have opposite sign (see \eqref{enrg:n}).
One of the allowed limiting profiles corresponds after blow up to solutions of the following planar problem,
\begin{align}
\begin{cases}\label{profile:tildew}
 -\D \widetilde w=\widetilde V(y) e^{\dis \widetilde w} \quad\text{in}\quad \R^2, \\
\widetilde V(y)=c_0\left(\epsilon_0^2+|y|^2\right)^{\bis},\quad c_0>0,\\
\widetilde \beta := \int_{\R^2}\widetilde V e^{\dis \widetilde w}\leq\lm_\ii<16\pi,\\
\widetilde w(y)\leq \widetilde w(y_0)=0,\quad y_0\in\R^2.
\end{cases}
\end{align}
We summarize few facts about \eqref{profile:tildew} right before Lemma \ref{massstrange} in the Appendix.\\

Recall that we always assume \eqref{lambdantoinfty}, that is $\lm_n\to \lm_\ii>0$.
\begin{theorem}\label{profile-new} Under the hypothesis of Theorem \ref{massboundarycontrol},
assume furthermore that,
$$
\sg\in\left(0,\frac12\right)\quad \mbox{and} \quad 8\pi\leq \lm_\ii< \frac{4\pi}{\sg}
$$ and set
\[
 t_n=\max\{\delta_n,|x_n|\}.
\]
Then, either,
\begin{itemize}
 \item \mbox{\rm (I):}  there exists a subsequence such that $\tfrac{\epsilon_n}{t_n}\to+\infty$,
\end{itemize}
in which case we have $\lm_\ii=8\pi$ and
\begin{equation}\label{profilev:H}
v_n(x)=\log\left(\dfrac{e^{\dis v_n(x_n)}}
{\left(1+\gamma_n\theta_n^{2(1+\bns)}\epsilon_n^{-\frac{\lm_n}{4\pi}}|x-{x_n}|^{\frac{\lm_n}{4\pi}}\right)^2}\right)+O(1),\quad z\in B_{r}(0);
\end{equation}
where
$\theta_n^{2(1+\al_{n,\sg})}=
\left(\tfrac{\epsilon_n}{\delta_n}\right)^{2(1+\al_{n,\sg})}\to +\ii$, $8 \gamma_n={(1+|\frac{x_n}{\epsilon_n}|^2)^{\bns} K_n({x_n})}$, or
\begin{itemize}
 \item \mbox{\rm (II):} there exists a subsequence such that $\tfrac{\epsilon_n}{t_n}\leq C$ and $\tfrac{|x_n|}{\delta_n}\to+\infty$,
\end{itemize}
in which case$\lm_\ii=8\pi$ and
\begin{equation}\label{profilev:H1}
v_n(x)=\log\left(\dfrac{e^{\dis v_n(x_n)}}
{\left(1+\bar\gamma_n\bar \theta_n^{2(1+\al_{n,\sg})}|x_n|^{-\frac{\lm_n}{4\pi}} |x-{x_n}|^{\frac{\lm_n}{4\pi}}\right)^2}\right)+O(1),\quad z\in B_{r}(0),
\end{equation}
where $\bar \theta_n^{2(1+\al_{n,\sg})}=\left(\tfrac{|x_{n}|}{\delta_n}\right)^{2(1+\al_{n,\sg})}\to +\ii$,
$8\bar\gamma_n={((\tfrac{\eps_n}{|x_n|})^2+1)^{\bns}}K_n(x_n)$, or
\begin{itemize}
 \item \mbox{\rm (III):} there exists a subsequence such that $\tfrac{\epsilon_n}{t_n}\leq C$ and
 $\tfrac{|x_n|}{\delta_n}\leq C$,
\end{itemize}
in which case, along a further subsequence if necessary, we have $\tfrac{\eps_n}{\delta_n}\to \eps_0\geq 0$ and\\ $\lm_\ii\in (8\pi,\tfrac{8\pi}{1-2\sg}]$, if $\sigma\in(0,\tfrac{1}{4})$, or $\lm_\ii\in(8\pi,\tfrac{4\pi}{\sigma})$, if $\sg\in[\tfrac{1}{4},\tfrac12)$ and
\begin{equation}\label{profilevtilde:H1-IIb}
 v_n(x)=v_n(x_n)+ \widetilde U_n(\delta_n^{-1}x)+O(1), \quad x\in B_{r}(0),
\end{equation}

where
\begin{equation*}
    \widetilde U_n(y)=\graf{\widetilde w(y)+O(1),\quad |y|\leq R \\ -\frac{\lm_n}{2\pi}\log(|y|) + O(1),\quad R\leq  |y|\leq r\delta_n^{-1}}
\end{equation*}
and $\widetilde w$ is the unique solution of \eqref{profile:tildew}.
\end{theorem}

\proof
Recall that, in view of Theorems \ref{Concentration}, \ref{Concentration-Compactness alternative} and \ref{massboundarycontrol}, we have
that,
$$
H_{n}e^{\dis v_n}\rightharpoonup \beta \delta_{p=0},\quad \beta=\lm_\ii,
$$
$$
\lm_n=\ino H_{n}e^{\dis v_n}\to \lm_\ii,
$$
where,
\beq\label{mass:lminfty1}
\graf{8\pi\leq \lm_\ii\leq \frac{8\pi}{1-2\sg},\quad \sg\in (0,\frac14)\\$\,$\\ 8\pi\leq \lm_\ii< \frac{4\pi}{\sg},\quad \sg\in [\frac14,\frac12)}
\eeq
whence in particular,
$$
\lm_\ii<16\pi.
$$

Remark that, putting $m_n=\inf\limits_{B_1} v_n$, in view of \eqref{oscillation} we have that,
\begin{equation}\label{profilewn-n}
v_n(x)-m_n=\lm_\ii G_\om(x,0)+O(1),\quad x\in B_1\setminus B_r(0).
\end{equation}

Here and in the rest of this proof we set,
$$
c_0=K(0).
$$
Putting,
\[
 \delta_n^{2+2\bns}=e^{-{\dis v_n(x_n)}}\to 0,\,\, \text{as} \,\, n\to\infty,
\]
we proceed to the discussion of the cases in the same order as in the (minimal mass) Lemma \ref{minimalmasslemma},  
recalling that
$
 t_n=\max\{\delta_n,|x_n|\}.
$

\begin{itemize}
 \item {\bf CASE} (I): there exists a subsequence such that $\tfrac{\epsilon_n}{t_n}\to+\infty$.
\end{itemize}

Recall that in this case \eqref{epsilondelta&epsilonxn} is satisfied (that is $\frac{\epsilon_n}{\delta_n}\to+\infty$ and $\frac{\epsilon_n}{|x_n|}\to+\infty$)
and we define $B^{(n)}$ and $w_n$ as in \eqref{bn1}, \eqref{wn:IA}. Therefore $w_n$ is a solution of
\begin{align}\label{eqwn:IA2}
\begin{cases}
-\D w_n=V_n(z)e^{\dis w_n} \quad\text{in}\quad B^{(n)},\\
V_n(z)=\left({1+|z|^2}\right)^{\bns}K_n(\epsilon_n z)\to
c_0\left({1+|z|^2}\right)^{\bis}\;\mbox{in}\;C^t_{loc}(\R^2),\\
\int\limits_{B^{(n)}}V_n e^{\dis w_n}= \lm_n +o(1) = \lm_\ii +o(1),\\
w_n(\tfrac{x_n}{\epsilon_n})=2(1+\bns)\log(\tfrac{\epsilon_n}{\delta_n})\to +\infty.
\end{cases}
\end{align}

\begin{remark}\label{rem:reg} Clearly the concentration compactness theory (\cite{bm}, \cite{yy})
for regular Liouville type equations can be applied to $w_n$ on any compact subset of $\R^2$, implying that
the blow up set relative to $w_n$ is finite and the limiting total mass cannot exceed $\lm_\ii$. Actually, for any compact set $U\subset \R^2$
far away from the blow up set of $w_n$, since $v_n$ satisfies \eqref{oscillation}, by Lemma 2.2 in \cite{bclt} we have that,
\begin{equation}\label{osc:wnU}
  \underset{U}\max\,w_n-\underset{U}\min\,w_n\leq C_U.
\end{equation}
Therefore, in particular
\begin{equation*}
  \underset{\p B_{R}(0)}\max\,w_n-\underset{\p B_{R}(0)}\min\,w_n\leq C_{R},
\end{equation*}
for any $R$ small enough and by the result in \cite{yy} the origin is a simple blow up point for $w_n$,
$$
\int\limits_{B_{R}(0)}V_n e^{\dis w_n}=8\pi +o(1).
$$
\end{remark}

We want to show that, for any $\dt\in (0,2]$, there exist $C_{r,\dt}>0$ such that
\begin{equation}
\label{hypolsI-A}
 \sup_{B_{2r}\backslash B_{\dt\epsilon_n}}\{v_n(z)+2(1+\bns)\log |z|\}\leq C_{r,\dt}.
\end{equation}
By contradiction, we assume the existence of a sequence $z_n$ such that,
\[
 v_n(z_n)+2(1+\bns)\log |z_n|\to+\infty, \,\,\text{as}\,\,n\to+\infty,
\]
\[
 \mbox{and }\quad \dt\epsilon_n\leq |z_n|\leq 2r.
\]
Then, let us define
\[
  v_{n,1}(z)= v_n(|z_n|z)+2(1+\bns)\log |z_n|
\]
which satisfies
\begin{align*}
\begin{cases}
-\D  v_{n,1}= V_{n,1}(z)e^{\dis v_{n,1}} \qquad\text{in}\,\,B_{\frac{r}{|z_{n}|}}(0)\\
V_{n,1}(z)=\left({\frac{\epsilon_n^2}{|z_{n}|^2}+|z|^2}\right)^{\bns}K_{n}(|z_n| z)
\to c_0\left({\epsilon_{0,1}^2+|z|^2}\right)^{\bis}\;\mbox{in}\;C^t_{loc}(\R^2)\\
\int_{B_{\frac{r}{|z_n|}}} V_{n,1} e^{\dis  v_{n,1}}= \lm_n +o(1)= \lm_\ii +o(1),\\
 v_{n,1}(\tfrac{z_n}{|z_n|})=v_n(z_n)+2(1+\bns)\log |z_n|\to+\infty,
\end{cases}
\end{align*}
where we used the fact that
$\tfrac{\epsilon_n}{|z_n|}\to \epsilon_{0,1}$, where $\epsilon_{0,1}\leq\tfrac{1}{\dt}$.\\
Possibly along a subsequence we can assume that $\tfrac{z_n}{|z_n|}\to \bar z,\;|\bar z|=1$ and then, as in Remark \ref{rem:reg},
the concentration-compactness theory (\cite{bm},\cite{yy}) for regular Liouville type equations applies to $ v_{n,1}$ in $B_R(\bar z)$ for any
$R\leq \frac12$ and in particular we have that,
\begin{equation*}
  \underset{\p B_R(\bar z)}\max v_{n,1}-\underset{\p B_R(\bar z)}\min v_{n,1}\leq C_R,\quad \forall \, R\leq\frac{1}{2},
\end{equation*}
whence $z_0$ is a simple blow up point (\cite{yy}), i.e.
$$
\int\limits_{B_R(\bar z)}  V_{n,1} e^{v_{n,1}}\to 8\pi,
$$
for any $R\leq \frac12$. Remark that for any $R\leq \frac14$ and $r\leq \frac14$ we have,
$$
\mbox{dist}(B_{R}(\tfrac{x_n}{\epsilon_n}),B_{r}(\tfrac{z_n}{|z_n|}))>0,
$$
for $n$ large enough, and then we have,
$$
16\pi>\lm_\ii\geq \liminf\limits_{n\to +\ii} \int\limits_{B_R(0)} V_n e^{\dis w_n}+
\liminf\limits_{n\to +\ii} \int\limits_{B_r(\bar z)}  V_{n,1} e^{\dis  v_{n,1}}=16\pi,
$$
which is a contradiction and thereby proves \eqref{hypolsI-A}. At this point, in view of \eqref{hypolsI-A} and since $\frac{\lm_\ii}{4\pi}\sg<1$,  we can apply Lemma \ref{LS-quant} with
$$
d_n=\frac{\eps_n}{4},
$$
to deduce that,
\begin{align*}
 \lambda_n=8\pi+o(1),
\end{align*}
that is,
\beq\label{decay}
\lm_\ii=8\pi \quad  \mbox{ and } \quad 2\alpha_{n,\sg}\to \frac{\lm_\ii}{2\pi}\sg =4\sg<2.
\eeq

Next, we use once more that, from \eqref{hypolsI-A}, for any $R>0$,
\begin{equation}\label{bdd:w}
\sup\limits_{B^{(n)}\setminus B_R(0)}w_n\leq C_R.
\end{equation}

In view of \eqref{osc:wnU}, for any $R>0$ we also have that,
\begin{equation*}
  \underset{\p B_R}\max\,w_n-\underset{\p B_R}\min\,w_n\leq C_R.
\end{equation*}
By Remark \ref{rem:reg} and by the pointwise estimates in \cite{yy} we have that,
\begin{equation}\label{profilew}
w_n(z)=\log\left(\dfrac{\theta_n^{2(1+\al_{n,\sg})}}
{\left(1+\gamma_n\theta_n^{2(1+\al_{n,\sg})}|z-\frac{x_n}{\epsilon_n}|^2\right)^2}\right)+O(1),\quad z\in B_{R}(0),
\end{equation}
for any $R>0$, where,
$$
\theta_n^{2(1+\al_{n,\sg})}=e^{\dis w_n(\tfrac{x_n}{\epsilon_n})}=e^{\dis v_n(x_n)}\epsilon_n^{2(1+\al_{n,\sg})}=
\left(\tfrac{\epsilon_n}{\delta_n}\right)^{2(1+\al_{n,\sg})}
$$
and $\gamma_n=\frac{V_{n}(\tfrac{x_n}{\epsilon_n})}{8}$. At this point,
via a careful adaptation of pointwise estimates in \cite{bclt} and \cite{bt2}, we deduce that,
\begin{equation}\label{profilew:H}
w_n(z)=\log\left(\dfrac{\theta_n^{2(1+\al_{n,\sg})}}
{\left(1+\gamma_n\theta_n^{2(1+\al_{n,\sg})}|z-\frac{x_n}{\epsilon_n}|^{\frac{M_n}{4\pi}}\right)^2}\right)+O(1),\quad z\in B^{(n)},
\end{equation}
where
\[
 M_n=\int_{B^{(n)}}V_n e^{\dis w_n}=\lm_n+o(1).
\]
To simplify the exposition, we postpone the details of the proof of \eqref{profilew:H} to the Appendix, see Section \ref{app:est-1}.

 At this point we wish to show that \eqref{profilew:H} still holds true with $M_n$ replace by $\lm_n$. Indeed, remark that for any fixed $|x|>0$, we deduce from \eqref{profilewn-n} that,
\beq\label{vn:cn}
v_n(x)=O(1)+m_n,
\eeq
while from \eqref{profilev:H} that,
$$
v_n(x)=2\log\left(\tfrac{\delta_n^{1+\alpha_{n,\sg}}}{\epsilon_n^{2(1+\al_{n,\sg})-\frac{M_n}{4\pi}}}\right)+O(1),
$$
that is
\begin{equation}\label{eps:I-A}
m_n= -2\log\left(\tfrac{\epsilon_n^{2(1+\al_{n,\sg})-\frac{M_n}{4\pi}}}{\delta_n^{1+\alpha_{n,\sg}}}\right) + O(1).
\end{equation}
Recalling \eqref{decay}, since $M_n=\lm_n+o(1)=8\pi+o(1)$ and $\alpha_{n,\sg}=2\sg +o(1)$, we also have that,
\beq\label{new:n}
\frac{\epsilon_n^{2(1+\al_{n,\sg})-\frac{M_n}{4\pi}}}{\delta_n^{1+\alpha_{n,\sg}}}=\frac{\epsilon_n^{(1+\al_{n,\sg})+1+2\sg -2 +o(1)}}{\delta_n^{1+\alpha_{n,\sg}}}=\theta_n^{(1+\bns)}\epsilon_n^{2\sg-1+o(1)}\to+\infty,
\eeq
where we use $2\sg<1$. At this point, because of \eqref{eps:I-A}, \eqref{new:n}, we see that,
\[
0\leq \lm_n-M_n= \int\limits_{\om\setminus B_r}H_n e^{\dis v_n(x)}\leq C_r\theta_n^{-2(1+\bns)}\epsilon_n^{2(1-2\sg)+o(1)}.
\]
As a consequence we also have that, for $R\leq |z|\leq \frac{1}{\eps_n}$,
$$
\left|\log\left(\dfrac{\left(1+\gamma_n\theta_n^{2(1+\al_{n,\sg})}|z-\frac{x_n}{\epsilon_n}|^{\frac{\lm_n}{4\pi}}\right)^2}
{\left(1+\gamma_n\theta_n^{2(1+\al_{n,\sg})}|z-\frac{x_n}{\epsilon_n}|^{\frac{M_n}{4\pi}}\right)^2}\right)\right|=
O(1)+\frac{1}{2\pi}({\lm_n}-{M_n})\log(|z|)\leq
$$
$$
O(1)-C_r\theta_n^{-2(1+\bns)}\epsilon_n^{2(1-2\sg)+o(1)}\log{\eps_n}=O(1),
$$
and then we see from \eqref{profilew:H} that in fact,
\begin{equation}\label{profilew:H-lm}
w_n(z)=\log\left(\dfrac{\theta_n^{2(1+\al_{n,\sg})}}
{\left(1+\gamma_n\theta_n^{2(1+\al_{n,\sg})}|z-\frac{x_n}{\epsilon_n}|^{\frac{\lm_n}{4\pi}}\right)^2}\right)+O(1),\quad z\in B^{(n)},
\end{equation}
which in particular in terms of $v_n$ takes the form \eqref{profilev:H}. This fact concludes the discussion of {CASE} (I).

\bigskip

Next, we discuss the following,
\begin{itemize}
 \item {\bf CASE} (II): there exists a subsequence such that $\tfrac{\epsilon_n}{t_n}\leq C$ and $\tfrac{|x_n|}{\delta_n}\to+\infty$.
\end{itemize}

In this case \eqref{epsilonxn} is satisfied (that is $\frac{\epsilon_n}{|x_n|}\leq C$)
and we define $D^{(n)}$, $\bar w_n$, as in \eqref{dn1}, \eqref{barwn:IB}. Therefore $\bar w_n$ is a solution of \begin{align}
\begin{cases} \label{eqwn:IB1}
-\D \bar w_{n}=\bar V_{n}(z)e^{\dis \bar w_{n}} \quad\text{in}\quad D^{(n)},\\
\bar V_{n}(z)=\left({\frac{\epsilon_n^2}{|x_{n}|^2}+|z|^2}\right)^{\bns}K_{n}(|x_{n}| z)
\to c_0(\bar\epsilon_0^2+|z|)^{2\bis}\;\mbox{in}\;C^t_{loc}(\R^2),\\
\int\limits_{D^{(n)}}\bar V_n e^{\dis \bar w_n}= \lm_n +o(1)= \lm_\ii +o(1),\\
\bar w_{n}(\tfrac{x_{n}}{|x_{n}|})=2(1+\bns)\log(\tfrac{|x_{n}|}{\delta_{n}})\to +\infty,
\end{cases}
\end{align}
for some $\bar \epsilon_0\geq 0$. Possibly along a subsequence we can assume that $\tfrac{x_{n}}{|x_{n}|}\to z_1,\;|z_1|=1$ and then as in Remark \ref{rem:reg},
the concentration-compactness theory (\cite{bm},\cite{yy}) for regular Liouville type equations can be applied
to $\bar w_{n}$ so that in particular \eqref{osc:wnU}  holds for $\bar w_{n}$ as well.
By using  the fact that,
$$
\int\limits_{D^{(n)}} \bar V_n e^{\dis \bar w_n}\leq \lm_\ii<16\pi,
$$
together with the (minimal mass) Lemma \ref{minimalmasslemma}, we first check that $z_1$ is the unique  blow up point in, say, $B_R$ for any $R\leq 5$. Thus, for any such $R$ and for any compact subset $U \Subset B_R\setminus\{ 0\}$
we have
\begin{equation*}
  \underset{U}\max\,\bar w_{n}-\underset{U}\min\,\bar w_{n}\leq C_U,
\end{equation*}
whence in particular $z_1$ is a simple blow up point (\cite{yy}), i.e.
\begin{equation}\label{hypolsI-B.1}
\int\limits_{B_R(z_1)} \bar V_{n} e^{\dis \bar w_{n}}\to 8\pi,
\end{equation}
for any $R\leq 4$. We want to show that there exists $C_{r}>0$ such that
\begin{equation}
\label{hypolsI-B}
 \sup_{B_{2r}\backslash B_{2|x_n|}}\{v_n(z)+2(1+\bns)\log |z|\}\leq C_{r}.
\end{equation}
By contradiction, we assume the existence of a sequence $z_n$ such that
\[
 v_n(z_n)+2(1+\bns)\log |z_n|\to+\infty, \,\,\text{as}\,\,n\to+\infty,
\]
\[
 \mbox{and }\quad 2|x_n|\leq |z_n|\leq 2r.
\]

Then, let us define
\[
 \bar v_n(z)= v_n(|z_n|z)+2(1+\bns)\log |z_n|
\]
which satisfies
\begin{align*}
\begin{cases}
-\D \bar v_{n}=\bar V_{n,1}(z)e^{\dis \bar v_{n}} \qquad\text{in}\,\,B_{\frac{r}{|z_{n}|}}(0),\\
\bar V_{n,1}(z)=\left({\frac{\epsilon_n^2}{|z_{n}|^2}+|z|^2}\right)^{\bns}K_n(|z_n|z)\to c_0|z|^{2\bis}\;\mbox{in}\;C^t_{loc}(\R^2),\\
\int_{B_{\frac{r}{|z_n|}}}\bar V_{n,1} e^{\dis \bar v_{n}}= \lm_n +o(1) = \lm_\ii +o(1),\\
\bar v_{n}(\tfrac{z_n}{|z_n|})=v_n(z_n)+2(1+\bns)\log |z_n|\to+\infty,
\end{cases}
\end{align*}
where we used the fact that
\[
 \tfrac{\epsilon_n}{|z_n|}=\tfrac{\epsilon_n}{|x_n|}\tfrac{|x_n|}{|z_n|}\leq\tfrac{1}{2}C.
\]

Possibly along a subsequence we can assume that $\tfrac{z_n}{|z_n|}\to \bar z,\;|\bar z|=1$ and then, as in Remark \ref{rem:reg},
the concentration-compactness theory (\cite{bm},\cite{yy}) for regular Liouville type equations applies to $\bar v_{n}$ in $B_R(\bar z)$ for any
$R\leq \frac12$ and in particular we have that,
\begin{equation*}
  \underset{\p B_R(\bar z)}\max\bar v_n-\underset{\p B_R(\bar z)}\min\bar v_n\leq C_R,\quad \forall \, R\leq\frac{1}{2},
\end{equation*}
whence $z_0$ is a simple blow up point (\cite{yy}), i.e.
$$
\int\limits_{B_R(\bar z)} \bar V_{n,1} e^{\bar v_n}\to 8\pi,
$$
for any $R\leq \frac12$. Remark that for any $R\leq \frac14$ and $r\leq \frac14$ we have,
$$
\mbox{dist}(B_{R|x_n|}(x_n),B_{r|z_n|}(z_n))\geq |x_n|\left(\frac{|z_n|}{|x_n|}(1-r)-(1+R)\right)=|x_n|(1-2r-R)>0,
$$
and then we have,
$$
16\pi>\lm_\ii\geq \liminf\limits_{n\to +\ii} \int\limits_{B_R(z_1)} \bar V_n e^{\dis \bar w_n}+
\liminf\limits_{n\to +\ii} \int\limits_{B_r(\bar z)} \bar V_{n,1} e^{\dis \bar v_{n}}=16\pi,
$$
which is a contradiction and proves \eqref{hypolsI-B}.
Thus, in view of \eqref{hypolsI-B.1} and \eqref{hypolsI-B} and since $\frac{\lm_\ii}{4\pi}\sg<1$, we can apply Lemma \ref{LS-quant} with
$$
d_n=\frac{|x_n|}{2},
$$
to deduce that,
\begin{align*}
 \lambda_n=8\pi+o(1),
\end{align*}
implying that \eqref{decay} is satisfied in this case as well.\\
At this point, in view of \eqref{hypolsI-B.1},
by the pointwise estimates in \cite{yy} we have that,
\begin{equation*}
\bar w_n(z)=\log\left(\dfrac{\bar \theta_n^{2(1+\al_{n,\sg})}}
{\left(1+\bar\gamma_n\bar \theta_n^{2(1+\al_{n,\sg})}|z-\frac{x_n}{|x_{n}|}|^2\right)^2}\right)+O(1),\quad z\in B_{R}(0),
\end{equation*}
where
$$
\bar \theta_n^{2(1+\al_{n,\sg})}=e^{\dis \bar w_n(\tfrac{x_n}{|x_{n}|})}=e^{\dis v_n(x_n)}|x_{n}|^{2(1+\al_{n,\sg})}=
\left(\tfrac{|x_{n}|}{\delta_n}\right)^{2(1+\al_{n,\sg})}
$$
and $\bar\gamma_n=\frac{\bar V_{n}(\tfrac{x_n}{|x_n|})}{8}$. Thus, we can argue as in CASE (I) above to deduce that,
\begin{equation}\label{profilew:H1}
 \bar w_n(z)=\log\left(\dfrac{\bar \theta_n^{2(1+\al_{n,\sg})}}
{\left(1+\bar\gamma_n\bar \theta_n^{2(1+\al_{n,\sg})}|z-\frac{x_n}{|x_n|}|^{\frac{M_n}{4\pi}}\right)^2}\right)+O(1),\quad z\in D^{(n)}.
\end{equation}
where
\[
 M_n=\int_{D^{(n)}}\bar V_n e^{\dis \bar w_n}=\lm_n+o(1).
\]

The proof is a step by step adaptation with minor changes of the one provided in section \ref{app:est-1} and we omit it here
to avoid repetitions. By arguing as in \eqref{vn:cn}, \eqref{eps:I-A}, \eqref{new:n}, we see that,
\begin{equation}\label{eps:II-A}
m_n= -2\log\left(\tfrac{|x_n|^{2\al_{n,\sg}}}{\delta_n^{1+\alpha_{n,\sg}}}\right) + O(1),
\end{equation}
and in particular that
$$
\frac{|x_n|^{2\al_{n,\sg}}}{\delta_n^{1+\alpha_{n,\sg}}}=
\left(\tfrac{|x_n|}{\delta_n}\right)^{1+\al_{n,\sg}}|x_n|^ {\alpha_{n,\sg}-1}=\left(\tfrac{|x_n|}{\delta_n}\right)^{1+\al_{n,\sg}}|x_n|^{2\sg-1+o(1)}\to +\ii,
$$
where we used $2\sg<1$. Remark that this shows as above that
\[
0\leq \lm_n-M_n= \int\limits_{\om\setminus B_r}H_n e^{\dis v_n(x_n)}\leq C_r\left(\bar \theta_n\right)^{-2(1+\bns)}|x_n|^{2(1-2\sg)+o(1)}.
\]
Therefore we deduce that \eqref{profilew:H1} holds true with $M_n$ replaced by $\lm_n$, in this case as
well, which in particular in terms of $v_n$ takes the form \eqref{profilev:H1}. This fact concludes the discussion of {CASE} (II).

\bigskip

At last we have the following,
\begin{itemize}
 \item {\bf CASE} (III): there exists a subsequence such that $\tfrac{\epsilon_n}{t_n}\leq C$ and $\tfrac{|x_n|}{\delta_n}\leq C$.
\end{itemize}
As in \eqref{epsilondeltan}, we have that
\[
 \frac{\epsilon_n}{\delta_n}\leq C.
\]
Recall that $D_n:=B_{\frac{r}{\delta_n}}(0)$ and observe that,
possibly along a subsequence, we can assume that \eqref{x0} and \eqref{epsilon_0} hold true.
Next we define $\widetilde w_n$ as in \eqref{phicaseIA} which satisfies,
\begin{align*}
\begin{cases}
 -\D \widetilde w_n=\widetilde V_n(y) e^{\dis \widetilde w_n} \quad\text{in}\quad D_n \\
\widetilde V_n(y)= \left({\frac{\epsilon_n^2}{\delta_n^2}+|y|^2}\right)^{\bns}
 K_n(\delta_ny)\to c_0\left(\epsilon_0^2+|y|^2\right)^{\bis}\;\mbox{in}\;C^t_{loc}(\R^2),\\
\int_{D_n}\widetilde V_n e^{\dis \widetilde w_n}=\lm_n+o(1)=\lm_\ii+o(1)\\
\widetilde w_n(y)\leq \widetilde w_n(\frac{x_n}{\delta_n})=0.
\end{cases}
\end{align*}

Clearly the concentration-compactness alternative of regular Liouville type equations (\cite{bm}) can be applied to $\widetilde w_n$
and we readily deduce that $\widetilde w_n$ is uniformly bounded in  $L^{\infty}_{\rm loc}(D_n)$.
Therefore, by standard elliptic estimates, possibly along a subsequence we have that $\widetilde w_n\to \widetilde w$ in $C^{2}_{loc}(\R^2)$,
where $\widetilde w$ is a solution of \eqref{profile:tildew}.
Recall that by assumption $\sg \lm_\infty< 4\pi$, that is $\bis=\frac{\lm_\ii}{4\pi}\sg< 1$, and then $\max\{8\pi,4\pi(1+\bis)\}=8\pi$. Therefore,
by Lemma \ref{massstrange} below, we have,
\begin{align}\label{tildebeta}
8\pi<\widetilde \beta \leq  8\pi(1+\bis).
\end{align}
However, as mentioned in the lines right before Lemma \ref{massstrange}, for fixed $\widetilde \beta$ in the range \eqref{tildebeta}, $\widetilde w$ is unique, radially symmetric and non degenerate.
Remark that, since $\lm_\ii\geq \widetilde \beta$, in view also of \eqref{mass:lminfty1}, we have,
\begin{align}\label{bt.first}
8\pi < \widetilde \beta\leq
 \graf{\lm_\ii\leq \frac{8\pi}{1-2\sg},\quad \sg\in (0,\frac14),\\$\,$\\ \lm_\ii<\frac{4\pi}{\sg},\quad \sg\in [\frac14,\frac12),}
\end{align}
whence in particular $\lm_\ii<16\pi$.
Therefore, we readily deduce that, for fixed $\sg\in(0,\frac12)$, there exist $\varepsilon_0>0$ and a sequence $R_n\to +\ii$ such that

\[
\|\tilde w_n - \tilde w\|_{C^{2}(B_{4R_n})}\to 0,
\]

\begin{equation}\label{masstilde0}
\lim\limits_{n\to +\ii}\int_{B_{4R_n}}\widetilde V_n e^{\dis \widetilde w_n} =\lim\limits_{n\to +\ii}\int_{B_{2R_n}}\widetilde V_n e^{\dis \widetilde w_n}= \tilde \beta\geq 8\pi+2\eps_0,
\end{equation}
and
\begin{equation*}
\int_{D_n\setminus B_{4 R_n}}\widetilde V_n e^{\dis \widetilde w_n}\leq \int_{D_n\setminus B_{2 R_n}}\widetilde V_n e^{\dis \widetilde w_n} \leq 8\pi-\varepsilon_0.
\end{equation*}

We want to show that there exist $C_{r}>0$ such that
\begin{equation}
\label{hypolsII-B}
 \sup_{B_{2r}\backslash B_{2R_n\dt_n}}\{v_n(z)+2(1+\bns)\log |z|\}\leq C_{r}.
\end{equation}
By contradiction, we assume the existence of a sequence $z_n$ such that
\[
 v_n(z_n)+2(1+\bns)\log |z_n|\to+\infty, \,\,\,\,\,\text{as}\,\,n\to+\infty,
\]
\[
 \mbox{and }\quad 2 R_n \dt_n\leq |z_n|\leq 2r.
\]

Then, let us define
\[
 \tilde v_n(z)= v_n(|z_n|z)+2(1+\bns)\log |z_n|,
\]
which satisfies
\begin{align*}
\begin{cases}
-\D  \tilde v_{n}= \tilde V_{n,1}(z)e^{\dis \tilde v_{n}} \qquad\text{in}\,\,B_{\frac{r}{|z_{n}|}}(0),\\
 \tilde V_{n,1}(z)=\left({\frac{\epsilon_n^2}{|z_{n}|^2}+|z|^2}\right)^{\bns}K_n(|z_n|z)
\to c_0|z|^{2\bis}\;\mbox{in}\;C^t_{loc}(\R^2),\\
\int_{B_{\frac{r}{|z_n|}}} \tilde V_{n,1} e^{\dis  \tilde v_{n}}= \lm_n +o(1) = \lm_\ii +o(1),\\
 \tilde v_{n}(\tfrac{z_n}{|z_n|})=v_n(z_n)+2(1+\bns)\log |z_n|\to+\infty,
\end{cases}
\end{align*}
where we used the fact that
\[
 \tfrac{\epsilon_n}{|z_n|}=\tfrac{\epsilon_n}{\dt_n}\tfrac{\dt_n}{|z_n|}\leq\tfrac{C}{2 R_n}\to 0.
\]
Possibly along a subsequence we can assume that $\tfrac{z_n}{|z_n|}\to \tilde z,\;|\tilde z|=1$ and then, as in Remark \ref{rem:reg},
the concentration-compactness theory (\cite{bm},\cite{yy}) for regular Liouville type equations applies to $\tilde v_{n}$ in $B_R(\tilde z)$ for any
$R\leq \frac12$ and in particular we have that,
\begin{equation*}
  \underset{\p B_R(\tilde z)}\max\tilde v_n-\underset{\p B_R(\tilde z)}\min\tilde v_n\leq C_R,\quad \forall \, R\leq\frac{1}{2},
\end{equation*}
whence $\tilde z$ is a simple blow up point (\cite{yy}), i.e.
$$
\int\limits_{B_R(\tilde z)} \tilde V_{n,1} e^{\tilde v_n}\to 8\pi,
$$
for any $R\leq \frac12$. Remark that for any $r\leq \frac14$ we have,
$$
\mbox{dist}(B_{R_n\dt_n}(x_n),B_{r|z_n|}(z_n))\geq R_n\dt_n
\left(\frac{|z_n|}{R_n |\dt_n|}(1-r)-(1+\frac{|x_n|}{R_n\dt_n})\right)\geq
$$
$$
R_n\dt_n(2\frac{3}{4}- 1+o(1))>0,
$$
and then we have,
$$
16\pi>\lm_\ii\geq \liminf\limits_{n\to +\ii} \int\limits_{B_{R_n\dt_n}(x_n)} \tilde V_{n} e^{\dis \tilde w_n}+
\liminf\limits_{n\to +\ii} \int\limits_{B_r(\tilde z)} \tilde V_{n,1} e^{\dis \tilde v_{n}}=16\pi,
$$
which is a contradiction and proves \eqref{hypolsII-B}. In view of \eqref{masstilde0} and  \eqref{hypolsII-B} and since $\frac{\lm_\ii}{4\pi}\sg<1$ we can apply Lemma \ref{LS-quant} with
$$
d_n=\frac{1}{2} R_n\dt_n,
$$
 to deduce that, by using also \eqref{tildebeta},
 \begin{equation}\label{mass-IIb}
 \lm_n=\widetilde \beta+o(1) \quad \mbox{and}\quad \lm_\ii= \widetilde \beta\in (8\pi,8\pi(1+\al_{\ii,\sg})].
 \end{equation}
 This fact immediately implies that we can refine \eqref{bt.first} as follows,
\begin{align*}
 \graf{8\pi < \widetilde \beta=\lm_\ii\leq \frac{8\pi}{1-2\sg},\quad \sg\in (0,\frac14),\\\,\\ 8\pi < \widetilde \beta=\lm_\ii<\frac{4\pi}{\sg},\quad \sg\in [\frac14,\frac12).}
\end{align*}
\bigskip

We claim that
\begin{equation}\label{profilewtilde:H1}
\widetilde w_n(y)=\widetilde U_n(y) + O(1),\quad y\in D_n,
\end{equation}
where, for any large $R\geq 1$ we have,
\begin{equation}\label{profilevtilde:H11}
\widetilde U_n(y;\widetilde M_n)=\graf{\widetilde w(y)+O(1),\quad |y|\leq R, \\ -\frac{\widetilde M_n}{2\pi}\log(|y|) + O(1),\quad R\leq  |y|\leq r\delta_n^{-1},}
\end{equation}
where
\[
 \widetilde M_n=\int_{D_{n}}\widetilde V_n e^{\dis \widetilde w_n}=\lm_n+o(1),
\]
and $\widetilde w$ is the unique solution of \eqref{profile:tildew}.
The argument is a somehow simpler version of what
done in the Appendix (see Section \ref{app:est-1}) about CASE (I) and much in the spirit of similar
well known estimates obtained in \cite{bclt} and \cite{tar-sd}, where one just replaces the radial profiles in
\cite{pt} with the profile \eqref{profilevtilde:H11}. Therefore, we omit the details of this proof here to avoid repetitions.\\
At this point we wish to show that \eqref{profilevtilde:H11} still holds true with  $\widetilde M_n$ replaced by $\lm_n$.
First of all, in terms of $v_n$, \eqref{profilewtilde:H1} takes the form,
\begin{equation}\label{profilevtilde:H1}
 v_n(x)=v_n(x_n)+ \widetilde U_n(\delta_n^{-1}x;\widetilde M_n)+O(1), \quad x\in B_{r}(0),
\end{equation}
and implies, in view of \eqref{mass-IIb} and \eqref{profilevtilde:H11}, that
$$
v_n(x)=(1-\frac{\widetilde M_n}{4\pi(1+\bns)})v_n(x_n)+O(1),\quad 2 |x|=r.
$$

On the other side from \eqref{profilewn-n} we have that,
\beq\label{vn:cn-IIb}
v_n(x)=O(1)+m_n,\quad 2 |x|=r,
\eeq
implying in turn by \eqref{masstilde0} that,
\beq\label{eps:II-B}
m_n=O(1)-\left(\frac{\widetilde \beta+o(1)}{4\pi(1+\bns)}-1\right)v_n(x_n)\to +\infty.
\eeq
Thus, again because of \eqref{vn:cn-IIb},  \eqref{masstilde0} and $\bis< 1$, we see that
\[
0\leq \lm_n-\widetilde M_n= \int\limits_{\om\setminus B_r}H_n e^{\dis v_n(x)}\leq C_r\dt_n^{ \frac{\tilde \beta}{2\pi}  -2(1+\bns)}\leq C_r\dt_n^{ 2\eps_0},
\]
implying that, for any $1\leq R\leq |y| \leq r\dt_n^{-1}$, we have that
$$
|\lm_n-\widetilde M_n||\log(|y|)|\leq C\dt_n^{ 2\eps_0}\log(\dt_n^{-1})=o(1).
$$
As a consequence, it is readily seen that \eqref{profilevtilde:H11} still holds true with $\widetilde M_n$ replaced by $\lm_n$,
which in particular in terms of $v_n$ takes the form
\eqref{profilevtilde:H1-IIb}. This fact concludes the discussion of CASE (III).
\finedim

\bigskip
\bigskip

\section{Free energy and existence of solutions: the Canonical Variational Principle}\label{sec4}
In this section we define the mean field  theory (\cite{clmp1},\cite{clmp2}) of the vortex model with singular sources.

\begin{definition}\label{piecewisec2} An open and bounded domain is said to be piecewise $C^2$ if it is of class $C^2$ but for a finite number of points
$\{Q_1,\cdots,Q_\ell\}\subset \pa\om$ such that:\\
(i) at each $Q_j$ is well defined the tangent cone, whose inner angle $\theta_j$ satisfies $0<\theta_j<2\pi,\,\theta_j\neq \pi$;\\
 (ii) at each $Q_j$ there exist $\eps_j>0$ and a univalent conformal map
$f_j\in C^{1}(B_{\eps_j}(Q_j)\cap\overline{\om};\mathbb{C})$, such that $f(\pa\om\cap {B}_{\eps_j}(Q_j))$ is the support of a curve of class $C^2$.
\end{definition}

Let $\O\subset\R^2$ be a bounded piecewise $C^2$ domain, $0\in \O$ and assume the existence of a fixed vortex whose strength is
$\sigma\in\R$ and whose center coincides with $0$. Let us define $\mathcal{P}(\O)$, $G[\rho]$, $\mathcal{S}(\rho)$, $\mathcal{E}(\rho)$ and $\mathcal{E}_\sg(\rho)$ as in the introduction.\\
By classical elliptic estimates (\cite{stampacchia}), the stream function relative to $\rho \in\mathcal{P}(\O)$, that is $\psi=G[\rho]$,
is in $W^{1,p}_0(\O)$ for any $p\in[1,2)$, and it is the unique solution of,
\begin{align*}
 \begin{cases}
  -\D \psi=\rho &\text{in}\,\, \O, \\
  \psi=0 &\text{on}\,\,\p\O.
 \end{cases}
\end{align*}
Moreover, by using in particular the fact that $\int_{\O}\rho\log \rho<+\infty$, it is well known (\cite{stampacchia2}) that
$\mathcal{E}_\sg(\rho)$ is well defined, $\psi\in L^{\infty}(\O)\cap W^{1,2}_0(\O)$ and also that
\begin{equation*}
 \mathcal{E}(\rho)=\tfrac{1}{2}\int_{\O}\rho G[\rho]\,dx=\tfrac{1}{2}\int_{\O}(-\D\psi)\psi\,dx=\tfrac{1}{2}\int_{\O}|\nabla\psi|^2\,dx,
\end{equation*}
which means that (one half of) the Dirichlet energy of $\psi$ is the energy. We define the (scaled) Free Energy functional,
\begin{align}
 \nonumber
 \mathcal{F}_{\lambda}(\rho)&= -\mathcal{S}(\rho)-\lambda(\mathcal{E}(\rho)-\mathcal{E}_\sg(\rho)) \\
 \nonumber
 &=-\mathcal{S}(\rho)-\tfrac{\lambda}{2}\int_{\O}\rho\, G[\rho]\,dx+\sigma\lambda\int_{\O}\rho \,G_\O(x,0)\,dx\\\label{free.en}
 &= \int_{\O}\rho\log(\rho)\,dx-\tfrac{\lambda}{2}\int_{\O}\rho\, G[\rho]\,dx-
 \tfrac{\sigma\lambda}{2\pi}\int_{\O}\rho\log|x|\,dx+\sigma\lambda\int_{\O}\rho \,R_\O(x,0)\,dx.
\end{align}

The energy of the uniform distribution $\rho_0=1$, a.e. in $\om$, is denoted
\begin{equation}\label{e0sec4}
E_0=E_0(\om)=\mathcal{E}(\rho_0)-\mathcal{E}_\sg(\rho_0)=\frac12\int\limits_{\om}|\nabla \psi_0|^2-\sg \int\limits_{\om}G_\om(x,0),
\end{equation}
where $\psi_0$ is the unique solution of
\begin{align*}
\begin{cases}
 -\D \psi_{0}=1\,\,\,\,\,\,&\text{in}\,\,\O,\\
\psi_{0}=0 &\text{in}\,\,\p\O,
\end{cases}
\end{align*}

\bigskip

A general useful tool in this context will be the classical Young inequality:
\begin{equation}\label{Young}
ab\leq e^a-1+(b\log(b)-b+1)\chi_{b\geq 1},\; a\geq 0, b\geq 0.
\end{equation}

\bigskip

The existence of the minimum of $\mathcal{F}_{\lambda}$ can be proved by various well known arguments, see for example \cite{clmp1} or \cite{csw}.
We provide a proof of this fact which emphasizes the connection between the free energy functional and the classical functional,
\begin{equation*}
 \mathcal{J}_{\lambda}(\psi)=\tfrac{\lambda}{2}\int_{\O}|\nabla \psi|^2\,dx-\log\left(\int_\O H_{\lm}e^{ \lambda \psi}\right),
\end{equation*}
where $\psi\in W^{1,2}_0(\O)$ and $H_\lm$ takes the form \eqref{Hlm:intro}, that is,
\begin{equation*}
 H_{\lm}(x)=|x|^{\frac{\lambda}{2\pi}\sigma}e^{-\sigma\lm R_\O(x,0)}=e^{-\sigma\lm G_\O(x,0)},
\end{equation*}
see also \cite{wolansky}. As far as \eqref{lsg} is satisfied, it is easy to see, by the H\"{o}lder and Moser-Trudinger (see \cite{moser} and \cite{as})
inequalities, that $\int_\O H_{\lm}e^{ \lambda \psi}$ is positive and $\mathcal{J}_{\lambda}(\psi)$ is well defined on $W^{1,2}_0(\O)$.  Let us define,
\[
 j(\lambda)=\underset{\psi\in W_0^{1,2}(\O)\backslash\{0\}}\inf \mathcal{J}_{\lambda}(\psi),
\]
we have the following,
\begin{lemma}
\label{lemmafunctionalequality}\hfill \\
Let $\sg\in\R$ and $\lm\in [0,\lm_\sg]$. Then we have:\\
 \begin{itemize}
  \item[(i)]
  For any $\rho\in\mathcal{P}(\O)$, let $\psi=G[\rho]$. Then $\psi\in W_0^{1,2}(\O)$ and,
  \begin{equation*}
\mathcal{F}_{\lambda}(\rho)\geq  \mathcal{J}_{\lambda}(\psi),
\end{equation*}
where the equality holds if and only if,
\begin{equation*}
 \rho=\frac{H_{\lm}e^{\lambda \psi}}{\int_\O H_{\lm}e^{\lambda \psi}}\,\,\text{a.e in}\,\,\O.
\end{equation*}
\item[(ii)]
For any $\psi\in W_0^{1,2}(\O)$, let $\rho=\tfrac{H_{\lm}e^{\lambda \psi}}{\int_\O H_{\lm}e^{\lambda \psi}}$.
Then $\rho\in\mathcal{P}(\O)$ and,
  \begin{equation*}
\mathcal{F}_{\lambda}(\rho)\leq  \mathcal{J}_{\lambda}(\psi),
\end{equation*}
where the equality holds if and only if,
\begin{equation*}
 \psi=G[\rho]\,\,\text{a.e in}\,\,\O.
 \end{equation*}
 \end{itemize}
In particular, if $\lambda\leq \lm_\sg$, then we have that,
\[
 f(\lambda)=j(\lambda)>-\infty,
\]
and if $\lambda<\lm_\sg$ then $\mathcal{J}_{\lambda}$ admits at least one minimizer $\psi_{\lm}$,
and any such minimizer is a solution of the mean field equation \eqref{mf:lm.intro},
where $\rho_{\lambda}$ is a minimizer of $\mathcal{F}_{\lambda}(\rho)$.
\end{lemma}

\begin{proof}
Let $\rho\in\mathcal{P}(\O)$ and $\psi=G[\rho]$. We denote by $\mathcal L\log\mathcal L(\O)$ the Orlicz space of those measurable functions
which satisfy $\int_{\O}|f|\log(e+ |f|)<+\infty$.
It is well known (\cite{kjf}, \cite{RR}) that $W_0^{1,2}(\O)\hookrightarrow(\mathcal L\log\mathcal L(\O))'$ is continuous and consequently we see that
\begin{equation}\label{en:free}
 \int_\O \rho \psi=\int_\O \rho\, G[\rho]=\int_\O|\nabla \psi|^2,
\end{equation}
and in particular that $\psi\in W_0^{1,2}(\O)$. Actually it is also well known that $\psi$ is continuous, see \cite{Fer}.\\
By the H\"{o}lder and Moser-Trudinger inequalities (\cite{moser}), it is well defined
\[
 \rho_{\psi}:=\frac{H_{\lm}e^{\lambda \psi}}{\int_\O H_{\lm}e^{\lambda \psi}}
\]
and $\rho_{\psi}\in\mathcal{P}(\O)\cap L^{p}(\O)$, for some $p>1$ depending on $\sg$. Here we are using  the fact that
$\lm_\sg=\frac{8\pi}{1+2|\sg|}<\frac{4\pi}{|\sg|}$, as far as $\sg<0$, see Remark \ref{rem:conicsing}. Thus we have that,
\begin{align*}
 \mathcal{J}_{\lambda}(\psi)&=\tfrac{\lambda}{2}\int_{\O}|\nabla \psi|^2\,dx-\log\left(\int_\O H_{\lm}e^{ \lambda \psi}\right) \\
 &=\lambda\int_{\O}|\nabla \psi|^2\,dx-\tfrac{\lambda}{2}\int_{\O}|\nabla \psi|^2\,dx-\int_\O \rho\log\left(\int_\O H_{\lm}e^{ \lambda \psi}\right)\,dx\\
 &=\int_\O \rho\left[\lambda \psi-\log\left(\int_\O H_{\lm}e^{ \lambda \psi}\right)\right]-\tfrac{\lambda}{2}\int_{\O}|\nabla \psi|^2\,dx \\
 &=\int_\O \rho\log (\rho_{\psi})+\sigma\lambda\int_\O\rho G_\O(x,0)\,dx-
 \tfrac{\lambda}{2}\int_{\O}|\nabla \psi|^2\,dx.
\end{align*}
Hence, in view of \eqref{free.en} and \eqref{en:free}, we deduce that,
\begin{equation*}
 \mathcal{F}_{\lambda}(\rho)-\mathcal{J}_{\lambda}(\psi)=\int_{\O}\rho\log\left(\tfrac{\rho}{\rho_{\psi}}\right)\,dx.
\end{equation*}
On the other hand, by applying the Jensen inequality to the strictly convex function $k(t)=t\log(t)$ with $t=\tfrac{\rho}{\rho_{\psi}}$, we have that,
\[
 \int_{\O}\rho\log\left(\tfrac{\rho}{\rho_{\psi}}\right)\,dx=\int_{\O}\rho_{\psi}\tfrac{\rho}{\rho_{\psi}}
 \log\left(\tfrac{\rho}{\rho_{\psi}}\right)\,dx\geq \left(\int_{\O}\rho_{\psi}\tfrac{\rho}{\rho_{\psi}}\right)
 \log\left(\int_{\O}\rho_{\psi}\tfrac{\rho}{\rho_{\psi}}\right)=0
\]
and the equality holds if and only if $\tfrac{\rho}{\rho_{\psi}}$ is constant a.e in $\O$.
Actually, $\rho$ and $\rho_{\psi}$ belong to $\mathcal{P}(\O)$, so this is equivalent to say that $\rho=\rho_{\psi}$ a.e. in $\O$, 
which proves (i).\\

 Next let us fix $\psi\in W_0^{1,2}(\O)$ and define
 \[
  \rho_{\psi}:=\frac{H_{\lm}e^{\lambda \psi}}{\int_\O H_{\lm}e^{\lambda \psi}}.
 \]
As discussed in (i),  $\rho_{\psi}\in\mathcal{P}(\O)\cap L^p(\O)$, for some $p>1$ and we define
let $\phi=G[\rho_{\psi}]\in  W_0^{1,2}(\O)\cap W_0^{2,p}(\O)$. Therefore, we have that
\[
 \int_{\O}\rho_{\psi}\phi=-\int_\O \phi\D \phi=\int_\O|\nabla \phi|^2,
\]
and
\[
 \int_{\O}\rho_{\psi}\psi=-\int_\O \psi\D \phi=\int_\O\nabla \psi\cdot\nabla \phi.
\]
Thus, it is readily seen that,
\[
 \mathcal{F}_{\lambda}(\rho_{\psi})-\mathcal{J}_{\lambda}(\psi)=-\tfrac{\lambda}{2}\int_{\O}|\nabla(\psi-\phi)|^2\leq 0
\]
where the equality holds if and only if $\psi-\phi\in W_0^{1,2}(\O)$ is constant, which means $\psi=\phi(=G[\rho_{\psi}])$ a.e. in $\O$.
This fact concludes the proof of (ii).\\

At this point, if $\sg>0$ and $\lambda\in(0,8\pi]$, by using the Moser-Trudinger inequality (\cite{moser})
we have that $j(\lambda)>-\infty$. In particular, if $\lambda\in(0,8\pi)$,
$\mathcal{J}_{\lambda}$ is also coercive.
If $\sg>0$, putting $\al_\sg=\frac{\sg\lm}{4\pi}$, as far as \eqref{lsg} is satisfied (i.e. $|\al_\sg|<1$) by the singular Moser-Trudinger
inequality (\cite{as}) we find that,
\begin{align*}
\mathcal{J}_{\lambda}(\psi)=&\tfrac{\lambda}{2}\int_{\O}|\nabla \psi|^2\,dx-\log\left(\int_\O H_{\lm}e^{ \lambda \psi}\right)\geq \\
=&\tfrac{\lambda}{2}\int_{\O}|\nabla \psi|^2\;dx-C(\om,\al_\sg)-\frac{\lm^2}{16\pi(1+\al_{\sg})}\int_{\O}|\nabla \psi|^2\\
=&-C(\om,\al_\sg)+\tfrac{\lambda}{2}\int_{\O}|\nabla \psi|^2\;dx\left(\frac{8\pi-(2|\sg|+1)\lm}{8\pi+2\sg\lm}\right)
\end{align*}
Hence, if $\lambda\in(0,\frac{8\pi}{1+2|\sg|}]$ then $j(\lambda)>-\infty$ and, if $\lambda\in(0,\frac{8\pi}{1+2|\sg|})$, it is also coercive.\\
Therefore we deduce that, as far as \eqref{lsg} is satisfied and $\lm<\lm_\sg$, then at least one minimizer exists by the direct method.
Remark however that if $\sg<0$ then $\frac{8\pi}{1+2|\sg|}<\frac{4\pi}{|\sg|}$, whence \eqref{lsg} is granted and we can let
$\psi_{\lambda}$ be any such minimizer which in particular is a solution of \eqref{mf:lm.intro}. However at this point
we see  that to define $\rho_{\lambda}$ as in \eqref{mf:lm.intro} is exactly the same as to define it as we did in (ii) above.
Consequently it is not difficult to see from (ii) and (i) that
$$
f(\lambda)=\mathcal{F}_{\lambda}(\rho_{\lambda})=\mathcal{J}_{\lambda}(\psi_{\lambda})=j(\lambda).
$$
\end{proof}

In view of the Lemma we can prove Theorem \ref{thm:cvpintro}.\\
{\it The proof of Theorem \ref{thm:cvpintro}}\\
For $\lm=0$, as a consequence of the Jensen inequality, it is readily seen that $\rho_0=1$ a.e. in $\om$ (recall $|\om|=1$)
is the unique minimizer of
$\mathcal{F}_{0}$ and in particular $f(0)=0$. Therefore we assume w.l.o.g. that $\lm\in (0,\lm_\sg)$.\\
In view of Lemma \ref{lemmafunctionalequality}, for any
solution $\rho_{\lm}$ of the \eqref{cvp:intro} we have that $\psi_{\lm}=G[\rho_\lm]$ is a solution of \eqref{mf:lm.intro}.
At this point the claim about uniqueness follows from the uniqueness results in
\cite{bl1} for $\sg>0$ and from those in \cite{BJL} and \cite{BCJL} for $\sg<0$.
Remark that those uniqueness results are obtained for \eqref{mf:lm.intro} but for the fact that the singular part of the weight function takes
the form $|x|^{2\al}$, where the order of the exponent $\alpha$ is kept fixed, while here we have $\al=\al_\sg=\frac{\lm}{4\pi}\sg$.
However it is not difficult to see that for $\sg>0$ the uniqueness threshold is still $8\pi$,
while for $\sg<0$ the relevant information is just the inequality
$\lm=\lm\int\limits_{\om}\rho_{\lm}\leq 8\pi(1+\frac{\lm}{4\pi}\sg)$, which is granted in this case once more by the fact that $\lm<\lm_\sg=\frac{8\pi}{1-2\sg}$. \finedim

\bigskip

\section{Entropy and existence of solutions: the Microcanonical Variational Principle, domains of  type I/II.}
Let $\O\subset\R^2$ be a bounded piecewise $C^2$ domain, we take the same notations adopted in the introduction and in  section \ref{sec4} about
$\mathcal{P}(\O)$, $\mathcal{P}_E(\O)$
$G[\rho](x)$, the entropy
$\mathcal{S}(\rho)$, the energy of the density $\rho$, $\mathcal{E}(\rho)$,
and of the fixed vortex, $\mathcal{E}_\sg(\rho)$. As already mentioned in the introduction, 
we consider only those states whose energy is larger than the energy $E_0=E_0(\om)$ of the uniform distribution $\rho_0=1$, a.e. in $\om$, see 
\eqref{e0sec4}.\\
Let us recall the regularized version of the energy
\begin{equation}\label{en:reg}
\mathcal{E}_{\sg,n}(\rho)=\sg\ino \rho G_n=-\frac{\sg}{2\pi}\ino \rho \log(h_n)+\sg\ino \rho R_{n}(x),
\end{equation}
where $G_n$ has been defined in \eqref{def:Gn},
and set
$$
E_{0,n}:=\mathcal{E}(\rho_0)-\mathcal{E}_{\sg,n}(\rho_0).
$$
We first solve the following regularized (MVP),
\begin{align}\label{mvp:n}
 S_n(E)=\underset{\rho\in\mathcal{P}^{(n)}_E(\O)}\sup \mathcal{\mathcal{S}(\rho)},
\end{align}
\begin{equation*}
 \mathcal{P}^{(n)}_E(\O)=\left\{\rho\in \mathcal{P}(\O)|\, \mathcal{E}(\rho)-\mathcal{E}_{\sg,n}(\rho)=E\right\}.
\end{equation*}

Here $H_n$ takes the form of the weight defined in \eqref{Hlmn.intro}.
\begin{theorem}\label{thm:mvpn} For any $E\in [E_{0,n},+\ii)$ there exists at least one density $\rho$ which maximizes $\mathcal{S}$ at fixed energy
$$
\mathcal{E}(\rho)-\mathcal{E}_{\sg,n}(\rho)=E,
$$
and there exists $\lm_n=\lm_n(E)$ such that $\rho=\rho_{\lm_n}$ is a solution of,
\begin{equation}\label{mvp}
\rho_{\lm_n}=\dfrac{H_{n} e^{\lm_n G[\rho_{\lm_n}]}}{\int\limits_{\om} H_{n} e^{\lm_n G[\rho_{\lm_n}]}}\quad \mbox{a.e. in }\om.
\end{equation}
In particular $\psi_{\lm_n}=G[\rho_{\lm_n}]$ is a solution of the regularized Mean Field Equation \eqref{mf:n.intro}.
\end{theorem}
\begin{proof}
By applying the Jensen inequality to the strictly concave function $-t\log(t)$ with $t=\int\limits_{\om}\rho$ we have
$$
\mathcal{S}(\rho)=-\int\limits_{\om}\rho\log(\rho)\leq -\left(\int\limits_{\om}\rho\right)\log\left(\int\limits_{\om}\rho\right)=0,
$$
where the equality holds if and only if $\rho=\rho_0=1$ a.e.. Therefore there exists only one entropy maximizer at energy $E_0$, which clearly
satisfies \eqref{mvp} with $\lm=0$. Obviously $\psi=G[\rho_0]\equiv \psi_0$ solves \eqref{mf:lm.intro}.

\medskip

Let $E>E_0$, it is not difficult to prove that
$\mathcal{P}_E(\O)\neq \emptyset$ and since $-t\log(t)\leq e^{-1}$ we have $S_n(E)<+\ii$. Let us denote by $\rho_k$ any maximizing sequence which
in particular satisfies
\begin{equation}\label{llogl}
\int\limits_{\om}\rho_k\log(\rho_k)\leq C.
\end{equation}
Therefore we have
$$
\sup\limits_{k\in \N}\int\limits_{\rho_k\geq m} \rho_k\leq\frac{1}{\log(m)} \ino \rho_k|\log(\rho_k)| \leq \frac{C}{\log(m)},\to 0,\;m\to +\ii,
$$
implying that $\{\rho_k\}$ is equi-integrable. Since $\int\limits_{\om}\rho_k=1$, by the Dunford-Pettis criterion we can pass to a subsequence,
still denoted $\rho_k$, which converges in the weak-$L^1(\om)$ topology to some $\rho\in L^1(\om)$ with $\int\limits_{\om}\rho=1$.
Since the entropy is concave and upper semicontinuous w.r.t. the strong ${L}^1(\om)$-topology, then it is upper semicontinuous w.r.t. the weak
${L}^1(\om)$-topology implying that $S(E)=\mathcal{S}(\rho)$. Therefore in particular
\begin{equation*}
\int\limits_{\om}\rho \log(\rho)<+\ii,
\end{equation*}
and we claim that $\mathcal{E}(\rho)-\mathcal{E}_{\sg,n}(\rho)=E$. In fact, since $\rho_k \to \rho$ in the weak-$L^1(\om)$ topology,
then obviously in view of \eqref{en:reg} we have $\mathcal{E}_{\sg,n}(\rho_k)\to \mathcal{E}_{\sg,n}(\rho)$. By using \eqref{llogl}, the convergence
$\mathcal{E}(\rho_k)\to \mathcal{E}(\rho)$ can be proved as in \cite{clmp2}. \\
We are left to prove that $\rho=\rho_{\lm_n}$ solves \eqref{mvp} for some $\lm_n=\lm_n(E)\in \R$. By arguing at a formal level,
we would infer by the Lagrange multiplier rule that $\rho$ is a free critical point of the Lagrangian,
$$
\mathcal{S}(\rh)+\lm_n (\mathcal{E}(\rho)-\mathcal{E}_{\sg,n}(\rho))-\tau \ino \rho.
$$
However, since $\mathcal{P}^{(n)}_E(\O)$ is a subset of the cone of non-negative densities $\rho\geq 0$,
then to make the argument rigorous we have to prove first that $\rho$ stays far from the boundary of the cone, that is,
$\rho>0$ a.e. in $\om$. This could be done in principle by a subtle optimization trick as in \cite{clmp2}. We need another argument here
which is more in the spirit of the calculus of variations, showing that if $\rho=0$ on a set of positive measure, then
there exists some curve of densities in $\mathcal{P}^{(n)}_E(\O)$ along which the entropy is strictly increasing. Besides the independent interest,
we will need this argument again in the proof of Theorem \ref{thm:mvp} below, where in a similar situation we will not have the chance to
adopt the optimization trick as in \cite{clmp2}.
\begin{proposition}\label{rho:0}
Let $\rho$ be a maximizer of $\mathcal{S}$ at fixed energy $\mathcal{E}(\rho)-\mathcal{E}_{\sg,n}\sg(\rho)=E$. Then $\rho>0$ a.e. in $\om$.
\end{proposition}
\begin{proof}
We argue by contradiction and suppose that there exists $\om_0\subset \om$ such that $|\om_0|\in (0,|\om|)\equiv (0,1)$ and $\rho=0$ a.e. on $\om_0$.
Let $m>0$ be any positive number such that $|\{\rho > m\}|\neq 0$ and set $\om=\om_1\cup\om_1^c$, where $\om_1=\{\rho \leq  m\}$. Clearly
$\om_0\subseteq \om_1$. Suppose that
we could find a smooth map $f_\eps:[0,1)\to L^{\ii}(\om)$ satisfying the following properties:
$$
f_0=0,\quad f_\eps=f_{\eps,0}+f_{\eps,1},\quad f_{\eps,0} \equiv \hspace{-.35cm}\setminus\;0,\quad f_{\eps,0}(x)\geq 0\;\;\mbox{a.e. in }\; \om_0,
$$
$$
\mbox{\rm supp}(f_{\eps,0})\subseteq \om_0,
\;\mbox{\rm supp}(f_{\eps,1})\subseteq \om_1^c,
$$
where
$$
\ino (\rho+f_{\eps})=1,\;\mathcal{E}(\rho+f_{\eps})-\mathcal{E}_{\sg,n}(\rho+f_{\eps})=E,\quad \fo\; \eps\in [0,1).
$$

In particular we assume that we could find $f_\eps$ such that,
$$
f_{\eps,0}=\eps g_0+\mbox{\rm o}_x(\eps), \quad f_{\eps,1}=\eps g_1+\mbox{\rm o}_x(\eps),
\quad g_i \not\equiv 0,\;i=1,2, \quad \inf\limits_{\om_0} g_0\geq 1,
$$
where $\mbox{\rm o}_x(1)$ and $\mbox{\rm o}_x(\eps)$ will denote any function satisfying,
$$
\lim\limits_{\eps\to 0^+}\sup\limits_{\om}{|\mbox{\rm o}_x(1)|} =0,\quad\mbox{and}\quad
\lim\limits_{\eps\to 0^+}\sup\limits_{\om}\frac{|\mbox{\rm o}_x(\eps)|}{\eps} =0.
$$

If any such $f_\eps$ exists, then the proof is concluded observing that
$\rho+f_\eps$ is an admissible variation of $\rho$ so that we have
$
\mathcal{S}(\rho+f_{\eps})\leq \mathcal{S}(\rho),
$
which reads as follows,
$$
-\int\limits_{\om_0}f_{\eps,0}\log (f_{\eps,0})-\int\limits_{\om_{1}\setminus \om_0}\rho\log (\rho)-
\int\limits_{\om_{1}^c}(\rho+f_{\eps,1})\log (\rho+f_{\eps,1})\leq
$$
$$
-\int\limits_{\om_{1}\setminus \om_0} \rho \log \rho -
\int\limits_{\om_{1}^c}\rho \log \rho,\;\fo\,\eps\in [0,1),
$$
that is,
$$
-\int\limits_{\om_0}	\eps (g_0+\mbox{\rm o}_x(1) )\log (\eps (g_0+\mbox{\rm o}_x(1) ) )
$$
$$
-\int\limits_{\om_{1}^c}(\rho+\eps (g_{1}+\mbox{\rm o}_x(1) ))\log
(\rho+\eps (g_{1}+\mbox{\rm o}_x(1) ))+\int\limits_{\om_{1}^c}\rho \log \rho\leq 0.
$$
Since $\rho\geq m $ in $\om_{1}^c$, then we have,
$$
-\int\limits_{\om_{1}^c}(\rho+\eps (g_{1}+\mbox{\rm o}_x(1) ))\log (\rho+\eps (g_{1}+\mbox{\rm o}_x(1) ))+
\int\limits_{\om_{1}^c} \rho \log \rho=
$$
$$
\eps\left(\int\limits_{\om_{1}^c}(-1-\log(\rho))g_1\right) + \mbox{\rm o}(\eps),
$$

and so we find that,
$$
-\int\limits_{\om_0}(g_0+\mbox{\rm o}_x(1) )\log (\eps (g_0+\mbox{\rm o}_x(1) ) ) +
\int\limits_{\om_{1}^c}(-1-\log(\rho))f_1 \leq  \mbox{\rm o}(1),
$$
which is the desired contradiction, since we also have,
$$
-\int\limits_{\om_0}(g_0+\mbox{\rm o}_x(1) )\log (\eps (g_0+\mbox{\rm o}_x(1) ) )=
-\left(\;\int\limits_{\om_0}(g_0+\mbox{\rm o}_x(1) )\right)\log(\eps)
$$
$$
+\mbox{\rm o}(1) -\int\limits_{\om_0}g_0\log (g_0) \to +\ii,\mbox{ as }\eps\to 0.
$$

Therefore we are left to show that a variation $f_\eps$ with the required properties exists.
In particular we seek,
$$
f_\eps=\eps (g_0 + g_{1}) +\mbox{o}_x(\eps),\mbox{ as }\eps\to 0,
$$
where $g_0\geq 1$ and $g_{1}$ are $L^{\ii}(\om)$ functions supported in $\om_0,\om_1$, respectively. According to \eqref{def:Gn} let us set
$G_n(x)=-\frac{1}{2\pi}\log(h_n)+R_{n}(x)$. We claim that it is enough to prove that we can choose
$g_0,g_1$ in such a way that the following holds true:
\begin{equation}\label{eq:g}
\ino (g_0+g_1)=0,\quad \ino (G[\rho]-\sg G_n)(g_0+g_1)=0.
\end{equation}

Indeed, once we have $g_0,g_1$ satisfying \eqref{eq:g}, then $f_\eps$ is found as the unique solution of the Cauchy problem,
$$
\graf{\frac{d f_t}{dt}=\phi-a_1(t) [G[\rho+f_t]-\sg G_n], \\ f_0=0,}
$$
where
$$
\phi=g_0+g_1,\quad a_1(t)=\frac{\ino \phi [G[\rho+f_t]-\sg G_n]}{\ino [G[\rho+f_t]-\sg G_n]^2},\quad [\phi]=\phi -\ino \phi.
$$
In fact, local existence and uniqueness of the Cauchy problem follows by standard ODE techniques in Banach spaces, while by a straightforward
evaluation one can see that
$$
\frac{d}{dt} \ino(\rho+f_t)=0\quad \mbox{and} \quad \frac{d}{dt}\left(\mathcal{E}(\rho+f_t)-\mathcal{E}_{\sg,n}(\rho+f_t)\right)=0,
$$
that is, the velocity field is defined in such a way that $f_t\in \mathcal{P}_E$ for any $t$.\\
To simplify the remaining part of the argument let us set,
$$
\ov{G}=G[\rho]-\sg G_n(x).
$$
If there exists $g_0\in L^{\ii}(\om_0)$, $g_0\geq 1$ and $g_2\in L^{\ii}(\om_1)$, $g_2\equiv \hspace{-.35cm}\setminus\;0$ such that,
$$
\frac{\int_{\om_1}\ov{G} g_2}{\int\limits_{\om_1} g_2}=\frac{\int_{\om_0}\ov{G} g_0}{\int\limits_{\om_0} g_0},
$$

then, setting $g_1=-\frac{\int_{\om_0} g_0}{\int\limits_{\om_1} g_2}g_2$, the pair $\{g_0,g_1\}$ solves \eqref{eq:g}. Otherwise for any
$g_0\in L^{\ii}(\om_0)$, $g_0\geq 1$ and $g_2\in L^{\ii}(\om_1)$, $g_2\equiv \hspace{-.35cm}\setminus\;0$ we have,
$$
\frac{\int\limits_{\om_1}\ov{G} g_2}{\int\limits_{\om_1} g_2}\neq \frac{\int\limits_{\om_0}\ov{G} g_0}{\int\limits_{\om_0} g_0}.
$$
In this case, for any such $g_2$, we let $g_3$ be any other non trivial function in $L^{\ii}(\om_1)$ linearly independent from $g_2$.
Then we set $g_1=-\frac{1}{x}g_3+\frac{y}{x}g_2$ and we see that  \eqref{eq:g} is equivalent to the system,

$$
\graf{\left(\int\limits_{\om_0}g_0\right)x+\left(\int\limits_{\om_1}g_2\right)y=\int\limits_{\om_1}g_3,\\
\left(\int\limits_{\om_0}\ov{G} g_0\right)x+\left(\int\limits_{\om_1}\ov{G} g_2\right)y=\int\limits_{\om_1}\ov{G} g_3.}
$$
Clearly, in this case, the determinant of the system is not zero, whence $(x,y)$ are uniquely determined. In particular $x\neq 0$ since otherwise $g_3$ would be proportional
to $g_2$.
\end{proof}

\bigskip

As mentioned above, in view of Proposition \ref{rho:0}, by the Lagrange multiplier rule in Banach spaces we find that $\rho=\rho^{(n)}\in \mathcal{P}^{(n)}_E(\om)$
satisfies
$$
-\log(\rho)-\lm G[\rho]+\sg\lm G_n -\tau=0,\mbox{ that is }\rho=e^{-\tau}e^{\lm G[\rho]-\sg\lm G_n},\mbox{ a.e. in }\om.
$$
We first use the Lagrange multiplier $\tau$ to satisfy the mass constraint,
$$
\ino e^{-\tau}e^{\lm G[\rho] -\sg\lm G_n}=1,\;\iff\; e^{-\tau}=\dfrac{1}{\ino e^{\lm G[\rho] -\sg\lm G_n}},
$$
whence
$$
\rho^{(n)}=\rho^{(n)}_{\lm}=\dfrac{H_{\lm,n} e^{\lm G[\rho^{(n)}_{\lm}]}}{\int\limits_{\om} H_{\lm,n} e^{\lm G[\rho^{(n)}_{\lm}]}}\quad \mbox{a.e. in }\om,
$$
where $H_{\lm,n}=e^{-\sg\lm G_n}$. Since $\rho^{(n)}\in \mathcal{P}_E(\om)$ we have that there exists $\lm=\lm_n(E)$ such that
the energy constraint is satisfied, which is just \eqref{mvp}, where $H_n$ takes the form \eqref{Hlmn.intro}. Putting $\psi_n=G[\rho^{(n)}]$ we come up with a solution of \eqref{mf:n.intro} for that
particular $\lm=\lm_n(E)$, as claimed.
\end{proof}

\bigskip

At this point we wish to pass to the limit as $n\to +\ii$ in \eqref{mvp} to come up with a solution of the (MVP) defining in particular a solution
$\psi$ of \eqref{mf:lm.intro} for some $\lm=\lm(E)$. The situation is rather delicate because we do not know much about $\lm_n=\lm_n(E)$ in \eqref{mvp}.
A uniform bound from above for $\lm_n$ via the Pohozaev identity
is not available for connected but not necessarily simply connected domains. Thus we need to adapt an argument in \cite{BCN24}, although in
a non trivial way, since we cannot take for granted that the weak limit (in the same sense of the proof of Theorem \ref{thm:mvpn}) of $\rho_n$ is
positive a.e. on $\om$. A careful refinement of the argument in \cite{BCN24} is needed which uses an argument similar to that adopted in
Proposition \ref{rho:0}.

\bigskip

{\it The Proof of Theorem \ref{thm:mvp}.}\\
Clearly $E_{0,n}(\om) \to E_{0}$ as $n\to +\ii$ and by the Jensen inequality as in the proof of Theorem \ref{thm:mvpn}, $\rho_0$ is the unique maximizer for $E=E_{0}$. In particular
$\psi_0=G[\rho_0]$ is the unique solution of \eqref{mf:lm.intro} for $\lm=0$.\\
Let $E>E_0$, then for $n$ large enough $E>E_{0,n}$ and by Theorem \ref{thm:mvpn} there exists at least one maximizer $\rho_n:=\rho_{\lm_n}$
of $\mathcal{S}$ satisfying \eqref{mvp} such that $\mathcal{E}(\rho_n)-\mathcal{E}_{\sg,n}(\rho_n)=E$, that is
\begin{align*}
\mathcal{S}(\rho_n)=S_n(E)=\underset{\rho\in\mathcal{P}^{(n)}_E(\O)}\sup \mathcal{\mathcal{S}(\rho)}.
\end{align*}

In particular $\psi_n=G[\rho_n]$ is a solution of \eqref{mf:lm.intro} for some $\lm_n=\lm_n(E)$. Since $\rho_n$ is a sequence of entropy maximizers then it is easy to see
that $\mathcal{S}(\rho_n)\geq -C$, and then, as in the proof of Theorem \ref{thm:mvpn}, we have that $\rho_n\to \rho_\infty\in L^{1}(\om)$ in the weak-$L^1(\om)$ topology and
\begin{equation}\label{enrg.conv:sgp}
\ino \rho_n\to \ino\rho_\infty=1,\quad \mathcal{E}(\rho_n)\to \mathcal{E}(\rho_\infty)\in (0,E].
\end{equation}
By the upper semicontinuity of the entropy we have $\mathcal{S}(\rho_\infty)\geq \limsup\limits_{n}\mathcal{S}(\rho_n)$.\\
At this point we split the discussion in two different cases which require similar but still different arguments.\\

CASE $\sg<0$.\\
Remark that if $\om$ is assumed to be simply connected and piecewise $C^2$, then it is well known that (see \cite{pom}) the Riemann map from $\om$
to the unit disk admits a univalent $C^1$ extension up to the boundary. Therefore, after a transplantation of \eqref{mf:lm.intro} to the unit disk, by a straightforward adaptation of a
well known argument based on the Pohozaev identity (see
for example $\S$ 7 in \cite{clmp1}) we have that $\lm_n\leq \overline{\lm}$, for some $\overline{\lm}$ depending by $\om$ and $\sg$ only, for any solution of
\eqref{mf:lm.intro}. It is worth to remark that this upper bound, which holds for any solution of \eqref{mf:lm.intro}, does not hold for general connected
domains, as well known examples show (\cite{dPKM}, \cite{KPV}, \cite{ns}). We work out a careful adaptation of an argument in \cite{BCN24}
to obtain a similar estimate for the $\lm_n$ corresponding to entropy maximizers obtained in Theorem \ref{thm:mvpn}.
\ble\label{lem:connect}  Assume that $\sg<0$ and $\om$ is connected but not simply connected. Then we have that $\lm_n\leq \overline{\lm}$, $\forall\,n$, for some $\overline{\lm}>0$.
\ele
\begin{proof}
Since $\sg<0$, then $\mathcal{E}(\rho_n)\leq E $ and consequently $\ino |\nabla \psi_n|^2\leq 2E,$ implying that along a subsequence,
for some $\psi\in H^1_0(\om)$, we have $\psi_n\to\psi$ a.e. in $\om$,
strongly in $L^p(\om)$ for any $1\leq p<+\ii$ and weakly in $H^1_0(\om)$.
Since $\psi_n$ is a sequence of solutions of \eqref{mf:n.intro} with $\lm=\lm_n$ we have,
$$
-\ino \psi_n \Delta \phi=\ino \rho_n \phi \to -\ino \psi \Delta \phi = \ino \rho_\infty \phi, \quad \forall \phi \in C^\ii_0{(\ov{\om})},
$$
where $C^{\ii}_0(\ov{\om})$ is the space of $C^{\ii}(\ov{\om})$ functions which vanish on $\pa\om$. By well known results
(see chapters 3, 4 in \cite{Ponce}), since
$\psi\in H^1_0(\om)$, this is equivalent to
\begin{equation}\label{psi:weak}
 -\ino \nabla\psi \cdot \nabla \phi = \ino \rho_\ii \phi, \quad \forall \phi \in C^\ii_c{({\om})}.
\end{equation}

Let $\phi$ be any $H^1_0(\om)$ function and $C^\ii_c{({\om})}\ni \phi_n\to \phi$ in $H^1_0(\om)$, then we obviously have
$$
 -\ino \nabla\psi \cdot \nabla \phi_n \to  -\ino \nabla\psi \cdot \nabla \phi
$$
and
$$
\left| \ino \rho_\ii (\phi_n-\phi)\right|\leq \int\limits_{\{\rho_\ii\leq m\}} \rho_\ii (\phi_n-\phi) +\int\limits_{\{\rho_\ii\geq m\}} \rho_\ii(\phi_n-\phi).
$$
The first term in the r.h.s. is $o(1)$ for each fixed $m\geq 1$ as $n\to +\ii$, while for the second we have, by the Young inequality \eqref{Young},
$$
\int\limits_{\{\rho_\ii\geq m\}} \rho_\ii(\phi_n-\phi)\leq \int\limits_{\{\rho_\ii\geq m\}} [\rho_\ii\log(\rho_\ii)-(\rho_\ii-1)]+\int\limits_{\{\rho_\ii\geq m\}} \left(e^{|\phi_n-\phi|}-1\right).
$$
Since $\ino \rho_\ii \log\rho_\ii <+\ii$, the first term in the r.h.s. is $o(1)$ as $m\to +\ii$ while by the Moser-Trudinger inequality we have
$$
\int\limits_{\{\rho_\ii\geq m\}} \left(e^{|\phi_n-\phi|}-1\right)\leq \int\limits_{\om} e^{|\phi_n-\phi|}|\phi_n-\phi|\leq
e^{\frac{\ino|\nabla (\phi_n-\phi)|^2}{8\pi}}\|\phi_n-\phi\|_{L^{2}(\om)}\leq C \|\phi_n-\phi\|_{L^{2}(\om)}.
$$
Therefore $\psi\in H^1_0(\om)$ satisfies \eqref{psi:weak} for any $\phi \in H^1_0(\om)$ and consequently is a weak solution of
\begin{align*}
\begin{cases}
 -\D \psi=\rho_\ii\,\,\,\,\,\,&\text{in}\,\,\O,\\
 \psi=0 &\text{in}\,\,\p\O.
\end{cases}
\end{align*}
Thus, by known results (\cite{Fer}), $\psi$ is continuous in $\ov{\om}$.

\begin{remark} Just remark that this is not enough to pass to the limit in the singular energy term, which satisfies, recalling \eqref{enrg.conv:sgp},
$$
0\leq -\mathcal{E}_{\sg,n}(\rho_n)=|\sg|\ino \rho_n G_n= E-\mathcal{E}(\rho_n)\to E-\mathcal{E}(\rho_\ii).
$$
\end{remark}

Since $\mathcal{E}(\rho_n)\to \mathcal{E}(\rho_\ii)>0$, we can assume w.l.o.g. that
$$
\mathcal{E}(\rho_n)\geq 2 \gamma E,
$$
for some $\gamma\in (0,\frac12)$ and for any $n$ large enough. Next, let us set,
$$
\om_{E}=\{x\in\om\,:\, \psi(x)-\sg G_{\om}(x,0)\leq \gamma {E}\}.
$$
Since $\psi(x)-\sg G_{\om}(x,0)=\psi(x)+|\sg| G_{\om}(x,0)$ is continuous far away from $x=0$, then $|\om_E|>0$.\\
Remark that, since $\sg<0$, then $-\sg G_{\om}(x,0)\to +\ii$ as $x\to 0^+$, whence there exists $r_0>0$ such that $B_{r_0}(0)\cap \om_{E}=\emptyset$.
Thus, $G_{n}(x)\to G_{\om}(x,0)$ uniformly in $\om_{E}$ and by the Egorov Theorem for any $\eps>0$ there exists
$\om_{\eps,E}\subset \om_{E}$
such that $\psi_n-\sg G_{n}$ converges uniformly in $\om_{\eps,E}$ and
$$
|\om_{E}\setminus \om_{\eps,E}|<\eps.
$$
We claim that
\begin{equation}\label{mE:1}
\int\limits_{\om_{E}}\rho_\ii=:m_{E}>0.
\end{equation}
We argue by contradiction and assume that $\int\limits_{\om_{E}}\rho_\ii=0$, thus for any $\dt>0$ we would have
$$
\int\limits_{\om_{E}}\rho_n=\ino \rho_n\chi(\om_E)<\dt,
$$
for any $n$ large enough. We adopt a refinement of the proof of Proposition \ref{rho:0}. For any large enough $n$ set $\om_1=\{\rho_n\leq m\}$ for some $m>0$ such that $|\{\rho_n>m\}|>0$ and observe that
$$
|\{\rho_n>m\}\cap \om_E|\leq \frac{\dt}{m}.
$$
Thus, for $\dt$ small enough, $|\{\rho_n\leq m\}\cap \om_E|>0$ and we fix any $n$ large enough corresponding to this choice of $\dt$, set
$\ov{\rho}=\rho_n$ and define $\om_0=\{\ov{\rho}\leq m\}\cap \om_E$, which therefore satisfies
$$
\int\limits_{\om_0}\ov{\rho}<\dt, \quad \om_0 \subseteq \om_1 \quad \mbox{and} \quad
0<|\om_0|<|\om|=1.
$$

At this point we let $f_\eps=f_{\eps,0}+f_{\eps,1}$ be defined as in the proof of Proposition \ref{rho:0} with the same notations and
observe once again that
$\ov{\rho}+f_\eps$ is an admissible variation of $\ov{\rho}$ so that we have
$
\mathcal{S}(\ov{\rho}+f_{\eps})\leq \mathcal{S}(\ov{\rho}),
$
which reads as follows,
\begin{align*}
&-\int\limits_{\om_0}(\ov{\rho}+f_{\eps,0})\log (\ov{\rho}+f_{\eps,0})-\int\limits_{\om_{1}\setminus \om_0}\ov{\rho}\log (\ov{\rho})-
\int\limits_{\om_{1}^c}(\ov{\rho}+f_{\eps,1})\log (\ov{\rho}+f_{\eps,1})\leq\\
&-\int\limits_{\om_0}\ov{\rho} \log \ov{\rho}
-\int\limits_{\om_{1}\setminus \om_0} \ov{\rho} \log \ov{\rho} -
\int\limits_{\om_{1}^c}\ov{\rho} \log \ov{\rho},\;\fo\,\eps\in [0,1),
\end{align*}
that is

\begin{align*}
-\int\limits_{\om_0}\left[
(\ov{\rho} + \eps (g_0+\mbox{\rm o}_x(1) ))\log (\ov{\rho}+\eps (g_0+\mbox{\rm o}_x(1) ) ) -\ov{\rho} \log \ov{\rho}\right]\\
-\int\limits_{\om_{1}^c}\left[ (\ov{\rho}+\eps (g_{1}+\mbox{\rm o}_x(1) ))\log
(\ov{\rho}+\eps (g_{1}+\mbox{\rm o}_x(1) ))- \ov{\rho}\log \ov{\rho}\right]\leq 0.
\end{align*}

Since $\rho\geq m $ in $\om_{1}^c$, then we have,
\begin{align*}
&-\int\limits_{\om_{1}^c}\left[ (\ov{\rho}+\eps (g_{1}+\mbox{\rm o}_x(1) ))\log
(\ov{\rho}+\eps (g_{1}+\mbox{\rm o}_x(1) ))- \ov{\rho}\log \ov{\rho}\right]=\\
&\eps\left(\int\limits_{\om_{1}^c}(-1-\log(\ov{\rho}))g_1\right) + \mbox{\rm o}(\eps).
\end{align*}

On the other side, since it will be clear by the proof that there is no loss of generality in assuming $\ov{\rho}>0$ in $\om_0$, for some
$\ov{\rho}(x)<\xi(x)<\ov{\rho}(x)+\eps (g_0(x)+\mbox{\rm o}_x(1))$, we have that,
\begin{align*}
&-\int\limits_{\om_0}\left[
(\ov{\rho} + \eps (g_0+\mbox{\rm o}_x(1) ))\log (\ov{\rho}+\eps (g_0+\mbox{\rm o}_x(1) ) ) -\ov{\rho} \log \ov{\rho}\right]=\\
& \eps\int\limits_{\om_0}(-1-\log(\xi))(g_0+\mbox{\rm o}_x(1) )
\end{align*}
and so we find that,
$$
\int\limits_{\om_0}(-1-\log(\xi))(g_0(x)+\mbox{\rm o}_x(1))\leq
\int\limits_{\om_{1}^c}(1+\log(\rho))g_1  + \mbox{\rm o}(1)\leq C.
$$
However we also have, by using the Jensen inequality, in the limit as $\eps \to 0^+$,
\begin{align*}
&\int\limits_{\om_0}(-1-\log(\xi))(g_0(x)+\mbox{\rm o}_x(1))\geq -C_1 -\int\limits_{\om_0}\log(\xi)(g_0(x)+\mbox{\rm o}_x(1))\geq \\
& -C_1 -2\|g_0\|_{\ii}\int\limits_{\om_0}\log(\xi)\geq -C_1 -2\|g_0\|_{\ii}|\om_0|\log\left(\,\int\limits_{\om_0}\frac{\xi}{|\om_0|}\right)\geq \\
&-C_2 -2\|g_0\|_{\ii}|\om_0|\log\left(\,\int\limits_{\om_0}{\xi}\right)\geq
-C_2 -2\|g_0\|_{\ii}|\om_0|\log\left(\,\int\limits_{\om_0}{\ov{\rho}(x)+\eps (g_0(x)+\mbox{\rm o}_x(1))}\right)\\
&\stackrel{\eps \to 0}{\rightarrow} -C_2 -2\|g_0\|_{\ii}|\om_0|\log\left(\,\int\limits_{\om_0}\ov{\rho}\right)\geq -C_2 -2\|g_0\|_{\ii}|\om_0|\log(\,\dt),
\end{align*}
which is the desired contradiction as $\dt\to 0 ^+$ which proves \eqref{mE:1}.\\

\medskip
In view of \eqref{mE:1} we can choose $\eps$ small enough to guarantee that,
$$
\int\limits_{\om_{\eps,E}}\rho_\ii\geq \frac12 \int\limits_{\om_{E}}\rho_\ii=\frac12 m_{E}>0.
$$
Since
$\psi_n-\sg G_n$ converges uniformly in $\om_{\eps,E}$, for any $n$ large enough we have that,
$$
\sup\limits_{\om_{\eps,E}} (\psi_n-\sg G_{n})\leq 2\gamma E \leq \mathcal{E}(\rho_n).
$$
Therefore, putting
$\mathcal{Z}_n=\ino H_{n}e^{\lm_n \psi_n}=\ino e^{\lm_n (\psi_n-\sg G_n)}$ and recalling that
\begin{equation}\label{entropy:sgm}
S_n(E)=\mathcal{S}(\rho_n)=-\ino \rho_n\log(\rho_n)=-\lm_n\ino \rho_n\left(\psi_n-\sg G_n\right)+\log(\mathcal{Z}_n)=
\end{equation}
$$
-2\lm_n \mathcal{E}(\rho_n)+\lm_n\mathcal{E}_{\sg,n}(\rho_n)+\log(\mathcal{Z}_n),
$$
we have that
$$
\frac12 m_{E}\leq \int\limits_{\om_{\eps,E}}\rho_\ii=
\liminf\limits_{n}\int\limits_{\om_{\eps,E}}\rho_n=
\liminf\limits_{n}\frac{1}{\mathcal{Z}_n}\int\limits_{\om_{\eps,E}}e^{\lm_n (\psi_n-\sg G_n)}\leq
$$
$$
\liminf\limits_{n}\frac{1}{\mathcal{Z}_n}\int\limits_{\om_{\eps,E}}e^{\lm_n \mathcal{E}(\rho_n)}=
\liminf\limits_{n}\frac{(\mathcal{Z}_n)^{\frac12}}{\mathcal{Z}_n}e^{-\frac12 S_n(E)}e^{\frac{\lm_n}{2} \mathcal{E}_{\sg,n}(\rho_n)}
\int\limits_{\om_{\eps,E}}dx\leq |\om_{E}|C_E\liminf\limits_{n}\frac{1}{(\mathcal{Z}_n)^{\frac12}},
$$

implying that $\mathcal{Z}_n\leq C_z$ for some positive constant $C_z$. As a consequence we have that
$$
2\lm_n(\mathcal{E}(\rho_\ii)+o(1))=2\lm_n \mathcal{E}(\rho_n)\leq 2\lm_n \mathcal{E}(\rho_n)-\lm_n\mathcal{E}_{\sg,n}(\rho_n)=-\mathcal{S}(\rho_n)+\log \mathcal{Z}_n\leq C.
$$

\end{proof}

\bigskip

In view of Lemma \ref{lem:connect}, possibly along a subsequence, we can assume that $\lm_n\to \lm_\ii\geq 0$.
First of all we have the following,
\ble Let $\rho_n$ and $\lm_n$ be defined above. There exists a subsequence such that
$$\lm_n\to \lm_\ii <\frac{4\pi}{|\sg|}.$$
\ele
\proof If not, possibly along a subsequence, we would have $\lm_n\to \lm_\ii\geq \frac{4\pi}{|\sg|}$ and consequently, recalling that $\psi_n\geq  0$ and
putting $\mathcal{Z}_n=\ino H_{n}e^{\lm_n \psi_n}=\ino e^{\lm_n (\psi_n-\sg G_n)}$, by the Fatou Lemma we would also have that,

$$
\liminf_{n\to +\ii} \mathcal{Z}_n\geq \liminf_{n\to +\ii}\ino e^{\lm_n |\sg| G_n}\geq \ino e^{\lm_\ii |\sg| G_{\om}(x,0)}=+\ii,
$$
where we just used that $\lm_\ii |\sg| G_{\om}(x,0)\geq C-2\log(|x|)$ in $\om$. Therefore,  recalling that $G_n\geq 0$, we would also have,
$$
C_E\geq S_n(E)=\mathcal{S}(\rho_n)=-\ino \rho_n\log(\rho_n)=-\lm_n\ino \rho_n\left(\psi_n-\sg G_n\right)+\log(\mathcal{Z}_n)=
$$
$$
-2\lm_n \mathcal{E}(\rho_n)+\lm_n\mathcal{E}_{\sg,n}(\rho_n)+\log(\mathcal{Z}_n)=-2\lm_n E-\lm_n\mathcal{E}_{\sg,n}(\rho_n)+\log(\mathcal{Z}_n)=
$$
$$
-2\lm_n E+\lm_n|\sg| \ino\rho_n G_n+\log(\mathcal{Z}_n)\geq -2(\lm_\ii+o(1)) E +\log(\mathcal{Z}_n)\to +\ii,
$$
which is a contradiction.
\finedim

\bigskip

Therefore, possibly along a subsequence, we can assume that $\lm_n\to \lm_\ii\in [0,\frac{4\pi}{|\sg|})$. Thus, if along a subsequence $\lm_n\psi_n$ were bounded, in view of the Lemma and by standard elliptic estimates we would easily deduce that,
possibly along a further subsequence,
$\psi_n\to \psi_\ii$ uniformly and then that $\rho_n\to \rho_\ii$ uniformly, so that we would have that,
\begin{equation}\label{enrg:mvp.1}
E=\mathcal{E}(\rho_n)-\mathcal{E}_{\sg,n}(\rho_n)\to \mathcal{E}(\rho_\ii)-\mathcal{E}_{\sg}(\rho_\ii) \quad \mbox{and}\quad 1=\ino \rho_n\to
\ino \rho_\ii.
\end{equation}
In particular this would imply that $\rho_\ii\in \mathcal{P}_E(\om)$ and we claim that this is enough to guarantee that $\rho_\ii$ is the needed entropy maximizer and that $\psi_\ii$ satisfies \eqref{mf:lm.intro} for
$\lm=\lm_\ii(E)$. Remark that we can exclude that $\lm_\ii(E)=0$ because otherwise we would necessarily have $\rho=\rho_0$ which would imply
$E=E_0$ against the assumption $E>E_0$.\\
\ble\label{lem:unif.conv} Assume that, possibly along a subsequence, as $\lm_n\to \lm_\ii\in (0,\frac{4\pi}{|\sg|})$ we have that
$\psi_n\to \psi_\ii$ uniformly in $\ov{\om}$. Then $\rho_n\to \rho_\ii$ uniformly in $\ov{\om}$ where
$\rho_\ii$ is a maximizer of the {\rm (MVP)
(i.e. \eqref{mvp:intro})}  and satisfies \eqref{mvpe}.
\ele
\begin{proof} Obviously, since $\rho_n$ satisfies \eqref{mvp}, by the assumptions we have that
$\rho_n\to \rho_\ii$ uniformly and consequently
$\rho_\ii$ solves \eqref{mvpe}.
Also, it is easy to
see that $S(E)>-\ii$. Therefore we are just left to prove that, possibly along a subsequence,
$$
\mathcal{S}(\rho_n)\to S(E).
$$
Clearly, in view of \eqref{enrg:mvp.1}, we have that $\rho_\ii\in \mathcal{P}_E(\om)$ implying by uniform convergence that,
$$
\mathcal{S}(\rho_n)\to \mathcal{S}(\rho_\ii)\leq S(E).
$$
We argue by contradiction and assume that,
$$
S(E)\geq 4\dt+\mathcal{S}(\rho_\ii) \quad \mbox{for some}\quad \dt>0.
$$
Let $\rho^{(k)}\in \mathcal{P}_E(\om)$ be any maximizing sequence such that
$$
S(E)\geq \mathcal{S}(\rho^{(k)})\geq 3\dt+\mathcal{S}(\rho_\ii),
$$
for any $k$ large enough, then for any $n$ large enough we also have
\begin{equation}\label{mvp:contr1}
\mathcal{S}(\rho^{(k)})\geq 2\dt+\mathcal{S}(\rho_n)=2\dt+\sup\{\mathcal{S}(\rho),\;\rho \in \mathcal{P}^{(n)}_E(\O)\},
\end{equation}
where we recall that
\begin{equation*}
\mathcal{P}^{(n)}_E(\O)=\left\{\rho\in \mathcal{P}(\O)|\, \mathcal{E}(\rho)-\mathcal{E}_{\sg,n}(\rho)=E\right\}.
\end{equation*}

We will obtain the desired contradiction showing that for any such fixed $k$ and any $n$ large enough there exists $\rho_{k,n}\in \mathcal{P}^{(n)}_E(\O)$ such that
\begin{equation}\label{mvp:contr3}
\mathcal{S}(\rho_{k,n})\geq \mathcal{S}(\rho^{(k)})-\dt,
\end{equation}
which clearly contradicts \eqref{mvp:contr1}. We fix for the rest of this proof any $k$ large enough as above. Put $G(x)=G_{\om}(x,0)$ and define
\begin{align*}
    \epsilon_{n,k}& :=\mathcal{E}_{\sg}(\rho^{(k)})-\mathcal{E}_{\sg,n}(\rho^{(k)})=(\mathcal{E}(\rho^{(k)})-\mathcal{E}_{\sg,n}(\rho^{(k)}))-(\mathcal{E}(\rho^{(k)})-\mathcal{E}_{\sg}(\rho^{(k)})).
\end{align*}
Then,
\begin{align*}
|\epsilon_{n,k}|& :=\left|(\mathcal{E}(\rho^{(k)})-\mathcal{E}_{\sg,n}(\rho^{(k)}))-(\mathcal{E}(\rho^{(k)})-\mathcal{E}_{\sg}(\rho^{(k)}))\right|\\
&=|\sg| \ino\rho^{(k)}\left| G-G_n\right|\leq |\sg|\int\limits_{\om \cap B_\eps}\rho^{(k)}\left| G-G_n\right|+
|\sg|\int\limits_{B_\eps}\rho^{(k)}\left| G-G_n\right|\\
&= o(1)+|\sg|\int\limits_{B_\eps}\rho^{(k)}G,
\end{align*}
where $o(1)\to 0$ as $n\to +\ii$ and we used $G_n\leq G$. However, by the Young inequality \eqref{Young}, we have that,
$$
\int\limits_{B_\eps}\rho^{(k)}G\leq \int\limits_{B_\eps}\left[\rho^{(k)}\log(\rho^{(k)})-(\rho^{(k)}-1)\right]\chi({\rho^{(k)}\geq 1})+
\int\limits_{B_\eps} (e^{G}-1)\to 0,\;\mbox{as}\;\eps\to 0,
$$
where we used $\int\limits_{\om}\rho^{(k)}\log(\rho^{(k)})<+\ii$ and the exponential integrability of $G$. Therefore we have $\eps_{n,k}\to 0$ as $n\to +\ii$
and we define,
$$
\rho^{(k)}_t=\rho^{(k)}+t g_n,
$$
where $g_n$ is any function satisfying,
$$
g_n\in L^{\ii}(\om), \quad  \ino g_n=0,
$$
so that $\ino \rho^{(k)}_t=1$ for any $t$. Thus we have,
$$
|\mathcal{E}(\rho^{(k)}_t)-\mathcal{E}_{\sg,n}(\rho^{(k)}_t)-E|=
|\mathcal{E}(\rho^{(k)}_t)-\mathcal{E}_{\sg,n}(\rho^{(k)}_t)-(\mathcal{E}(\rho^{(k)})-\mathcal{E}_{\sg}(\rho^{(k)}))|=
$$
\begin{equation}\label{mvp:contr2}
|\eps_{n,k}+ta_{n,k}+b_{n}t^2|,
\end{equation}
where,
$$
a_{n,k}:=\ino (G[\rho^{(k)}]+|\sg| G_n)g_n,\qquad b_n=\tfrac12\ino g_n G[g_n].
$$
At this point we choose $g_n=g_{+,n}-g_{-,n}$,
$$
g_{+,n}=\frac{\chi(B_{r_n})}{|B_{r_n}|},\qquad g_{-,n}=\frac{\chi(\om_{-,n})}{|\om_{-,n}|},
$$
where $\om_{-,n}\subset \om$ is any nonempty set such that $\sup\limits_{x\in \om_{-,n}}\mbox{dist}(x,\pa \om)\leq d_n$, for a suitable sequence $d_n=o_n(1)$ to be fixed later,
and $r_n\to 0$ is any sequence such that,
$$
\frac{\eps_n^2}{r_n^2}\log(\eps_n)\to 0.
$$
Clearly $\ino g_n=0$ and standard evaluations show that,
$$
a_{n,k}=-\frac{|\sg|}{2\pi r_n^2}\left[(\eps_n^2+r_n^2)\log(\eps_n^2+r_n^2)-\eps_n^2\log(\eps_n^2)\right]+O_k(1)=
-\frac{|\sg|+o_k(1)}{2\pi}\log(r_n^2),
$$
$$
b_n=-\frac{1}{8\pi}\log(r_n^2)+O_k(1)=-\frac{1+o_k(1)}{8\pi}\log(r_n^2),
$$
where $o_k(1)\to 0$ for $k$ fixed and large as above and likewise $|O_k(1)|\leq C_k$. Therefore we deduce that, putting $t=\frac{s}{\log(r_n^2)}$,
$$
\eps_{n,k}+ta_{n,k}+b_{n}t^2=\eps_{n,k}+t(\log(r_n^2))\left(-\frac{|\sg|}{2\pi}(1+o_k(1))+\frac{t}{8\pi}(1+o_k(1))\right)=
$$
$$
\eps_{n,k}-s\frac{|\sg|}{2\pi}(1+o_k(1)).
$$
At this point we see that there exists some uniformly bounded sequence $s_{n,k}>0$ such that
$\eps_{n,k}-s_{n,k}\frac{|\sg|}{2\pi}(1+o_k(1))=0$, implying
from \eqref{mvp:contr2} that,
$$
\rho_{k,n}=\left.\rho^{(k)}_{t}\right|_{t=\frac{s_{n,k}}{\log(r_n^2)}},
$$
satisfies $\mathcal{E}(\rho_{k,n})-\mathcal{E}_{\sg,n}(\rho_{k,n})=E$. Moreover, by our choice of $s_{n,k}$ and $g_n$, 
$$\int_\O\rho_{k,n}=1\qquad\int_\O\rho_{k,n}\log\rho_{k,n}<+\infty$$ and, by choosing $d_n$ such that $d_n^{-1}=o(\log(r_n^2))$, it is easy to verify that $\rho_{k,n}\geq0$ a.e. Hence, $\rho_{k,n}\in \mathcal{P}^{(n)}_E(\O)$. Since $s_{n,k}$ is uniformly bounded, then
$\frac{s_{n,k}}{\log(r_n^2)}\to 0$, as $n\to +\ii$ and consequently $\rho_{k,n}\to \rho^{(k)}$ uniformly implying that $\rho_{k,n}$ satisfies
\eqref{mvp:contr3}, as claimed.
\end{proof}

\medskip

Therefore, we are left with showing that $\lm_n\psi_n$ admits a bounded subsequence. We argue by contradiction and assume that,
\begin{equation}\label{blowup:psi}
\lm_n\psi_n(x_n)=\sup\limits_{\om}\lm_n\psi_n\to +\ii,
\end{equation}
as $n\to +\ii$, for some sequence $\{x_n\}\subset \om$. By using that the domain is piecewise $C^2$,
it has been shown in \cite{CCL}
that $x_n$ stays uniformly bounded far away from the boundary, whence we can assume that,
\begin{equation}\label{unif:bdy}
\sup\limits_{\om_0}\lm_n\psi_n\leq C_0,
\end{equation}
for some small open neighborhood $\om_0$ of $\pa \om$. As a consequence in particular we have $x_n\to x_\ii\in \om$. Denoting $H_n$ as in \eqref{Hlmn.intro}, let us set,
\beq\label{vn:def}
v_{n}=\lm_n \psi_n-\log\left(\ino H_{n} e^{\lm_n \psi_n}\right)+\log\lm_n,
\eeq
which satisfies

\begin{align}\label{mf:cvp-mvpn}
\begin{cases}
 -\D v_{n}=H_{n}e^{v_n}\,\,\,\,\,\,&\text{in}\,\,\O,\\
 \lm_n= \int\limits_{\O}H_{n}e^{v_n},\\
 v_{n}=-c_n &\text{in}\,\,\p\O,
\end{cases}
\end{align}
where,
\beq\label{cn:def}
c_n=-\log(\lm_n)+\log\left(\ino H_{n} e^{\lm_n \psi_n}\right).
\eeq
There is no loss of generality in assuming that $v_n(x_n)\to +\ii$, since otherwise $v_n$ would be uniformly bounded from above
and we would immediately deduce that $\rho_n\leq C H_n$, implying by standard elliptic estimates that $\psi_n$  would be uniformly bounded as well (recall that $\lm_n\to\lm_\ii$).\\

At this point we can apply Theorem \ref{Concentration-Compactness alternative}, implying that necessarily {\rm (iii)} holds true for some $m\geq 1$ and in particular
$$
\lm_n\rho_n=H_{n}e^{\dis v_{n}}\rightharpoonup \sum_{i=1}^m\beta_i\delta_{q_i},
$$
weakly in the sense of measures in $B_r$, for any $r\in(0,1)$, with $\beta_i\in 8\pi\N$ if $q_i\neq0$, while if $q_i=0$ for some $i\in\{1,\ldots,m\}$ then
$\beta_i$ satisfies \eqref{mmbeta}. In any case, in particular $\lm_n\rho_n$ converges in the sense of distributions in $B_r$ to $\sum_{i=1}^m\beta_i\delta_{q_i}$, which is a contradiction because in the same time $\lm_n\rho_n$ converges in the weak-$L^1(\om)$ topology to $L^1(\om)$ density, say $\rho$. Thus,
in particular $\lm_n\rho_n$ converges to $\rho$ in the sense of distributions in $B_r$ and the uniqueness of the limit would imply that $\rho\chi(B_r)$ should coincide
with $\sum_{i=1}^m\beta_i\delta_{q_i}$, which is the desired contradiction. This fact concludes the proof of the case $\sg<0$.

\bigskip

CASE $\sg>0$. If $\om$ is simply connected, as discussed above for the case $\sg<0$, we have $\lm_n\leq \ov{\lm}$ by the Pohozaev identity. Suppose we were able to prove that $\lm_n\leq \ov{\lm}$ also for connected but possibly not simply connected domains.
Then, by a straightforward adaptation of Lemma \ref{lem:unif.conv} to the case $\sg>0$, we see again that it is enough to prove that $\psi_n\to \psi$
uniformly. At this point we can apply Theorem \ref{Concentration} and the contradiction follows as in the case $\sg<0$.
Therefore we are left with the following analogue of Lemma \ref{lem:connect} whose proof requires a suitably adapted argument.

\ble\label{lem:connect2}  Assume that $\sg>0$ and $\om$ is connected but not simply connected. Then we have that $\lm_n\leq \overline{\lm}$, $\forall n$, for some $\overline{\lm}>0$.
\ele
\begin{proof}
Since $\psi_n$ is a sequence of (continuous) solutions of \eqref{mf:n.intro} with $\rho_n$ satisfying $\ino \rho_n\log\rho_n\leq C$, by the Green representation formula and the Young inequality \eqref{Young} we have,
$$
\|\psi_n\|_\ii=\psi_n(x_n)\leq \int\limits_{\om} [\rho_n\log(\rho_n)-(\rho_n-1)]\,dy+\int\limits_{\om} \left(e^{G_\om(x_n,y)}-1\right)\,dy\leq C_0,
$$
where the last integral is estimated uniformly in a standard way, see for example \cite{bm}. Therefore we have that, for any $r$ small enough,
\beq\label{unif:sgp}
\psi_n(x)-\sg G_n(x) \leq C_1+\frac{\sg}{4\pi}\log((\eps_n^2+|x|^2)),\quad \forall\, |x|\leq r.
\eeq
On the other side we have, recalling \eqref{enrg.conv:sgp},
$$
0\geq -\mathcal{E}_{\sg,n}(\rho_n)=-\sg\ino \rho_n G_n= E-\mathcal{E}(\rho_n)\to E-\mathcal{E}(\rho_\ii)\geq -C.
$$
Therefore, in view of \eqref{unif:sgp}, there exists some $r_0<1$ small enough such that, for any $n$ large enough,
$$
\sup\limits_{B_{r_0}}(\psi_n(x)-\sg G_n(x)) \leq \mathcal{E}(\rho_n)-\frac12\mathcal{E}_{\sg,n}(\rho_n).
$$

The same argument used to prove \eqref{mE:1} in case $\sg<0$ works fine with minor changes showing that
\begin{equation}\label{mE:2}
\int\limits_{B_{r_0}}\rho=:m_{0}>0.
\end{equation}
As a consequence we deduce that, recalling \eqref{entropy:sgm},
$$
m_{0}= \int\limits_{B_{r_0}}\rho=
\liminf\limits_{n}\int\limits_{B_{r_0}}\rho_n=
\liminf\limits_{n}\frac{1}{\mathcal{Z}_n}\int\limits_{B_{r_0}}e^{\lm_n (\psi_n-\sg G_n)}\leq
$$
$$
\liminf\limits_{n}\frac{1}{\mathcal{Z}_n}\int\limits_{B_{r_0}}e^{\lm_n (\mathcal{E}(\rho_n)-\frac12\mathcal{E}_{\sg,n}(\rho_n))}=
\liminf\limits_{n}\frac{(\mathcal{Z}_n)^{\frac12}}{\mathcal{Z}_n}e^{-\frac12 S_n(E)}
\int\limits_{B_{r_0}}dx\leq |B_{r_0}|C_E\liminf\limits_{n}\frac{1}{(\mathcal{Z}_n)^{\frac12}},
$$

implying that $\mathcal{Z}_n\leq C_z$ for some positive constant $C_z$, in this case as well. Therefore we have that
$$
\lm_n E=\lm_n(\mathcal{E}(\rho_n)-\mathcal{E}_{\sg,n}(\rho_n))=-\lm_n \mathcal{E}(\rho_n)+\lm_n(2\mathcal{E}(\rho_n)-\mathcal{E}_{\sg,n}(\rho_n))=
$$
$$
-\lm_n \mathcal{E}(\rho_n)-\mathcal{S}(\rho_n)+\log \mathcal{Z}_n\leq -\mathcal{S}(\rho_n)+\log \mathcal{Z}_n\leq C.
$$
This fact concludes the proof of the case $\sg>0$.
\end{proof}

\bigskip
\bigskip

At this point, by using Theorems \ref{thm:cvpintro} and \ref{thm:mvp} and the uniqueness results in \cite{bl1} for $\sg>0$ and \cite{BJL}, \cite{BCJL} for $\sg<0$,
as in \cite{clmp2} we can prove the equivalence of statistical ensembles for domains of type I. More exactly this is the generalization of
Proposition 3.3 in \cite{clmp2} to
the case where $\om$ is an open, bounded, piecewise $C^2$ and simply connected domain of Type I.

\begin{theorem}\label{thm:Equiv}  Let $\om$ be an open, bounded, piecewise $C^2$ and simply connected domain of Type I and
$\beta=-\lambda$.
Then we have:\\
(i) $F(\beta)=-f(-\beta)$ is defined and strictly convex for $\beta> -\lm_\sg$;\\
(ii) $F(\beta)$ is differentiable for $\beta>-8\pi$ and
$$
E(\beta)=-F^{'}(\beta)=\frac12 \ino \rho_{\lm}G[\rho_\lm]-\sg\ino \rho G_{\om}(x,0),
$$
where $\rho_{\lm}$ solves \eqref{mf:lm.intro} for $\lm=-\beta$. In particular $E(\beta)$ is a continuous and strictly monotone decreasing bijection;\\
(iii) $S(E)=\inf\limits_{\beta}\{F(\beta)+\beta E(\beta)\}$ and hence is a smooth and concave function of $E$;\\
(iv) If $\rho^{(E)}$ is a maximizer for \eqref{mvp} then there exists a unique $\beta>-8\pi$ such that
$\rho_{_{-\beta}}=\rho^{(E(\beta))}$ which solves \eqref{mf:lm.intro} for $\lm=-\beta$ (equivalence of statistical ensembles) and the solution is unique.
\end{theorem}
\proof Relying on Theorems \ref{thm:cvpintro} and \ref{thm:mvp}, the proof adopted in \cite{clmp2} works exactly as it stands,
where one uses the uniqueness results in \cite{bl1} for $\sg>0$ and those in \cite{BJL}, \cite{BCJL} for $\sg<0$. Remark that those uniqueness results are
obtained for \eqref{mf:lm.intro} where the order of the exponent $\alpha$ of the weight function $|x|^{2\al}$ is kept fixed, while
in this context we have $\al=\frac{\lm}{4\pi}\sg$. However is not difficult to see that for $\sg>0$ the uniqueness threshold is still $8\pi$,
while for $\sg<0$ the relevant information is just the inequality
$\lm=\lm\int\limits_{\om}\rho_{\lm}\leq 8\pi(1+\al)$, which is granted in this case by the fact that $\lm<\lm_\sg=\frac{8\pi}{1-2\sg}$.
In other words, as far as uniqueness is concerned, the overall effect of the replacement $\al\mapsto \frac{\lm}{4\pi}\sg$ is just that
of a suitable scaling of the threshold.
\finedim

\bigskip
\bigskip

\section{An example of a domain of Type I with $\sg<0$.}\label{sec5}

In this section we discuss the simplest case of a domain of Type I with $\sg<0$, which is $\om=B_1=B_1(0)$ with the singularity always assumed to be
the origin,
so that \eqref{mf:lm.intro} takes the form,
\begin{align}
\begin{cases}\label{mf:lm.disk}
 -\D \psi_{\lambda,\sg}=\rho_{\lambda,\sg}=\dfrac{{ H_{\lm,\sg}} e^{\lambda \psi_{\lambda,\sg}}}
 { \int\limits_{B_1}  { H_{\lm,\sg}} e^{\lambda \psi_{\lambda,\sg}}}\,\,\,\,\,\,&\text{in}\,\,B_1,\\
 \psi_{\lambda,\sg}=0 &\text{on}\,\,\p B_1,
\end{cases}
\end{align}
where
$$
H_{\lm,\sg}(x)=|x|^{\frac{\sigma\lambda}{2\pi}}\equiv |x|^{\al_{\lm,\sg}}, \quad \al_{\lm,\sg}:=\frac{\sg \lambda}{2\pi}.
$$
Recall that $\sg<0$ here, whence we have that,
$$
\lm_\sg=\frac{8\pi}{1+2|\sg|}.
$$
In view of Remark \ref{rem:conicsing}, we may rely on the uniqueness of solutions of \eqref{mf:lm.disk} for $\lm<\lm_\sg$ (\cite{BJL}), implying that  the unique solution $\psi_{\lm,\sg}$ of
\eqref{mf:lm.disk} reads,
$$
\lm \psi_{\lm,\sg}(x)=2\log\left(\frac{1+\gamma_{\lm,\sg}^2}{1+\gamma_{\lm,\sg}^2|x|^{2(1+\al_{\lm,\sg})}}\right),\quad |x|<1,\qquad
\gamma_{\lm,\sg}^2=\frac{1}{1+2|\sg|}\frac{\lm}{\lm_\sg-\lm},
$$
for any $\lm<\lm_\sg$, so that in particular,
$$
\int\limits_{B_1}  H_{\lm,\sg}e^{\lambda \psi_{\lambda}}=\pi \gamma_{\lm,\sg}^2.
$$
It is useful to observe that,
$$
\gamma_{\lm,\sg}^2=\frac{1}{1+2|\sg|}\frac{\lm}{\lm_\sg-\lm}\equiv \frac{\lm}{8\pi-(1+2|\sg|)\lm}\equiv \frac{\lm}{8\pi(1+\al_{\lm,\sg})-\lm}.
$$
Consequently the vorticity density takes the form,
$$
\rho_{\lambda,\sg}(x)=\tfrac{\left(1+\gamma_{\lm,\sg}^2\right)^2}{\pi \gamma_{\lm,\sg}^2}\frac{|x|^{\al_{\lm,\sg}}}
{\left(1+\gamma_{\lm,\sg}^2|x|^{2(1+\al_{\lm,\sg})}\right)^2},
$$
which satisfies,
$$
\rho_{\lambda,\sg} \rightharpoonup \delta_{p=0},\quad \lm\to (\lm_\sg)^{-},
$$
weakly in the sense of measures in $B_1$. Also, we have that,
$$
\al_{\lm,\sg}\to -\frac{2|\sg|}{1+2|\sg|},\quad \lm\to (\lm_\sg)^{-},
$$
which is bounded far away from $-1$. Remark that for $\sg=0$ we just recover the well known
solutions of the regular problem in the unit disk (\cite{clmp1}).\\
For the time being, just to simplify notations, let us set $\al=\al_{\lm,\sg}$. Then by a straightforward evaluation we have,
\beq\label{enrg:sec5}
E_{\lm,\sg}:=\mathcal{E}(\rho_{\lambda,\sg})=
-\frac{1}{8\pi(1+\al)}\frac{8\pi(1+\al) }{\lm}\left(1+\frac{8\pi(1+\al)}{\lm}\log(1-\frac{\lm}{8\pi(1+\al)})\right),
\eeq
and
\beq\label{entr:sec5}
\mathcal{S}(\rho_{\lambda,\sg})=2+\log(\frac{\pi}{1+\al})+\left(\frac{16\pi(1+\al)}{\lm}-\frac{\al}{1+\al}\frac{8\pi(1+\al)}{\lm}-1\right)
\log(1-\frac{\lm}{8\pi(1+\al)}).
\eeq
First of all, by using the uniqueness results in \cite{BJL},
since from \eqref{enrg:sec5} we have $E_{\lm,\sg}\to +\ii$ as $\lm\to (\lm_\sg)^{-}$ and since $E_{0,\sg}=E_0(B_1)$, then we see that
$B_1$ is of Type I according to Definition \ref{typeI}. Also, by using \eqref{enrg:sec5}, \eqref{entr:sec5} and
again by the uniqueness results in \cite{BJL}, to be used this time together with
the equivalence of statistical ensembles (Theorem \ref{thm:Equiv}), a lengthy but straightforward evaluation shows that
the entropy $S(E)$ satisfies,
$$
S(E)=-8\pi E+2+\frac{1}{1+\al}+\log(\frac{\pi}{1+\al}) +\frac{e}{1+\al}e^{-8\pi(1+\al)E}+{\mbox \rm O}\left(E^2e^{-16\pi(1+\al)E}\right), \quad E\to +\ii.
$$
For $\sg=0$ we recover the regular case, the leading term in the right hand side being of course in agreement with
a well known general result (see Proposition 6.1 in \cite{clmp2}) claiming that
we have $-8\pi E +C_1 \leq S(E) \leq -8\pi E +C_2$ for any domain. However, the linear divergence rate is $-8\pi$ even if $|\sg|\neq 0$.
This is not good news as far as one is interested in quantitative estimates to distinguish solutions
which are entropy maximizers from those like those in section \ref{sec6} below which should be not.
Indeed, it seems to suggest that the leading order in the entropy as $E\to +\ii$ is not affected by the singularity,
the difference being due to lower order, uniformly bounded terms. It seems that neither the refined estimates of Theorem \ref{profile-new} are enough to catch these difference, which is why we will discuss this point elsewhere (\cite{BCYZZ2}).

\bigskip
\bigskip

\section{High energy limit of solutions $\rho_n$ of \eqref{mf:n.intro}.}\label{sec6}
In the regular case ($\sg=0$) it has been shown in \cite{clmp2} that for \underline{any} domain (i.e. both Type I and Type II), in the limit
$E\to +\ii$ the entropy maximizers concentrate to a Dirac delta whose singular point is a maximum point of $\gamma_\om$. This is one of the major
achievements of the Onsager theory: as $E\to +\ii$ vortices of the same sign attract each other and possibly
coalesce to form a unique macroscopic vortex, which necessarily happens in the negative temperature regime $\lm(E)>0$. Concerning the case $\sg>0$
(that is, a negative counter-rotating vortex fixed at
the origin with strength $-\sg$) it seems therefore obvious that the states were the (positive) vorticity $\rho$ concentrates
on the fixed counter-rotating vortex location (indicating a collision of vortices of opposite sign), would not be entropy maximizers. Indeed,
as mentioned in the introduction, in the model corresponding to \eqref{mf:lm.intro} and for domains of Type I,
whenever this sort of collision occurs, then $\lm\to \frac{8\pi}{1-2\sg}$ for some $\sg<\frac12$ (see Theorem \ref{massboundarycontrol0}).
Thus, states of this sort are allowed only in a
temperature region far away (since $\frac{8\pi}{1-2\sg}>8\pi$) from the equivalence region $\lm(E)<8\pi$. An intriguing natural question arise at this stage:
is this true also for the entropy maximizers $\rho_n$ of the regularized problem as derived in Theorem \ref{thm:mvpn}? In other words, suppose that for $\sg>0$
some density $\rho_n$ with energy $E_n\to +\ii$ develops a singularity at the origin: is it true that $\lm_n$ is bounded far away from $8\pi$?
Remark that afterall
the total stream functions $\psi_{\sg,n}=\psi_n-\sg G_n(x)=G[\rho_n]-\sg G_n(x)$ are just solutions of the Euler equation in vorticity form,
$$
-\Delta \psi_{\sg,n}=\dfrac{e^{  \lm_n\psi_{\sg,n}}}{\int_\O e^{  \lm_n\psi_{\sg,n}}}-\sg g_n,
$$
where the last term $\sg g_n=\frac{\sg}{\pi}\frac{\eps_n^2}{\eps_n^2+r^2}$ is converging in the sense of measures to $\sg \dt_{p=0}$.
From the point of view of the description of a vorticity field interacting with a
fixed singular source, we would expect essentially the same behavior if the Dirac data describing the singular vortex
were replaced by a highly peaked smooth density.
It seems however that to provide a rigorous answer is not trivial and, as far as we could check,
the model allows the collapse of vortices on the fixed counter-rotating vortex as $\lm\to 8\pi$, see Theorem \ref{profile-new}.
Thus it would be interesting to have a quantitative entropy estimates for these solutions,
to be compared with entropy maximizers concentrating as $\lm\to 8\pi$ at a maximum point of
$\gamma_\om(x)-\sg G_\om(x,0)$. Estimates of these sort seems to be rather delicate (see section \ref{sec5}), which is why
we will discuss this point elsewhere (\cite{BCYZZ2}). Actually, it would be interesting to use
more sophisticated estimates in the same spirit of \cite{dwz,wz1,wz2} to check whether or not and
possibly under which conditions asymptotically non radial (around the singularity) blow up solutions of the same sort described in Theorem 
\ref{profile-new} really exist as $\lm\to 8\pi$.\\

We will show that, for solutions $\rho_n$ of \eqref{mf:n.intro},  
whenever a singularity arise according to Theorem \ref{profile-new} at the vortex
point with $\sg>0$, then necessarily we have, 
\begin{equation}\label{enrg:n}
E_n=\mathcal{E}(\rho_{n})-\mathcal{E}_{\sg,n}(\rho_{n})=\frac12 \ino \rho_nG[\rho_n]-\sg\ino \rho_n G_n(x)\to +\ii.
\end{equation}
In particular the non radial profiles (that is (I) and (II) of Theorem \ref{profile-new}) share this property as $\lm_n\to 8\pi$. Instead the radial profiles (III) are characterized by $\lm_n\to \lm_\ii>8\pi$.  \\

 By defining $H_n$, $v_n$ and $c_n$ as in the proof of Theorem \ref{thm:mvpn}, we have that $v_n$ is a solution of
\eqref{mf:cvp-mvpn}, whose energy takes the form $E_n$ in \eqref{enrg:n}.
We assume that $v_n$ satisfies \eqref{bounded} and \eqref{explosion}, that is, the origin is the unique blow up point of $\vn$ in $\om$. Thus, in view of
Theorem \ref{massboundarycontrol}, by assuming also that $\sg\in(0,\frac12)$ and $8\pi\leq \lm_\ii< \frac{4\pi}{\sg}$,
we see that all the hypothesis of Theorems \ref{Concentration}, \ref{massboundarycontrol} and \ref{profile-new} are fulfilled.
In this situation we have,

\medskip

\begin{lemma} Under the assumptions of Theorem \ref{profile-new}, whence in particular,\\
$\sg\in(0,\frac12)$, $\lm_n \to \lm_\ii$ and $8\pi\leq \lm_\ii< \frac{4\pi}{\sg}$,
we have that \eqref{enrg:n} holds true.
\end{lemma}

\proof

We proceed to the discussion of the cases in the same order as in Theorem \ref{profile-new}, which we assume without loss of generality
to hold with $\om$ replacing $B_1$. Let us also remark that, since $\psi_n$ is a solution of \eqref{mf:n.intro},
then $v_n$ defined as in \eqref{vn:def} is a solution of \eqref{mf:cvp-mvpn} and then in particular in view of \eqref{beqlm}, we have that
\begin{equation}\label{profilewn}
v_n(x)+c_n={\lm_n}\psi_n(x)=\beta G_\om(x,0)+o(1),\quad x\in \om\setminus B_r(0),
\end{equation}
for any $r$ small enough, where $c_n$ is defined as in \eqref{cn:def}. To match the notations of Theorem \ref{profile-new}, we just set
\begin{equation*}
K_n(x)=e^{-\lm_n \sg R_{n}(x)},
\end{equation*}
and recall that, in view of \eqref{profilewn-n}, we have that,
\begin{equation}\label{cnmn}
c_n=-m_n.
\end{equation}
\begin{itemize}
 \item {\bf CASE} (I): there exists a subsequence such that $\tfrac{\epsilon_n}{t_n}\to+\infty$.
\end{itemize}

Recall that in this case we have \eqref{profilew:H-lm}, that is
\begin{equation*}
w_n(z)=\log\left(\dfrac{\theta_n^{2(1+\al_{n,\sg})}}
{\left(1+\bar\gamma_n\theta_n^{2(1+\al_{n,\sg})}|z-\frac{x_n}{\epsilon_n}|^{\frac{\lm_n}{4\pi}}\right)^2}\right)+O(1),\quad z\in B^{(n)}=B_{\frac{r}{\eps_n}},
\end{equation*}
where,
$$
\theta_n^{2(1+\al_{n,\sg})}=e^{\dis w_n(\tfrac{x_n}{\epsilon_n})}=e^{\dis v_n(x_n)}\epsilon_n^{2(1+\al_{n,\sg})}=
\left(\tfrac{\epsilon_n}{\delta_n}\right)^{2(1+\al_{n,\sg})}
$$
and $\gamma_n$ as in the claim of Theorem \ref{profile-new}. Recalling that $V_n(z)=(1+|z|^2)^{\bns}K_n(z)$ and
in view of \eqref{eps:I-A}, \eqref{cnmn} and \eqref{enrg:n}, we have that,
\begin{align*}
& 2\lm_\ii E\geq \liminf\limits_{n\to +\ii}2\lm_n\left(\mathcal{E}(\rho_{n})-\mathcal{E}_{\sg,n}(\rho_{n})\right)=
\liminf\limits_{n\to +\ii}\left( \lm_n\ino \rho_nG[\rho_n]-2\sg\lm_n\ino \rho_n G_n(x)\right)=\\
& \liminf\limits_{n\to +\ii}\ino \rho_n \left( \lm_n \psi_n -2\sg\lm_n G_n(x)\right)=
\liminf\limits_{n\to +\ii}\ino \rho_n \left( v_n+c_n -2\sg\lm_n G_n(x)\right)=\\
& O(1)+\liminf\limits_{n\to +\ii}\ino \rho_n \left( v_n+c_n +\frac{\lm_n}{2\pi}\sg \log(\epsilon_n^2+|x|^2)\right)=\\
& O(1)+\liminf\limits_{n\to +\ii}\int\limits_{B_r(0)} \rho_n \left( v_n+c_n +2\al_{n,\sg}\log(\epsilon_n^2+|x|^2)\right)=\\
& O(1)+\liminf\limits_{n\to +\ii}\lm_n^{-1}\int\limits_{B^{(n)}} V_n(z)e^{\dis w_n(z)}
\left(v_n(\epsilon_n z)+c_n +2\al_{n,\sg}\log(\epsilon_n^2+\epsilon_n^2|z|^2)\right)=\\
& O(1)+\liminf\limits_{n\to +\ii}\left(\lm_n^{-1}\int\limits_{B_R(\frac{x_n}{\epsilon_n})} V_n(z)e^{\dis w_n(z)}
v_n(\epsilon_n z)+\int\limits_{B^{(n)}\backslash B_R(\frac{x_n}{\epsilon_n})} V_n(z)e^{\dis w_n(z)}
v_n(\epsilon_n z)+\right.\\
&\hspace{5cm}\left.+\int\limits_{B^{(n)}} V_n(z)e^{\dis w_n(z)}c_n+2\al_{n,\sg}\int\limits_{B^{(n)}}
V_n(z)e^{\dis w_n(z)}\log(\epsilon_n^2+\epsilon_n^2|z|^2)\right)=
\end{align*}
\begin{align*}
& O(1)+\liminf\limits_{n\to +\ii}\left(\lm_n^{-1}\int\limits_{B_R(\frac{x_n}{\epsilon_n})} V_n(z)e^{\dis w_n(z)}
\log\left(\tfrac{\delta_n^{-2(1+\bns)}}{(1+\bar\gamma_n\theta_n^{2(1+\bns)}|z-\frac{x_n}{\epsilon_n}|^2)^2}\right)+\right.\\
&\hspace{5cm}+\int\limits_{B^{(n)}\backslash B_R(\frac{x_n}{\epsilon_n})} V_n(z)e^{\dis w_n(z)}
\log\left(\tfrac{\delta_n^{-2(1+\bns)}}{(1+\bar\gamma_n\theta_n^{2(1+\bns)}|z-\frac{x_n}{\epsilon_n}|^{\frac{\lm_n}{4\pi}})^2}\right)+\\
&\hspace{5cm}\left.+\int\limits_{B^{(n)}} V_n(z)e^{\dis w_n(z)}c_n+2\al_{n,\sg}\int\limits_{B^{(n)}}
V_n(z)e^{\dis w_n(z)}\log(\epsilon_n^2+\epsilon_n^2|z|^2)\right)\geq\\
\end{align*}
\begin{align*}
&\mbox{(where, for a suitable $R>0$, we have used the pointwise estimate \eqref{profilew} and \eqref{profilev:H})}\\
& O(1)+\liminf\limits_{n\to +\ii}2\lm_n^{-1}\int\limits_{B^{(n)}} V_ne^{w_n} \log\left(\tfrac{\epsilon_n^{2(1+\al_{n,\sg})+2\bns-\frac{\lm_n}{4\pi}}}{\delta_n^{2(1+\al_{n,\sg})}}\right) +
\liminf\limits_{n\to +\ii} (I_{n,1}+I_{n,2}),
\end{align*}
with
\begin{equation*}
I_{n,1}:=\lm_n^{-1}\int\limits_{B_R(\frac{x_n}{\epsilon_n})}\frac{\theta_n^{2(1+\bns)}(1+|z|^2)^{\bns}e^{-\sigma\lm_n R_\O(\epsilon_n z,0)}}{(1+\bar\gamma_n\theta_n^{2(1+\bns)}|z-\frac{x_n}{\epsilon_n}|^2)^2}\log\left(\frac{(1+|z|^2)^{2\al_{n,\sg}}}
{\left(1+\bar\gamma_n\theta_n^{2(1+\al_{n,\sg})}|z-\frac{x_n}{\epsilon_n}|^2\right)^2}\right)
\end{equation*}
and
\begin{equation*}
I_{n,2}:=\lm_n^{-1}\int\limits_{B^{(n)}_1}\frac{\theta_n^{2(1+\bns)}(1+|z|^2)^{\bns}e^{-\sigma\lm_n R_\O(\epsilon_n z,0)}}{(1+\bar\gamma_n\theta_n^{2(1+\bns)}|z-\frac{x_n}{\epsilon_n}|^{\frac{\lm_n}{4\pi}})^2}\log\left(\frac{(1+|z|^2)^{2\al_{n,\sg}}}
{\left(1+\bar\gamma_n\theta_n^{2(1+\al_{n,\sg})}|z-\frac{x_n}{\epsilon_n}|^{\frac{\lm_n}{4\pi}}\right)^2}\right),
\end{equation*}
where $B^{(n)}_1=B^{(n)}\setminus B_{R}(\frac{x_n}{\epsilon_n})$.

We estimate $I_{n,1}$ as follows,
\begin{align*}
 |I_{n,1}|&\leq C_1\int\limits_{B_{R}(0)}\dfrac{\theta_n^{2(1+\al_{n,\sg})}(3R^2)^{\al_{n,\sg}}}
{\left(1+\bar\gamma_n\theta_n^{2(1+\al_{n,\sg})}|y|^2\right)^2}\left|\log\left(\frac{(3R^2)^{2\al_{n,\sg}}}
{\left(1+\bar\gamma_n\theta_n^{2(1+\al_{n,\sg})}|y|^2\right)^2}\right)\right| \\
&\leq C_2\int\limits_{|x|\leq R\theta_n^{1+\al_{n,\sg}}}\dfrac{\log\left({(3R^2)^{2\al_{n,\sg}}\left(1+\bar\gamma_n|y|^2\right)^2}
\right)}
{\left(1+\bar\gamma_n|y|^2\right)^2}=O(1).
\end{align*}

On the other side, concerning $I_{n,2}$ we can write,
\begin{align*}
 I_{n,2}&=I_{n,2}^{(1)}-I_{n,2}^{(2)},\quad \mbox{where}\\
& I_{n,2}^{(1)}=2\bns\lm_n^{-1}\int\limits_{B^{(n)}_1}\dfrac{\theta_n^{2(1+\al_{n,\sg})}(1+|z|^2)^{\al_{n,\sg}}e^{-\sigma\lm_nR_\O(\epsilon_nz,0)}}
{\left(1+\bar\gamma_n\theta_n^{2(1+\al_{n,\sg})}|z-\frac{x_n}{\epsilon_n}|^{\frac{\lm_n}{4\pi}}\right)^2}\log\left({1+|z|^2}\right),\\
& I_{n,2}^{(2)}=2\lm_n^{-1}\int\limits_{B^{(n)}_1}\dfrac{\theta_n^{2(1+\al_{n,\sg})}(1+|z|^2)^{\al_{n,\sg}}e^{-\sigma\lm_nR_\O(\epsilon_nz,0)}}
{\left(1+\bar\gamma_n\theta_n^{2(1+\al_{n,\sg})}|z-\frac{x_n}{\epsilon_n}|^{\frac{\lm_n}{4\pi}}\right)^2}\log{\left(1+\bar\gamma_n\theta_n^{2(1+\al_{n,\sg})}|z-\tfrac{x_n}{\epsilon_n}|^{\frac{\lm_n}{4\pi}}\right)}.
\end{align*}
Firstly, we notice that $I_{n,2}^{(1)}\geq0$. On the other hand, since in $B^{(n)}_1$ we have,
$$
\frac R4\leq \frac12 |z|\leq |z-\frac{x_n}{\epsilon_n}|\leq 2|z|,
$$
for any $n$ large enough, setting $B^{(n)}_2=B^{(n)}\setminus B_{\frac R2}(0)$ we can estimate $I_{n,2}^{(2)}$ as follows,

\begin{align*}
&0\leq I_{n,2}^{(2)}\leq 2 C\int\limits_{B^{(n)}_1}\dfrac{\theta_n^{2(1+\al_{n,\sg})}|z|^{2\al_{n,\sg}}}
{\left(1+2^{-\frac{\lm_n}{4\pi}}\bar\gamma_n\theta_n^{2(1+\al_{n,\sg})}|z|^{\frac{\lm_n}{4\pi}}\right)^2}\log
{\left(1+2^{\frac{\lm_n}{4\pi}}\bar\gamma_n\theta_n^{2(1+\al_{n,\sg})}|z|^{\frac{\lm_n}{4\pi}}\right)} \\
& \leq  C\int\limits_{\frac{R}{2}\theta_n^{1+\al_{n,\sg}}\leq|y|\leq \frac{r}{\epsilon_n}\theta_n^{1+\al_{n,\sg}}}\dfrac{\theta_n^{-{2\al_{n,\sg}}({1+\al_{n,\sg}})}|y|^{2\al_{n,\sg}}\log\left({(1+5\bar\gamma_n\theta_n^{2(1+\bns)-\frac{\lm_n}{4\pi}(1+\bns)}|y|^{\frac{\lm_n}{4\pi}})}\right)}
{\left(1+\frac15\bar\gamma_n\theta_n^{2(1+\bns)-\frac{\lm_n}{4\pi}(1+\bns)}|y|^{\frac{\lm_n}{4\pi}}\right)^2}\\
& \leq  \tilde C\int\limits_{\frac{R}{2}\theta_n^{1+\al_{n,\sg}}\leq|y|\leq \frac{r}{\epsilon_n}\theta_n^{1+\al_{n,\sg}}}\dfrac{\theta_n^{-{2\al_{n,\sg}}({1+\al_{n,\sg}})}(\log \theta_n)|y|^{2\al_{n,\sg}}\log|y|}
{\theta_n^{4(1+\bns)-\frac{\lm_n}{2\pi}(1+\bns)}|y|^{\frac{\lm_n}{2\pi}}}\\
&=o\left(\theta_n^{-2({1+\al_{n,\sg}})}(\log\theta_n)^2\right),
\end{align*}
where we used once more \eqref{decay}. Therefore we deduce from \eqref{new:n} that,
$$
2\lm_\ii E\geq O(1)+\liminf\limits_{n\to +\ii} \left(\log\left(\tfrac{\epsilon_n^{2(1+\al_{n,\sg})+2\bns-\frac{\lm_n}{4\pi}}}{\delta_n^{2(1+\al_{n,\sg})}}\right)\right)+\liminf\limits_{n\to +\ii} (I_{n,1}+ I_{n,2}^{(1)}-I_{n,2}^{(2)})\geq
$$
\beq\label{enrg:bll0}
O(1)+\liminf\limits_{n\to +\ii} \left(\log\left(\left(\tfrac{\epsilon_n}{\delta_n}\right)^{2(1+\al_{n,\sg})}\epsilon_n^{2\bns-\frac{\lm_n}{4\pi}}\right)\right)\to +\ii,
\eeq
which is \eqref{enrg:n}.
This fact concludes the discussion of {CASE} (I).\\ 

Next, we discuss the following,
\begin{itemize}
 \item {\bf CASE} (II): there exists a subsequence such that $\tfrac{\epsilon_n}{t_n}\leq C$ and $\tfrac{|x_n|}{\delta_n}\to+\infty$.
\end{itemize}
Recall that in this case we have,
\begin{equation}\label{profilew:H1-new}
 \bar w_n(z)=\log\left(\dfrac{\bar \theta_n^{2(1+\al_{n,\sg})}}
{\left(1+\gamma_n\bar \theta_n^{2(1+\al_{n,\sg})}|z-\frac{x_n}{|x_n|}|^{\frac{\lm_n}{4\pi}}\right)^2}\right)+O(1),\quad z\in D^{(n)}=B_{\frac{r}{|x_n|}},
\end{equation}
where
$$
\bar \theta_n^{2(1+\al_{n,\sg})}=e^{\dis \bar w_n(\tfrac{x_n}{|x_{n}|})}=e^{\dis v_n(x_n)}|x_{n}|^{2(1+\al_{n,\sg})}=
\left(\tfrac{|x_{n}|}{\delta_n}\right)^{2(1+\al_{n,\sg})}
$$
and $\gamma_n$ as in the statement of Theorem \ref{profile-new}. Essentially by the same evaluations
worked out above about {CASE} (I), where this time we use \eqref{cnmn} and \eqref{eps:II-A}, we deduce that,
$$
2\lm_\ii E\geq O(1)+\liminf\limits_{n\to +\ii}4\lm_n^{-1}(8\pi+o(1))\log\left(\tfrac{|x_n|^{2\al_{n,\sg}}}{\delta_n^{1+\al_{n,\sg}}}\right)\to +\ii,
$$
which proves \eqref{enrg:n}. We omit the details to avoid repetitions. This fact concludes the discussion of CASE (II).\\

At last we have the following,
\begin{itemize}
 \item {\bf CASE} (III): there exists a subsequence such that $\tfrac{\epsilon_n}{t_n}\leq C$ and $\tfrac{|x_n|}{\delta_n}\leq C$.
\end{itemize}
Recall that in this case we have \eqref{profilevtilde:H1-IIb}, that is,
\begin{equation*}
\widetilde w_n(y)=\widetilde U_n(y) + O(1),\quad y\in D_n=B_{\frac{r}{\dt_n}},
\end{equation*}
where, for any large $R\geq 1$ we have,
\begin{equation*}
\widetilde U_n(y)=\graf{\widetilde w(y)+O(1),\quad |y|\leq R, \\ -\frac{\lm_n}{2\pi}\log(|y|) + O(1),\quad R\leq  |y|\leq r\delta_n^{-1},}
\end{equation*}
where $\widetilde w$ is the unique solution of \eqref{profile:tildew}. In view of \eqref{cnmn}, \eqref{eps:II-B} and \eqref{enrg:n}, we have that,
\begin{align*}
& 2\lm_\ii E\geq \liminf\limits_{n\to +\ii}2\lm_n\left(\mathcal{E}(\rho_{n})-\mathcal{E}_{\sg,n}(\rho_{n})\right)=
\liminf\limits_{n\to +\ii}\left( \lm_n\ino \rho_nG[\rho_n]-2\sg\lm_n\ino \rho_n G_n(x)\right)=\\
& \liminf\limits_{n\to +\ii}\ino \rho_n \left( \lm_n \psi_n -2\sg\lm_n G_n(x)\right)=
\liminf\limits_{n\to +\ii}\ino \rho_n \left( v_n+c_n -2\sg\lm_n G_n(x)\right)=\\
& O(1)+\liminf\limits_{n\to +\ii}\ino \rho_n \left( v_n+c_n +\frac{\lm_n}{2\pi}\sg \log(\epsilon_n^2+|x|^2)\right)\geq\\
& O(1)+\liminf\limits_{n\to +\ii}\int\limits_{B_r(0)} \rho_n \left( v_n+c_n +2\al_{n,\sg}\log(\epsilon_n^2+|x|^2)\right)=\\
& O(1)+\liminf\limits_{n\to +\ii}\lm_n^{-1}\int\limits_{D_n}\widetilde V_n(y)e^{\dis \widetilde w_n(y)}
\left(v_n(\delta_n y)+c_n +2\al_{n,\sg}\log(\epsilon_n^2+\delta_n^2|y|^2)\right)=\\
& O(1)+\liminf\limits_{n\to +\ii}\lm_n^{-1}\int\limits_{D_n} \widetilde V_n e^{\widetilde w_n}\left(\frac{\widetilde \beta}{4\pi(1+\bns)}-1+\frac{1-\bns}{1+\bns}\right)v_n(x_n)
+\liminf\limits_{n\to +\ii} \widetilde I_n,\\
& \; \mbox{where from \eqref{profilewtilde:H1}, for some large enough $R\geq 1$, we have that,}\\
& \widetilde I_n=O(1)+\lm_n^{-1}\int\limits_{D_n\setminus B_R(0)}\widetilde V_n(y)e^{\dis \widetilde w_n(y)}
\log\left(\frac{(\frac{\epsilon_n^2}{\delta_n^2}+|y|^2)^{\al_{n,\sg}}}
{\left(1+|y|\right)^{\frac{\widetilde{\beta}}{2\pi}}}\right).
\end{align*}

However, in view of $\bis< 1$, from \eqref{masstilde0} the last integral is bounded and in particular, for some $c>0$ we have that,
$$
\frac{\widetilde \beta}{4\pi(1+\bns)}-1+\frac{1-\bns}{1+\bns}\geq c\varepsilon_0>0,
$$
whence we deduce that
$$
2\lm_\ii E\geq O(1)+\lm_n^{-1}(\widetilde \beta +o(1))c\varepsilon_0 v_n(x_n)\to +\ii,
$$
which proves \eqref{enrg:n} in this case as well. This fact concludes the discussion of CASE (III).
\finedim

\bigskip
\bigskip

\section{Appendix}

\subsection{Known results about problem \eqref{profile:tildew}}\hfill\\
We list some well known facts about the planar problem \eqref{profile:tildew}. For $\epsilon_0=0$, the set of solutions has been completely classified in \cite{pt} for every $\bis>-1$ and $\widetilde\beta=8\pi(1+\bis)$.
Moreover, if $\bis\notin\N$, the solution is necessarily unique and radially symmetric.\\
On the other hand, for $\epsilon_0>0$ the situation is much more involved (see \cite{Lin1}):
\begin{itemize}
\item If $\bis\in (-1,1]$, problem \eqref{profile:tildew} is solvable if and only if
$$\widetilde\beta\in (8\pi \min\{1, 1+\bis\}, 8\pi \max\{1, 1+\bis\})$$
and the solution is unique, radially symmetric and non degenerate.
\item If $\bis>1$, there exists a unique, radial and non degenerate solution, whenever
$$\widetilde\beta\in(8\pi\bis,8\pi(1+\bis)).$$
\end{itemize}
Nevertheless, in a more general context it has been proved in \cite{GM1}  that, if $\widetilde \beta\in(0,16\pi)$, then any solution of \eqref{profile:tildew} is  radially symmetric. In general the problem is still open, and we refer to \cite{det} for further details and to the refined estimates in \cite{llty} (Proposition 2.1), where the authors obtained some results for \eqref{profile:tildew} in the case $\bis>1$.\\
We will need in particular the following facts  (\cite{cl2}, \cite{det}).
\begin{lemma}\label{massstrange}
 Let $\phi$ be a solution of the planar problem
 \begin{equation*}
 \begin{cases}
 -\D \phi(x)= (t_0^2+|x|^2)^{\alpha}e^{\dis \phi}\quad\text{in}\quad \R^2, \\
 \underset{\R^2}\int(t_0^2+|x|^2)^{\alpha}e^{\dis \phi}<+\infty,
 \end{cases}
\end{equation*}
for $t_0\geq0$ and $\alpha>-1$ and set,
\begin{equation*}
 \beta=\tfrac{1}{2\pi}\underset{\R^2}\int(t_0^2+|x|^2)^{\alpha}e^{\dis \phi}.
\end{equation*}
Then,
\begin{itemize}
 \item $(t_0^2+|x|^2)^{\alpha}e^{\dis \phi}$ is bounded in $\R^2$,
 \item There exists a constant $C$ such that
 \begin{equation}\label{asymptotic}
  -\beta\ln(|x|+1)-C\leq \phi(x)\leq-\beta\ln(|x|+1)+C,
 \end{equation}
 \item The following identity holds,
 \begin{equation}\label{relation}
  2\alpha \underset{\R^2}\int|x|^2(t_0^2+|x|^2)^{\alpha-1}e^{\dis \phi}=\pi\beta(\beta-4).
 \end{equation}
 \item If $-1<\alpha< 0$, then
 \begin{equation}
 \label{massnegative}
  8\pi(1+\alpha)\leq \underset{\R^2}\int(t_0^2+|x|^2)^{\alpha}e^{\dis \phi}<8\pi,
 \end{equation}
 where the l.h.s equality holds if and only if $t_0=0$.\\
 If $\alpha\geq 0$, then
 \begin{equation}
 \label{masspositive}
\max\{8\pi,4\pi(1+\alpha)\}\leq \underset{\R^2}\int(t_0^2+|x|^2)^{\alpha}e^{\dis \phi}\leq  8\pi(1+\alpha),
 \end{equation}
 where the l.h.s. equality holds if and only if $\alpha=0$ while the right hand side equality holds if and only if $t_0=0$.
\end{itemize}
\end{lemma}

\bigskip
\bigskip

\subsection{Proof of the estimate \eqref{profilew:H}}\label{app:est-1}
\subsubsection{Initial estimate for \eqref{profilew:H}}\hfill\\

Let $r'\in(0,r)$. We use Green representation formula for $v_n$ and see that, for $x\in B_{r'}(0)$, we have that
\[
 v_n(x)=\tfrac{1}{2\pi}\int_{B_r}\log(\tfrac{1}{|x-y|})H_n e^{v_n}\,dy+m_n+\Psi_n(x),
\]
where $m_n=\inf_{\partial B_{r}}\,v_n$ and $\Psi_n\in C^{1}(\bar B_{r'})$. Then, for $|x|\leq\tfrac{r'}{\epsilon_n}$,
\[
 w_n(x)=\tfrac{1}{2\pi}\int_{B_{\frac{r}{\epsilon_n}}}\log(\tfrac{1}{|x-z|})V_n e^{w_n}\,dz-\left(\tfrac{M_n}{2\pi}-2(1+\bns)\right)\log\epsilon_n+m_n+\Psi_n(\epsilon_nx).
\]

Now, let $z_0\in\p B_{R_0}$, for a fixed $R_0>1$, then we have,
\[
 w_n(z_0)=\tfrac{1}{2\pi}\int_{B_{\frac{r}{\epsilon_n}}}\log(\tfrac{1}{|z-z_0|})V_n e^{w_n}\,dz-\left(\tfrac{M_n}{2\pi}-2(1+\bns)\right)\log\epsilon_n+m_n+\Psi_n(\epsilon_n z_0)
\]
and, by \eqref{profilew}, we also have,
\[
 w_n(z_0)=-2(1+\bns)\log\theta_n+O(1).
\]
Hence, it easy to see that
\begin{equation}
 \label{Greenrepresentation:wn}
 w_n(x)=\tfrac{1}{2\pi}\int_{B_{\frac{r}{\epsilon_n}}}\log(\tfrac{|z-z_0|}{|x-z|})V_n e^{w_n}\,dz-2(1+\bns)\log\theta_n+O(1).
\end{equation}
Now, let us fix $\epsilon\in (0,1)$ and, by \eqref{decay}, $R$ large enough such that $4R_0\leq R\leq \frac{r}{4\epsilon_n}$ and
\[
\int_{B_{\frac{R}{2}}(z_0)}V_ne^{w_n}\geq 8\pi-\pi\epsilon.
\]
For every $2R\leq|x|\leq \frac{r}{2\epsilon_n}$ and $|z-z_0|<\tfrac{R}{2}$, with $|z_0|=R_0$, we have,
\[
\tfrac{|z-z_0|}{|x-z|}\leq\tfrac{R}{|x|}\leq1.
\]
Then, for every $2R\leq|x|\leq \frac{r}{2\epsilon_n}$,
\begin{align*}
\tfrac{1}{2\pi}\int_{B_{\frac{R}{2}}(z_0)}\log(\tfrac{|z-z_0|}{|x-z|})V_n e^{w_n}\,dz&\leq \tfrac{1}{2\pi}\log(\tfrac{2R}{|x|})\int_{B_{\frac{R}{2}}(z_0)}V_n e^{w_n}\,dz\\
&\leq (4-\tfrac{\epsilon}{2})\log(\tfrac{2R}{|x|}).
\end{align*}
Moreover, for every $2R\leq|x|\leq \frac{r}{2\epsilon_n}$,
\begin{align}\nonumber
\tfrac{1}{2\pi}\int_{B_{\frac{|x|}{4}}(z_0)}&\log(\tfrac{|z-z_0|}{|x-z|})V_n e^{w_n}\,dz\\\nonumber
&\leq \tfrac{1}{2\pi}\int_{B_{\frac{R}{2}}(z_0)}\log(\tfrac{|z-z_0|}{|x-z|})V_n e^{w_n}\,dz+ \underbrace{\tfrac{1}{2\pi}\int_{{\frac{R}{2}\leq|z-z_0|\leq\frac{|x|}{4}}}\log(\tfrac{|z-z_0|}{|x-z|})V_n e^{w_n}\,dz}_{\leq0}\\
&\leq(4-\tfrac{\epsilon}{2})\log(\tfrac{2R}{|x|})\leq-(4-\tfrac{\epsilon}{2})\log(|x|)+ C_{1,R},\label{firstpiece1}
\end{align}
for a suitable constant $C_{1,R}$. Now, for $2R\leq|x|\leq \frac{r}{2\epsilon_n}$, we have that
\begin{align*}
 &\int_{B_{\frac{|x|}{4}}(x)}\log(\tfrac{|z-z_0|}{|x-z|})V_n e^{w_n}=\\
&= \underbrace{\int\limits_{B_{\frac{|x|}{4}}(x)}\log(|z-z_0|)V_n e^{w_n}}_{(A)}+  \underbrace{\int\limits_{|z-x|\leq\frac{1}{|x|^{1+\alpha_\infty}}}\log(\tfrac{1}{|x-z|})V_n e^{w_n}}_{(B)}+\underbrace{\int\limits_{\frac{1}{|x|^{1+\alpha_\infty}}\leq|z-x|\leq\frac{|x|}{4}}\log(\tfrac{1}{|x-z|})V_n e^{w_n}}_{(C)},
\end{align*}
where $\alpha_\infty=\al_{\ii,\sg}$.
Then,
\begin{align*}
 (A)+(C)&\leq (2+\alpha_\infty)\log(|x|)\int_{|z-x|\leq\frac{|x|}{4}}V_ne^{w_n}+O(1) \\
 &\leq (2+\alpha_\infty)\log(|x|)\int_{1\leq|z|\leq\frac{r}{\epsilon_n}}V_ne^{w_n}+O(1)\\
 &\leq  \tfrac{\epsilon}{4}\log|x|+O(1),
\end{align*}
where we have used the fact that, for $n$ large enough,
\[
 \int_{1\leq|z|\leq\frac{r}{\epsilon_n}}V_ne^{w_n}\leq \tfrac{\epsilon}{4(2+\alpha_\infty)}.
\]
On the other hand, by using \eqref{bdd:w}, we have that, putting $\alpha_n=\al_{n,\sg}$,
\begin{align*}
 (B)&\leq C_1 \int\limits_{\{|z-x|\leq\frac{1}{|x|^{1+\alpha_\infty}}\}}\log(\tfrac{1}{|x-z|})(1+|z|^2)^{\alpha_n}\\
 &\leq C_2 \int\limits_{\{|z-x|\leq\frac{1}{|x|^{1+\alpha_\infty}}\}}\log(\tfrac{1}{|x-z|})|z|^{2\alpha_n}\\
 &\leq C_3 (\tfrac{1}{|x|^{1+\alpha_\infty}}+|x|)^{2\alpha_n} \int\limits_{\{|z-x|\leq\frac{1}{|x|^{1+\alpha_\infty}}\}}\log(\tfrac{1}{|x-z|})
 \end{align*}
 \begin{align*}
 &\leq C_4 (\tfrac{1}{|x|^{1+\alpha_\infty}}+|x|)^{2\alpha_n} \tfrac{\log|x|+1}{|x|^{2(1+\alpha_\infty)}}\\
 &\leq C_5 (\tfrac{1}{R})^{2(1+\alpha_\infty)-2\alpha_n}\log|x|+O(1)
\end{align*}
and, by choosing $n$ and $R$ large enough, we have that
\begin{align*}
 (B)&\leq \tfrac{\epsilon}{4}\log|x|+O(1),
\end{align*}
which implies that
\begin{equation}
 \label{secondpiece1}
 \int_{B_{\frac{|x|}{4}}(x)}\log(\tfrac{|z-z_0|}{|x-z|})V_n e^{w_n}\leq \tfrac{\epsilon}{2}\log|x|+O(1).
\end{equation}
At last, if $z\in B^{(n)}\backslash (B_{\frac{|x|}{4}}(z_0)\cup B_{\frac{|x|}{4}}(x) )$, we have that
\begin{equation}
\label{thirdpiece1}
 \log(\tfrac{|z-z_0|}{|x-z|})\leq C.
\end{equation}
Then, by using \eqref{Greenrepresentation:wn},\eqref{firstpiece1},\eqref{secondpiece1} and \eqref{thirdpiece1}, for every $2R\leq|x|\leq \frac{r}{2\epsilon_n}$, we deduce that
\begin{align}\nonumber
 w_n(x)&\leq -2(1+\bns)\log\theta_n-(4-\tfrac{\epsilon}{2})\log|x|+ C_{1,\epsilon}+\tfrac{\epsilon}{2}\log|x|+(8\pi+o(1))C\\
 &\leq -2(1+\bns)\log\theta_n-(4-\epsilon)\log|x|+ C_{2,\epsilon} \nonumber
\end{align}
and consequently,
\begin{align}
 \label{initialestimatewn2}
 V_ne^{\dis w_n}\leq C\theta_n^{-2(1+\bns)}|x|^{2\bns-4+\epsilon},
\end{align}
for every $2R\leq|x|\leq \tfrac{r}{2\epsilon_n}$. We notice that, since $\bns<1$, we can choose $\epsilon>0$ such that $2\bns-4+\epsilon<-2$, for $n$ large enough.

\subsubsection{Integral representation for $w_n$ and refined estimate}\hfill\\
Our next goal is to prove that
\begin{equation}
 \label{refinedestimate:wn}
 w_n(x)=-\tfrac{M_n}{2\pi}\log|x|-2(1+\bns)\log\theta_n+O(1), \qquad\,\,6R\leq|x|\leq\tfrac{r'}{\epsilon_n},
\end{equation}
where we remind that
\[
 M_n=\int_{B_{\frac{r}{2\epsilon_n}}}V_ne^{w_n}.
\]

By inspection on \eqref{Greenrepresentation:wn} (where without loss of generality we replace $r$ with $\tfrac{r}{2}$), we see that it is enough to show that
\begin{equation}
\label{integrallogarithmicestimate}
 \tfrac{1}{2\pi}\int_{B_{\frac{r}{2\epsilon_n}}}\log\left(\tfrac{|x||z-z_0|}{|x-z|}\right) V_ne^{w_n}=O(1).
\end{equation}
Firstly, we notice that
\begin{equation}
 \label{logarithmestimate}
 \left|\log\left(\tfrac{|x||z-z_0|}{|x-z|}\right)\right|\leq 2|\log|z-z_0||+\left(\log\left(\tfrac{1}{|x-z|}\right)\right)^++\log (2R_0),
\end{equation}
for every $6R\leq|x|\leq\tfrac{r'}{\epsilon_n}$ and $|z|\leq\tfrac{r}{2\epsilon_n}$, where $f^+=\max\{0,f\}$. This is a slight variation of (5.6.34) in \cite{tar-sd}, we prove it discussing different cases.\\
If $z\in B_{\frac{|x|}{2}}(0)$, then $|x-z|\geq\tfrac{|x|}{2}>1$ and $\left(\log\left(\tfrac{1}{|x-z|}\right)\right)^+=0$ and
\[
 \left|\log\left(\tfrac{|x||z-z_0|}{|x-z|}\right)\right|=\left|\log|z-z_0|-\log(\big|\tfrac{x}{|x|}-\tfrac{z}{|x|}\big|)\right|\leq |\log|z-z_0||+\log2.
\]
If $z\in B_{\frac{|x|}{2}}(x)$, then $3R\leq\tfrac{|x|}{2}\leq|z|\leq\tfrac{3}{2}|x|$ and
\[
 \log\left(\tfrac{|x||z-z_0|}{|x-z|}\right)\geq\log(2|z-z_0|)\geq\log(8R_0)>0,
\]
and then we have,
\begin{align*}
 \left|\log\left(\tfrac{|x||z-z_0|}{|x-z|}\right)\right|=\log\left(\tfrac{|x||z-z_0|}{|x-z|}\right)&\leq \log|z-z_0|+\left(\log\left(\tfrac{1}{|x-z|}\right)\right)^++\log |x|\\
 &\leq \log|z-z_0|+\left(\log\left(\tfrac{1}{|x-z|}\right)\right)^++\log(2|z|)\\
 &\leq 2|\log|z-z_0||+\left(\log\left(\tfrac{1}{|x-z|}\right)\right)^++\log 2R_0.
\end{align*}
At last, if $z\in B_{\frac{r'}{\epsilon_n}}\backslash\left(B_{\frac{|x|}{2}}(0)\cup B_{\frac{|x|}{2}}(x)\right)$, then $\tfrac{1}{|x-z|}<\tfrac{2}{|x|}<\tfrac{1}{3R}<1$, therefore we have that, $\left(\log\left(\tfrac{1}{|x-z|}\right)\right)^+=0$. Moreover we have, $|x|+|z|\leq 3|z|$, and then,
\begin{align*}
 \log\left(\tfrac{|x||z-z_0|}{|x-z|}\right)&\geq \log\left(\tfrac{|x||z-z_0|}{|x|+|z|}\right)\geq \log\left(\tfrac{|x||z-z_0|}{3|z|}\right)\\
 &= \log\left(\tfrac{|x|}{3}\left|\tfrac{z}{|z|}-\tfrac{z_0}{|z|}\right|\right)\geq\log(\tfrac{1}{4}|x|)\geq\log(3R_0)>0.
\end{align*}
Hence,
\[
 \left|\log\left(\tfrac{|x||z-z_0|}{|x-z|}\right)\right|=\log\left(\tfrac{|x||z-z_0|}{|x-z|}\right)\leq|\log|z-z_0||+\log2.
\]
This concludes the proof of \eqref{logarithmestimate}.\\
At this point, we are able to obtain the estimate \eqref{integrallogarithmicestimate}. Indeed, by using \eqref{initialestimatewn2} and \eqref{logarithmestimate} we have that
\begin{align*}
    & \int\limits_{B_{\frac{r}{2\epsilon_n}}}\left|\log\left(\tfrac{|x||z-z_0|}{|x-z|}\right) V_ne^{w_n}\right|\leq\\ &C_1\int\limits_{B_\frac{r}{2\epsilon_n}}|\log|z-z_0||V_n e^{w_n}+ C_2\int\limits_{|z-x|\leq 1}\log\left(\tfrac{1}{|x-z|}\right)V_ne^{w_n}+C_3\leq C,
\end{align*}
which shows the validity of \eqref{refinedestimate:wn}. By using this refined estimate, we deduce that
\begin{equation*}
 \label{decaymass:Vnwn}
 V_ne^{w_n}=O\left(\frac{\theta_n^{-2(1+\bns)}}{|x|^{\gamma_n+1}}\right),\qquad |x|\in[6R,\tfrac{r'}{\epsilon_n}]
\end{equation*}
where $\gamma_n:=\frac{M_n}{2\pi}-2\bns-1$, and we are able to obtain the following pointwise estimate,
\begin{equation*}
 w_n(x)=\log\left(\frac{\theta_n^{2(1+\bns)}}{(1+\bar\gamma_n\theta_n^{2(1+\bns)}|x-\tfrac{x_n}{\epsilon_n}|^{\frac{M_n}{2\pi}})^2}\right)+O(1),\qquad |x|\in[6R,\tfrac{r'}{\epsilon_n}]
\end{equation*}
with $\bar\gamma_n=\tfrac18V_n(\tfrac{x_n}{\epsilon_n})$. Scaling back to $v_n$, this actually implies that
\begin{equation*}
 v_n(y)=\log\left(\frac{\delta_n^{-2(1+\bns)}}{(1+\bar\gamma_n\theta_n^{2(1+\bns)}\epsilon_n^{^{-\frac{M_n}{2\pi}}}|y-x_n|^{\frac{M_n}{2\pi}})^2}\right)+O(1),\qquad |x|\in[6\epsilon_nR,r']
\end{equation*}
which is \eqref{profilew:H}, as claimed. \finedim

\section*{Acknowledgements}

We are indebted to D. Benedetto, E. Caglioti and M. Nolasco for insightful discussions on uniform estimates for entropy maximizers in their recent paper \cite{BCN24}. A careful adaptation of their argument was crucial for the proof of Lemma \ref{lem:connect}.

\smallskip
Daniele Bartolucci and Paolo Cosentino were partially supported by the MIUR Excellence Department Project MatMod@TOV awarded to the Department of Mathematics, University of Rome ``Tor Vergata'', CUP E83C23000330006, by PRIN project 2022 2022AKNSE4, ERC PE1\_11, ``{\em Variational and Analytical aspects of Geometric PDEs}'', by the E.P.G.P. Project sponsored by the University of Rome ``Tor Vergata'', E83C25000550005, and by INdAM-GNAMPA Project, CUP E53C25002010001. They are members of the INDAM Research Group ``Gruppo Nazionale per l’Analisi Matematica, la Probabilit\`a e le loro Applicazioni".\\
L. Wu was partially supported by the National Natural Science Foundation of China (12201030).

\end{document}